\documentclass[reqno]{amsart}
\usepackage[left=1in,right=1in,top=.95in,bottom=.95in]{geometry}
\usepackage{tikz}
\usetikzlibrary{shapes,snakes,calc,arrows}
\usepackage{amsthm}
\usepackage{amscd}
\usepackage{amsfonts}
\usepackage{amsmath}
\usepackage{amssymb}
\usepackage{mathrsfs}
\usepackage{multirow}
\usepackage{verbatim}
\usepackage{url}
\usepackage{graphicx}
\usepackage{yfonts}

\begin{document}


\newcommand{\supp}{\text{supp}}
\newcommand{\Aut}{\text{Aut}}
\newcommand{\Gal}{\text{Gal}}
\newcommand{\Inn}{\text{Inn}}
\newcommand{\Irr}{\text{Irr}}
\newcommand{\Ker}{\text{Ker}}
\newcommand{\N}{\mathbb{N}}
\newcommand{\Z}{\mathbb{Z}}
\newcommand{\Q}{\mathbb{Q}}
\newcommand{\R}{\mathbb{R}}
\newcommand{\C}{\mathbb{C}}
\renewcommand{\H}{\mathcal{H}}
\newcommand{\B}{\mathcal{B}}
\newcommand{\A}{\mathcal{A}}
\newcommand{\K}{\mathcal{K}}
\newcommand{\M}{\mathcal{M}}

\newcommand{\J}{\mathscr{J}}
\newcommand{\D}{\mathscr{D}}

\newcommand{\ul}[1]{\underline{#1}}

\newcommand{\I}{\text{I}}
\newcommand{\II}{\text{II}}
\newcommand{\III}{\text{III}}

\newcommand{\<}{\left\langle}
\renewcommand{\>}{\right\rangle}
\renewcommand{\Re}[1]{\text{Re}\ #1}
\renewcommand{\Im}[1]{\text{Im}\ #1}
\newcommand{\dom}[1]{\text{dom}\,#1}
\renewcommand{\i}{\text{i}}
\renewcommand{\mod}[1]{(\operatorname{mod}#1)}
\newcommand{\mb}[1]{\mathbb{#1}}
\newcommand{\mc}[1]{\mathcal{#1}}
\newcommand{\mf}[1]{\mathfrak{#1}}
\newcommand{\im}{\operatorname{im}}


\newtheorem{thm}{Theorem}[section]
\newtheorem{prop}[thm]{Proposition}
\newtheorem{lem}[thm]{Lemma}
\newtheorem{cor}[thm]{Corollary}
\newtheorem{innercthm}{Theorem}
\newenvironment{cthm}[1]
  {\renewcommand\theinnercthm{#1}\innercthm}
  {\endinnercthm}
\newtheorem{innerclem}{Lemma}
\newenvironment{clem}[1]
  {\renewcommand\theinnerclem{#1}\innerclem}
  {\endinnerclem}

\theoremstyle{definition}
\newtheorem{defi}[thm]{Definition}
\newtheorem{ex}[thm]{Example}
\newtheorem*{exs}{Examples}
\newtheorem{rem}[thm]{Remark}
\newtheorem{innercdefi}{Definition}
\newenvironment{cdefi}[1]
  {\renewcommand\theinnercdefi{#1}\innercdefi}
  {\endinnercdefi}


\title{Free monotone transport without a trace}

\author{Brent Nelson}
\address{UCLA Mathematics Department}
\email{bnelson6@math.ucla.edu}
\thanks{Research supported by NSF grants DMS-1161411 and DMS-0838680}

\maketitle

\begin{abstract}
We adapt the free monotone transport results of Guionnet and Shlyakhtenko to the type III case. As a direct application, we obtain that the $q$-deformed Araki-Woods algebras are isomorphic (for sufficiently small $|q|$).
\end{abstract}


\section{Introduction}

Classically, transport from a probability space $(X,\mu)$ to a probability space $(Z,\nu)$ is a measurable map $T\colon X\rightarrow Z$ such that $T_*(\mu)=\nu$. In particular, it implies that there is a measure-preserving embedding of $L^\infty (Z,\nu)$ into $L^\infty(X,\mu)$ via $f\mapsto f\circ T$. These generalize to non-commutative probability theory in that the existence of transport from the law $\varphi_X$ of an $N$-tuple of non-commutative random variables $X=(X_1,\ldots, X_N)$ to the law $\varphi_Z$ of $Z=(Z_1,\ldots, Z_N)$ implies that there is a state-preserving embedding of $W^*(Z_1,\ldots, Z_N)$ into $W^*(X_1,\ldots, X_N)$. Provided there is sufficient control over the transport map, the embedding can be made into an isomorphism.\par
In \cite{SG11}, Guionnet and Shlyakhtenko produced a non-commutative analogue of Brenier's monotone transport theorem by solving a free analogue of the Monge-Amp\'{e}re equation. This provided criterion for when an $N$-tuple of non-commutative random variables generate the free group factor. Their result took place in the context of a tracial von Neumann algebra $(M,\tau)$ where $X_1,\ldots, X_N \in M_{s.a.}$ are free semi-circular variables. We adapt the result to the context of a von Neumann algebra $M$ with a (not necessarily tracial) state $\varphi$ on $M$. The random variables $X_1,\ldots, X_N$ are no longer assumed to be free; instead their joint law is assumed to be a free quasi-free state (\emph{cf.} \cite{S97}). This produces criterion for when an $N$-tuple of non-commutative random variables generate the free Araki-Woods algebra $\Gamma(\H_\R,U_t)\rq{}\rq{}$ (\emph{cf.} \cite{S97}).\par
One of the key assumptions that was needed in \cite{SG11} to produce transport was that the trace satisfied the so-called Schwinger-Dyson equation with a potential $V$ (that is close in a precise-sense to the Gaussian potential $\sum X_j^2$) for non-commutative polynomials in the transport variables. In the commutative case, this amounts to saying that if $\eta$ is the semicircle law ($d\eta(t)=\chi_{[-2,2]}(t) \frac{1}{2\pi}\sqrt{4-t^2} dt$) and $V(t)=\frac{1}{2} t^2 + W(t)$ (with $W$ analytic on a disk of radius $R$ and small $\|\cdot\|_\infty$-norm) then
	\begin{align*}
		\int_\R V'(t) f(t)\ d\eta(t) = \int_\R \int_\R \frac{f(s) - f(t)}{s-t} d\eta(s) d\eta(t)
	\end{align*}
for all $f$ which are analytic on the disk of radius $R$. In general, a measure satisfying this equation is called a Gibbs state with potential $V$.\par
In the free case, the measure $\eta$ is replaced with a state $\varphi$ on $M$ which we call the free Gibbs state with potential $V$ ($\varphi=\varphi_V$), and $W$ is a non-commutative power series in the $X_1,\ldots, X_N$ with radius of convergence at least $R$. For quadratic potentials of the form
	\begin{align*}
		V_0=\frac{1}{2}\sum_{j,k=1}^N \left[\frac{1+A}{2}\right]_{jk} X_kX_j,\qquad A\in M_N(\C)
	\end{align*}
the corresponding free Gibbs state is the free quasi-free state induced by the scalar matrix $A$ (\emph{cf.} \cite{S97}). Provided $W$ is small enough with respect to a particular Banach norm, then the free Gibbs state with potential $V=V_0+W$ is unique in the sense that if the laws of two $N$-tuples $Y=(Y_1,\ldots, Y_N)$ and $Z=(Z_1,\ldots, Z_N)$ both solve the Schwinger-Dyson equation, then in fact $W^*(Y_1,\ldots, Y_N)\cong W^*(Z_1,\ldots, Z_N)$ (\emph{cf.} Theorem 2.1 in \cite{GMS06}). In this paper, we produce transport from $\varphi_{V_0}$ to $\varphi_V$.\par
In the tracial case, starting with the non-commutative Schwinger-Dyson equation, Guionnet and Shlyakhtenko produced an equivalent version which is amenable to a fixed point argument. We prove the uniqueness of the free Gibbs state in our non-tracial setting, which allows us to proceed along very similar lines to Guionnet and Shlyakhtenko. However, we are forced to consider slightly more general potentials as well as non-tracial states. Consequently, the modular automorphism group plays a significant role.\par
In \cite{H}, Hiai developed a generalization of Shlyakhtenko\rq{}s algebras $\Gamma(\H_\R, U_t)$ from \cite{S97}, called $q$-deformed Araki-Woods algebras. Letting $A$ be the generator of the one-parameter family of unitary operators $\{U_t\}_{t\in\R}$, Hiai was able to show the von Neumann algebras $\Gamma_q(\H_\R,U_t)\rq{}\rq{}$ are factors and produced a type classification, but only in the case that $A$ has infinitely many mutually orthogonal eigenvectors. In particular, when the Hilbert space $\H_\R$ is finite dimensional the questions of factoriality and type classification remained unanswered. An application of our result in Section \ref{application} yields $\Gamma_q(\H_\R,U_t)\rq{}\rq{}\cong \Gamma(\H_\R,U_t)\rq{}\rq{}$ for small $|q|$, and hence we are able to settle these questions using Theorem 6.1 in \cite{S97}.\par
We begin the paper by recalling the free Araki-Woods factors and $q$-deformed Araki-Woods algebras and then considering several derivations defined on these von Neumann algebras.  An introduction to some Banach norms and additional differential operators follows. The definition of free transport is given, including the conditions for it to be monotone, and the Schwinger-Dyson equations and free Gibbs states are defined. As in \cite{SG11}, several equivalent forms of the Schwinger-Dyson equation for a potential $V$ are produced. Using estimates relevant to these equivalent forms of the Schwinger-Dyson equation, the existence of a solution is produced through a fixed point argument. The implications of this solution are discussed as well as how to improve it so that an isomorphism exists. We conclude the paper with an application of the isomorphism result to $q$-deformed Araki-Woods algebras.

\subsection*{Acknowledgments}

I would like to thank my advisor, Prof. D. Shlyakhtenko for the initial idea of the paper, many helpful suggestions and discussions, and his general guidance.


\section{Preliminaries}\label{prelim}

\subsection{The free Araki-Woods factor and $q$-deformed Araki-Woods algebras}\label{free_Araki-Woods}

Let $\H_\R$ be a real Hilbert space and $U_t$ a strongly continuous one-parameter group of orthogonal transformations on $\H_\R$. Letting $\H_\C:=\H_\R +\i \H_\R$ be the complexified Hilbert space, the $U_t$ can be extended to a one-parameter unitary group (still denoted as $U_t$). Let $A$ be the generator of the $U_t$ (i.e. $U_t=A^{it}$ and $A$ is a potentially unbounded positive operator). Let $\<\cdot,\cdot\>$ be the inner product on $\H_\C$ which is complex-linear in the second coordinate (as all other inner products will be in this section). Define an inner product $\<\cdot,\cdot\>_U$ on $\H_\C$ by
	\begin{align*}
		\<x,y\>_U  = \< \frac{2}{1+A^{-1}}x, y\>,\qquad x,y\in\H_\C.
	\end{align*}
Let $\H$ be the complex Hilbert space obtained by completing $\H_\C$ with respect to $\<\cdot,\cdot\>_U$. Note that if we start with the trivial one-parameter group $U_t=1$ for all $t$ then $A=1$, $\<\cdot,\cdot\>_U=\<\cdot,\cdot\>$ and $\H=\H_\C$. In this case we will write $\<\cdot,\cdot\>_1$ for $\<\cdot,\cdot\>_U$.\par
For $-1<q<1$, the $q$-Fock space $\mathcal{F}_q(\H)$ is the completion of $\mathcal{F}^{\text{finite}}(\H):=\bigoplus_{n=0}^\infty \H^{\otimes n}$, where $\H^{\otimes 0}=\C\Omega$ with vacuum vector $\Omega$, with respect to the sesquilinear form $\<\cdot,\cdot\>_{U,q}$ given by
	\begin{align*}
		\<f_1\otimes\cdots\otimes f_n, g_1\otimes\cdot\otimes g_m\>_{U,q} = \delta_{n=m} \sum_{\pi\in S_n} q^{i(\pi)} \<f_1,g_{\pi(1)}\>_U\cdots \<f_n,g_{\pi(n)}\>_U,
	\end{align*}
where $i(\pi)$ denotes the number of inversions of the permutation $\pi\in S_n$. We may at times denote $\mathcal{F}_q(\H_\R,U_t)=\mathcal{F}_q(\H)$ to emphasize $\{U_t\}$.

For any $h\in \H$ we can define the left $q$-creation operator $l(h)\in \B(\mc{F}_q(\H))$ by
	\begin{align*}
		&l_q(h)\Omega=h;\\
		&l_q(h)(f_1\otimes\cdots\otimes f_n)=h\otimes f_1\otimes\cdots \otimes f_n,
	\end{align*}
then its adjoint is the left $q$-annihilation operator:
	\begin{align*}
		&l_q^*(h)\Omega=0;\\
		&l_q^*(h)(f_1\otimes\cdots\otimes f_n)= \sum_{i=1}^n q^{i-1} \<h,f_i\>_U f_1\otimes \cdots \otimes f_{i-1}\otimes f_{i+1}\otimes\cdots \otimes f_n.
	\end{align*}
Also define
	\begin{align*}
		s_q(h):=l_q(h)+l_q^*(h).
	\end{align*}\par
We let $\Gamma_q(\H_\R,U_t)$ be the $C^*$-algebra generated by $\{s_q(h)\colon h\in\H_\R\}$. The corresponding von Neumann algebra $M_q:=\Gamma_q(\H_\R,U_t)''\subset \B(\mc{F}_q(\H))$ is called a \emph{$q$-deformed Araki-Woods algebra}, after \cite{H}, except when $q=0$ where $M_0=\Gamma_0(\H_\R,U_t)''$ is called a \emph{free Araki-Woods factor}, after \cite{S97}.\par
It was shown in \cite{H} that $\Omega$ is a cyclic and separating vector for $M_q$ and consequently the vacuum state $\varphi_q(\cdot)=\<\Omega,\cdot\  \Omega\>_{U,q}$ is faithful. For $q\neq 0$, $\varphi_q$ is called the \emph{$q$-quasi-free state}, or the \emph{$q$-quasi-free state associated to $A$}. For $q=0$, $\varphi_0$ is called the \emph{free quasi-free state}, or the \emph{free quasi-free state associated to $A$}.\par

\begin{rem}
For $f_1,\ldots, f_n\in \H_\R$, computing $\varphi_q(s_q(f_1)\cdots s_q(f_n))$ is best done diagrammatically through non-crossing (when $q=0$) and crossing (when $q\neq 0$) pairing diagrams. When $q=0$, visualize a rectangle with the vectors $f_1,\ldots, f_n$ arranged in order along the top:
	\begin{equation*}
		\begin{tikzpicture}
			\draw [thick] (0,0) -- (0,0.6) -- (2.2,0.6) -- (2.2,0) -- (0,0);
			\node at (1.1,0.3) {$f_1\ f_2\ \cdots\ f_n$};
		\end{tikzpicture}.
	\end{equation*}
$\varphi(s(f_1)\cdots s(f_n))$ counts all the ways to pair the vectors to each other via chords above the rectangle so that no two chords intersect and if a vector $f_i$ is connected to a vector $f_j$ (with $f_i$ on the left) then that diagram is weighted by a factor of $\<f_i,f_j\>_U$. For example the following diagram has the denoted weight:
	\begin{equation*}
		\begin{tikzpicture}[baseline]
			\draw [thick] (0,0.3) -- (0,-0.3) -- (3.6,-0.3) -- (3.6,0.3) -- (0,0.3);
			\node at (.3,0) {$f_1$};
			\node at (.9,0) {$f_2$};
			\node at (1.5,0) {$f_3$};
			\node at (2.1,0) {$f_4$};
			\node at (2.7,0) {$f_5$};
			\node at (3.3,0) {$f_6$};
			\draw[thick] (2.1,0.3) arc (0:180:0.9 and 0.6); 
			\draw[thick] (1.5,0.3) arc (0:180:0.3 and 0.2);
			\draw[thick] (3.3,0.3) arc (0:180:0.3 and 0.2);
		\end{tikzpicture}
		=\<f_1,f_4\>_U\<f_2,f_3\>_U\<f_5,f_6\>_U.
	\end{equation*}
Thus
	\begin{align*}
		\varphi(s(f_1)s(f_2)s(f_3)s(f_4)) &= 
			\begin{tikzpicture}[baseline]
			\draw [thick] (0,0.3) -- (0,-0.3) -- (2.4,-0.3) -- (2.4,0.3) -- (0,0.3);
			\node at (.3,0) {$f_1$};
			\node at (.9,0) {$f_2$};
			\node at (1.5,0) {$f_3$};
			\node at (2.1,0) {$f_4$};
			\draw[thick] (0.9,0.3) arc (0:180:0.3 and 0.2);
			\draw[thick] (2.1,0.3) arc (0:180:0.3 and 0.2);
			\end{tikzpicture}
		+
			\begin{tikzpicture}[baseline]
			\draw [thick] (0,0.3) -- (0,-0.3) -- (2.4,-0.3) -- (2.4,0.3) -- (0,0.3);
			\node at (.3,0) {$f_1$};
			\node at (.9,0) {$f_2$};
			\node at (1.5,0) {$f_3$};
			\node at (2.1,0) {$f_4$};
			\draw[thick] (2.1,0.3) arc (0:180:0.9 and 0.6);
			\draw[thick] (1.5,0.3) arc (0:180:0.3 and 0.2);
			\end{tikzpicture}\\
		&=\<f_1,f_2\>_U\<f_3,f_4\>_U + \<f_1, f_4\>_U\<f_2,f_3\>_U.
	\end{align*}
Note that $\varphi$ then clearly takes a value of zero on all monomials of odd degree.\par
When $q\neq 0$, the chords may intersect and do so at the cost of a factor of $q$ for each intersection. Revisiting the previous example in this case we then have
	\begin{align*}
		\varphi_q(s_q(f_1)s_q(f_2)s_q(f_3)s_q(f_4)) &= 
			\begin{tikzpicture}[baseline]
			\draw [thick] (0,0.3) -- (0,-0.3) -- (2.4,-0.3) -- (2.4,0.3) -- (0,0.3);
			\node at (.3,0) {$f_1$};
			\node at (.9,0) {$f_2$};
			\node at (1.5,0) {$f_3$};
			\node at (2.1,0) {$f_4$};
			\draw[thick] (0.9,0.3) arc (0:180:0.3 and 0.2);
			\draw[thick] (2.1,0.3) arc (0:180:0.3 and 0.2);
			\end{tikzpicture}
		+
			\begin{tikzpicture}[baseline]
			\draw [thick] (0,0.3) -- (0,-0.3) -- (2.4,-0.3) -- (2.4,0.3) -- (0,0.3);
			\node at (.3,0) {$f_1$};
			\node at (.9,0) {$f_2$};
			\node at (1.5,0) {$f_3$};
			\node at (2.1,0) {$f_4$};
			\draw[thick] (1.5,0.3) arc (0:180:0.6 and 0.4);
			\draw[thick] (2.1,0.3) arc (0:180:0.6 and 0.4);
			\end{tikzpicture}
		+
			\begin{tikzpicture}[baseline]
			\draw [thick] (0,0.3) -- (0,-0.3) -- (2.4,-0.3) -- (2.4,0.3) -- (0,0.3);
			\node at (.3,0) {$f_1$};
			\node at (.9,0) {$f_2$};
			\node at (1.5,0) {$f_3$};
			\node at (2.1,0) {$f_4$};
			\draw[thick] (2.1,0.3) arc (0:180:0.9 and 0.6);
			\draw[thick] (1.5,0.3) arc (0:180:0.3 and 0.2);
			\end{tikzpicture}\\
		&=\<f_1,f_2\>_U\<f_3,f_4\>_U + q\<f_1,f_3\>_U\<f_2,f_4\>_U + \<f_1, f_4\>_U\<f_2,f_3\>_U.
	\end{align*}
We note that in computing $\varphi_q(s_q(f_1)s_q(f_2)s_q(f_3)s_q(f_4))=\<\Omega, s_q(f_1)s_q(f_2)s_q(f_3)s_q(f_4)\Omega\>_{U,q}$ by writing out $s_q(f_1)s_q(f_2)s_q(f_3)s_q(f_4)\Omega$, the term $ q\<f_1,f_3\>_U\<f_2,f_4\>_U$ comes from when $s_q(f_1)s_q(f_2)$ acts on $f_3\otimes f_4$ and the operator $l_q^*(f_2)$ ``skips'' over the the first vector in the tensor product (hence the factor of $q$).\par
It is a worthwhile exercise to restrict to the case when there is only a single operator $s_q(f)$ (so that all inner-products are $1$) and draw out the diagrams corresponding to $\varphi(s_q(f)^n)$ for $n=2,4,6,8$.
\end{rem}

The Tomita-Takesaki theory for $M_q$ is established in Lemma 1.4 of \cite{H}, which we recall here for convenience. Let $S$ denote the closure of the map $x\Omega\mapsto x^*\Omega$, and let $S=J\Delta^{1/2}$ be its polar decomposition so that $J$ and $\Delta$ are the modular conjugation and modular operator, respectively. Then for $n\geq 1$
	\begin{align}\label{Tomita-Takesaki_formulas}
		S(f_1\otimes\cdots\otimes f_n)&=f_n\otimes \cdots \otimes f_n	&\text{for }f_1,\ldots, f_n\in\H_\R;\notag\\
		\Delta(f_1\otimes\cdots\otimes f_n)&=(A^{-1}f_1)\otimes \cdots\otimes (A^{-1} f_n)		&\text{for }f_1,\ldots,f_n\in \H_\R\cap \text{dom}{A^{-1}};\\
		J(f_1\otimes\cdots\otimes f_n)&=(A^{-1/2} f_n)\otimes\cdots \otimes (A^{-1/2}f_n)		&\text{for }f_1,\ldots, f_n\in \H_\R\cap\text{dom}{A^{-1/2}}.\notag
	\end{align}
Denote by $\sigma_t^{\varphi_q}(\cdot)= \Delta^{it} \cdot \Delta^{-it}$ the modular automorphism group of $\varphi_q$.\par
Henceforth we assume $\dim(\H_\R)=N<\infty$. Consequently $A$ and $A^{-1}$ are bounded operators and hence $\{\sigma_t^{\varphi_q}\}_{t\in\R}$ extends to $\{\sigma_z^{\varphi_q}\}_{z\in\C}$. In particular for $a,b\in M$,
	\begin{align*}
		\varphi(ab)&=\< a^*\Omega, b\Omega\>_{U,q}=\< S a\Omega, b\Omega\>_{U,q}=\<Jb\Omega, \Delta^\frac{1}{2} a\Omega\>_{U,q}\\
				&=\< \Delta \Delta^{-\frac{1}{2}} Jb\Omega,a\Omega\>_{U,q} = \< \Delta b^*\Omega, a\Omega\>_{U,q}= \varphi(\sigma_{i}^{\varphi_q}(b) a).
	\end{align*}
Moreover, the action of $\Delta$ in (\ref{Tomita-Takesaki_formulas}) extends to $f_1,\ldots, f_n\in \H$.\par
From Remark 2.12 in \cite{S97} it follows that for a suitable orthonormal basis $\{e_1,\ldots,e_N\}$ of $(\H_\R,\<\cdot,\cdot\>)$, the generator $A$ can be represented as a matrix of the form
	\begin{equation}\label{matrix_form_A}
		A=\text{diag}\left( A_1,\ldots, A_L,1,\ldots,1\right),
	\end{equation}
where for each $k\in\{1,\ldots,L\}$
	\begin{equation}\label{matrix_form_A_2}
		A_k=\frac{1}{2}\left(\begin{array}{cc}
						\lambda_k+\lambda_k^{-1}		& -i\left(\lambda_k-\lambda_k^{-1}\right)	\\
						i\left(\lambda_k-\lambda_k^{-1}\right)	& \lambda_k+\lambda_k^{-1}	\end{array}\right)\in M_2(\C),
	\end{equation}
and $\lambda_k>0$. Note that
	\begin{equation*}
		A_k^{it}=\left(\begin{array}{cc}	\cos(t\log{\lambda_k})	&	-\sin(t\log{\lambda_k})\\
								\sin(t\log{\lambda_k}) 	& 	\cos(t\log{\lambda_k})\end{array}\right),
	\end{equation*}
which is a unitary matrix such that $(A_k^{it})^*=(A_k^{it})^\text{T}=A_k^{-it}$. $A$ has the following properties:
	\begin{enumerate}
		\item[1.] $\text{spectrum}(A)=\left\{1,\lambda_1^{\pm 1},\ldots, \lambda_L^{\pm 1}\right\}$;
		\item[2.] $A^\text{T}=A^{-1}$;
		\item[3.] $\left( A^{it}\right)^*=\left( A^{it}\right)^\text{T}=A^{-it}$; and
		\item[4.] for any fixed $i\in\{1,\ldots, N\}$, 
			\begin{equation*}
				\sum_{j=1}^N \left|[A]_{ij}\right|\leq \max\left\{1,\lambda_1^{\pm 1},\ldots, \lambda_L^{\pm 1}\right\}\leq \|A\|.
			\end{equation*}
	\end{enumerate}
For each $j=1,\ldots, N$, let $X_j^{(q)}=s_q(e_j)$ and write $X^{(q)}=(X_1^{(q)},\ldots, X_N^{(q)})$. Since $s_q$ is real linear, it follows that $M_q=W^*(X_1^{(q)},\ldots, X_N^{(q)})$. We observe that
	\begin{align*}
		\sigma_z^{\varphi_q}(X_j^{(q)})=\sum_{k=1}^N [A^{iz}]_{jk} X_k^{(q)},\qquad \forall z\in \C,
	\end{align*}
or using the vector notation:
	\begin{align}\label{modular_semicircular}
		\sigma_z^{\varphi_q}(X^{(q)})= A^{iz} X^{(q)},\qquad \forall z\in\C.
	\end{align}
Indeed, using (\ref{Tomita-Takesaki_formulas}) it is easy to see that
	\begin{align*}
		\sigma_z^{\varphi_q}(l_q(e_j))&= l_q(A^{-iz} e_j)\\
		\sigma_z^{\varphi_q}(l_q^*(e_j))&=l_q^*( A^{-i\bar{z}} e_j).
	\end{align*}
Equation (\ref{modular_semicircular}) follows from the above properties of $A$, the linearity of $l_q$, and the conjugate linearity of $l_q^*$.

\subsection{Derivations on $M_q$}\label{derivations_on_M_q}

For the remainder of this section we will consider a single fixed $q\in (-1,1)$, so that we may repress the superscript $(q)$ notation on $X^{(q)}_j$, and write $\mathscr{P}$ for the $*$-subalgebra $\C\<X_1,\ldots, X_N\>\subset M_q$ of non-commutative polynomials in $N$-variables. We also simplify notation with $M:=M_q$, $\varphi:=\varphi_q$, and $\sigma_z:=\sigma_z^{\varphi_q}$ for $z\in \C$.\par
For each $j\in\{1,\ldots,N\}$ we let $\delta_j\colon \mathscr{P}\rightarrow \mathscr{P}\otimes\mathscr{P}^{op}$ be Voiculescu's free-difference quotient:
	\begin{align*}
		\delta_j(X_{i_1}\cdots X_{i_n})=\sum_{k=1}^n \delta_{j=i_k} X_{i_1}\cdots X_{i_{k-1}}\otimes \left(X_{i_{k+1}}\cdots X_{i_n}\right)^\circ;
	\end{align*}
that is, $\delta_j$ is the unique derivation satisfying $\delta_j(X_i)=\delta_{j=i}1\otimes 1$. We set the following conventions for working with elementary tensors in $\mathscr{P}\otimes\mathscr{P}^{op}$:
	\begin{itemize}
		\item $(a\otimes b^\circ)\#(c\otimes d^\circ):= (ac)\otimes (b^\circ d^\circ)=(ac)\otimes (db)^\circ$;
		\item $(a\otimes b^\circ)\#c=acb$;
		\item $(a\otimes b^\circ)^*:=a^*\otimes (b^*)^\circ$;
		\item $(a\otimes b^\circ)^\dagger:= b^*\otimes (a^*)^\circ$;
		\item $(a\otimes b^\circ)^\diamond:= b\otimes a^\circ$;
		\item $m(a\otimes b^\circ):=ab$.
	\end{itemize}
We also define the left and right actions of $\mathscr{P}$ as:
	\begin{itemize}
		\item $c\cdot (a\otimes b^\circ):=(ca)\otimes b^\circ$;
		\item $(a\otimes b^\circ)\cdot c:=a\otimes (bc)^\circ$.
	\end{itemize}
Note that 
	\begin{align*}
		c\cdot (a\otimes b^\circ)&= (c\otimes 1^\circ)\# (a\otimes b^\circ),\ \text{and}\\
		(a\otimes b^\circ)\cdot c &= (1\otimes c^\circ )\# (a\otimes b^\circ).
	\end{align*}
We will usually suppress the notation ``$\circ$" and at times represent tensors  of monomials in $\mathscr{P}$ diagrammatically as follows:
	\begin{align}\label{box_notation}
	\begin{tikzpicture}[baseline]
	\draw[thick] (0,-.5) rectangle (2.5,.5);
	\node[left] at (0,0) {$X_{i_1}\cdots X_{i_n}\otimes X_{j_1}\cdots X_{j_m}=$};
	\node at (1.25,.25) {$i_1i_2\cdots i_{n-1}i_n$};
	\node at (1.25,-.25) {$j_mj_{m-1}\cdots j_2j_1$};
	\end{tikzpicture}.
	\end{align}
Then multiplication is neatly expressed as:
	\begin{align*}
	\begin{tikzpicture}[baseline]
	\draw[thick] (0,-.5) rectangle (1.5,.5);
	\node at (.75,.25) {$i_1\cdots i_n$};
	\node at (.75,-.25) {$j_m\cdots j_1$};
	\node at (1.75,0) {$\#$};
	\draw[thick] (2,-.5) rectangle (3.5,.5);
	\node at (2.75,.25) {$k_1\cdots k_p$};
	\node at (2.75,-.25) {$l_q\cdots l_1$};
	\node at (3.75,0) {$=$};
	\draw[thick] (4,-.5) rectangle (6.6,.5);
	\node at (5.3,.25) {$i_1\cdots i_nk_1\cdots k_p$};
	\node at (5.3,-.25) {$j_m\cdots j_1l_q\cdots l_1$};
	\end{tikzpicture}.
	\end{align*}
We note the involutions $*,\dagger,\diamond$ amount to horizontal reflection, vertical reflection, and $180^\circ$ rotation of the diagrams, respectively.\par
For $j,k\in\{1,\ldots, N\}$, we use the shorthand notation
	\begin{align*}
		\alpha_{jk}:=\left[ \frac{2}{1+A}\right]_{jk} =\<e_k,e_j\>_U.
	\end{align*}
Note that the last equality implies $\overline{\alpha_{jk}}=\alpha_{kj}$, $\alpha_{jj}=1$, and $|\alpha_{jk}|\leq 1$ for all $j,k\in\{1,\ldots, N\}$.\par
Let $\Xi_q \in HS(\mc{F}_q(\H))$ be the Hilbert-Schmidt operator on $\mc{F}_q(\H)$ given by the sum $\sum_{n=0}^\infty q^n P_n$ where $P_n\colon \mc{F}_q(\H) \rightarrow \H^{\otimes n}$ is the projection onto vectors of length $n$. We identify the Hilbert space generated by the GNS construction with respect to $\varphi\otimes\varphi^{op}$ with $L^2(M\bar{\otimes}M^{op},\varphi\otimes\varphi^{op})\cong HS(\mc{F}_q(\H))$ via $a\otimes b^\circ\mapsto \<\Omega, b\ \cdot\>a\Omega$ (\emph{cf.} Proposition 5.11 in \cite{V94}); in particular, $\Xi_0= P_0$ corresponds to $1\otimes 1$. Realize that the involution $\dagger$ defined above corresponds precisely with the adjoint operation in $HS(\mc{F}_q(\H))$. Consequently, $\Xi_q^\dagger=\Xi_q$ since, as a real sum of projections, it is a self-adjoint Hilbert-Schmidt operator.\par
For each $j=1,\ldots, N$ we define the derivation $\partial_j^{(q)}\colon\mathscr{P}\rightarrow \mathscr{P}\otimes \mathscr{P}^{op}$ by
	\begin{align*}
		\partial_j^{(q)}(P)=\sum_{k=1}^N \alpha_{kj}\delta(P)\#\Xi_q.
	\end{align*}
That is, $\partial_j^{(q)}$ is the unique derivation satisfying $\partial_j^{(q)}(X_i)=\alpha_{ij}\Xi_q$. We shall also consider the derivations
	\begin{align*}
		\bar{\partial}_j^{(q)}(P):=\sum_{k=1}^N \alpha_{jk}\delta_k(P)\# \Xi_q\qquad\text{ and }\qquad \tilde{\partial}_j^{(q)} :=\sum_{k=1}^N \alpha_{jk} \left(\delta_k(P)\# \Xi_q\right)^\diamond,
	\end{align*}
which are related to $\partial_j^{(q)}$ by
	\begin{align*}
		\partial_j^{(q)}(P)^\dagger=\bar{\partial}_j^{(q)}(P^*)\qquad\text{ and }\qquad \partial_j^{(q)}(P)^* = \tilde{\partial}_j^{(q)}(P^*).
	\end{align*}
We remark that in the tracial case ($U_t=1_t$), we have $\bar{\partial}_{j}^{(q)}=[\cdot, r_q(e_j)]$, where $r_q(e_j)$ is the right $q$-creation operator. This is precisely the derivation considered in Lemma 27 of \cite{D}.\par
From (\ref{modular_semicircular}) we see that
	\begin{align*}
		\partial_j^{(q)}(\sigma_{it}(X_k))=\sum_{l=1}^N [A^{-t}]_{kl} \alpha_{lj} \Xi_q = \left[\frac{2A^{-t}}{1+A}\right]_{kj} \Xi_q,
	\end{align*}
and thus $(\sigma_{-it}\otimes\sigma_{-it})\circ\partial_j\circ\sigma_{it}$ defines the unique derivation satisfying $X_k\mapsto \left[\frac{2A^{-t}}{1+A}\right]_{kj}\Xi_q$. In particular, since
	\begin{align*}
		\left[\frac{2A}{1+A}\right]_{kj}=\left[\frac{2}{1+A^{-1}}\right]_{kj}=\left[\frac{2}{1+A}\right]_{jk},
	\end{align*}
we see that
	\begin{equation}\label{differentiating_sigma_with_q}
		(\sigma_i\otimes\sigma_i)\circ\partial_j^{(q)}\circ\sigma_{-i}=\bar{\partial}_j^{(q)}.
	\end{equation}
The motivation for considering such derivations is precisely the following proposition.

\begin{prop}\label{adjoint_of_q-derivations}
View $\partial_j^{(q)}$ and $\bar{\partial}_j^{(q)}$ as densely defined operators from $L^2(\mathscr{P},\varphi)$ to $L^2(\mathscr{P}\otimes \mathscr{P}^{op},\varphi\otimes\varphi^{op})$. Then $1\otimes 1\in \dom{\partial_j^{(q)*}}$ with
	\begin{equation}\label{adjoint_of_1}
		\partial_j^{(q)*}(1\otimes 1)=X_j.
	\end{equation}
Moreover, $1\otimes1\in\dom{\bar{\partial}_j^{(q)*}}$ with
	\begin{equation}
		\bar{\partial}_j^{(q)*}(1\otimes 1)=\sigma_{-i}(X_j).
	\end{equation}
\end{prop}

\begin{rem}
The above proposition states that $X_1,\ldots, X_N$ (resp. $\sigma_{-i}(X_1),\ldots, \sigma_{-i}(X_N)$) are \emph{conjugate variables to $X$ with respect to the derivations $\partial_1^{(q)},\ldots, \partial_N^{(q)}$}  (resp. $\bar{\partial}_1^{(q)},\ldots, \bar{\partial}_N^{(q)}$) (\emph{cf.} Section 3 of \cite{V94}).
\end{rem}

\begin{proof}
Consider the monomial $P=X_{i_1}\cdots X_{i_n}\in\mathscr{P}$. Then,
	\begin{align*}
		\varphi(X_j P)=\<X_j\Omega, P\Omega\>_{U,q} = \< P_1X_j\Omega, P\Omega\>_{U,q} = \<P_1X_j\Omega, P_1P\Omega\>_{U,q},
	\end{align*}
where $P_1\in \B (\mc{F}_q(\H))$ is the projection onto tensors of length one. As $P$ is a product of the $X_{i_k}$, it is clear that $P_1P\Omega$ will be a linear combination of $e_{i_1},\ldots, e_{i_n}$, say $P_1P\Omega=\sum_{k=1}^n c_k e_{i_k}$. We claim that
	\begin{align*}
		c_k =\sum_{l=0}^\infty q^l \< P_l X_{i_{k-1}}\cdots X_{i_1}\Omega, P_l X_{i_{k+1}}\cdots X_{i_n}\Omega\>_{U,q}.
	\end{align*}
Indeed, diagrammatically each term contributing to $c_k$ is a pairing of the vectors $e_{i_1},\ldots, e_{i_n}$ with $e_{i_k}$ excluded. We can arrange such pairings according to the number of pairs whose connecting chords cross over $e_{i_k}$. Fix $l\geq 0$ and consider pairings with $l$ chords passing over $e_{i_k}$. Write $P=A_kX_{i_k}B_k$, then $P_lB_k\Omega$ gives pairings within $B_k$ that leave $l$ vectors unpaired. Hence $\<\Omega, A_kP_lB_k\Omega\>$ counts the pairings in which there are exactly $l$ pairs with one vector coming from $A_k$ and one coming from $B_k$. Since the cost of skipping over $e_{i_k}$ $l$ times is $q^l$ we see that
	\begin{equation*}
		c_k = \sum_{l=0}^\infty q^l \<\Omega, A_kP_l B_k\Omega\>_{U,q} = \sum_{l=0}^\infty q^l \< P_l A_k^*\Omega, P_l B_k\Omega\>_{U,q},
	\end{equation*}
as claimed. Thus
	\begin{align*}
		\< P_1X_j\Omega, P_1P\Omega\>_{U,q} =\sum_{k=1}^n \< P_1 X_j\Omega, e_{i_k}\>_{U,q} c_k = \sum_{k=1}^n \<e_j, e_{i_k}\>_U  \sum_{l=0}^\infty q^l \< P_l A_k^*\Omega, P_l B_k\Omega\>_{U,q}.
	\end{align*}\par
Now, we inductively orthonormalize the monomials $X_{\ul{i}}\in\mathscr{P}$ with respect to $\<\cdot\ \Omega,\cdot\ \Omega\>_{U,q}$ to obtain a basis $\{r_{\ul{j}}\}_{|\ul{j}|\geq 0}$ so that for each $l$, $\text{span}\{r_{\ul{j}}\colon |\ul{j}|=l\}=\text{span}\{X_{\ul{i}}\colon |\ul{i}|=l\}$. Then $P_lB_k=\sum_{|\ul{j}|=l} \<r_{\ul{j}}\Omega, B\Omega\>_{U,q} r_{\ul{j}}$ and using our identification with $L^2(M\bar{\otimes}M^{op},\varphi\otimes\varphi^{op})$ we see that $P_l = \sum_{|\ul{j}|=l} r_{\ul{j}}\otimes r_{\ul{j}}^*$. Thus we have
	\begin{align*}
		\varphi(X_jP) &= \sum_{k=1}^n \<e_j, e_{i_k}\>_{U}\sum_{l=0}^\infty q^l \< P_l A_k^*\Omega, P_l B_k\Omega\>_{U,q}= \sum_{k=1}^n \< e_j, e_{i_k}\>_U \sum_{l=0}^\infty q^l \sum_{|\ul{j}|=l}\< A_k^*\Omega, r_{\ul{j}}\Omega\>_{U,q}\<r_{\ul{j}}\Omega, B_k\Omega\>_{U,q}\\
			&= \sum_{k=1}^n \<e_j, e_{i_k}\>_U \sum_{l=0}^\infty q^l \sum_{|\ul{j}|=l} \varphi\otimes\varphi^{op}\left( A_k\otimes B_k \# r_{\ul{j}}\otimes r_{\ul{j}}^* \right) = \sum_{k=1}^n \<e_j, e_{i_k}\>_U \varphi\otimes \varphi^{op}\left( A_k\otimes B_k \#\Xi_q\right) \\
			&= \varphi\otimes \varphi^{op}\left( \partial_j^{(q)}P \right),
	\end{align*}
or $\<X_j, P\>_{\varphi} = \< 1\otimes 1, \partial_j^{(q)}P\>_{\varphi\otimes\varphi^{op}}$, which implies $\partial_j^{(q)*}(1\otimes 1)=X_j$.\par
Now,
	\begin{align*}
		\<\sigma_{-i}(X_j),P\>_{\varphi}&=\varphi(\sigma_i(X_j)P)=\varphi(PX_j)=\<P^*,\partial_j^{(q)*}(1\otimes 1)\>_{\varphi}=\<\bar{\partial}_j^{(q)}(P)^\dagger,1\otimes 1\>_{\varphi}\\
			&=\varphi\otimes\varphi^{op}( \bar{\partial}_j^{(q)}(P)^\diamond)=\varphi\otimes\varphi^{op}(\bar{\partial}_j^{(q)}(P))=\<1\otimes 1, \bar{\partial}_j^{(q)}(P)\>_{\varphi},
	\end{align*}
so that $1\otimes 1\in \text{dom}\left(\bar{\partial}_j^{(q)}\right)$ and $\bar{\partial}_j^{(q)*}(1\otimes 1)=\sigma_{-i}(X_j)$.
\end{proof}

\begin{cor}\label{difference_quotient_adjoint}
Viewing $\partial_j^{(q)}\colon L^2(\mathscr{P},\varphi)\rightarrow L^2(\mathscr{P}\otimes\mathscr{P}^{op},\varphi\otimes\varphi^{op})$ as a densely defined operator we have $\mathscr{P}\otimes\mathscr{P}^{op}\subset \dom{\partial_j^{(q)*}}$. In particular, if $\eta\in \dom{\partial_j^{(q)*}}$ and $P\in\mathscr{P}$ then
	\begin{align*}
		\partial_j^{(q)*}(\eta\cdot P)&=\partial_j^{(q)*}(\eta)\sigma_{-i}(P) - 1\otimes\varphi^{op} \left(\eta\# \hat{\sigma}^{\varphi}_i\circ\bar{\partial}_j^{(q)}(P)^\diamond\right),\ \text{and}\\
		\partial_j^{(q)*}( P\cdot\eta) &= P\partial_j^{(q)*}(\eta) - \varphi\otimes 1^{op}\left( \hat{\sigma}_i(\eta)\# \bar{\partial}_j^{(q)}(P)^\diamond\right),
	\end{align*}
where $\hat{\sigma}_z=\sigma_z\otimes \sigma_{\bar{z}}$ with $z\in \C$. In particular, for $P,Q\in \mathscr{P}$ we have
	\begin{align}\label{adjoint_formula}
		\partial_j^{(q)*}(P\otimes Q) = [1\otimes\sigma_{-i}](P\otimes Q)\# X_j -m\circ\left( 1\otimes\varphi\otimes \sigma_{-i}\right)\circ\left(1\otimes\bar{\partial}_j^{(q)}+\bar{\partial}_j^{(q)}\otimes 1\right)(P\otimes Q),
	\end{align}
or equivalently (using Equation (\ref{differentiating_sigma_with_q}))
	\begin{align}\label{adjoint_formula_2}
		\partial_j^{(q)*}(P\otimes Q) = [1\otimes\sigma_{-i}](P\otimes Q)\#X_j-m\circ\left( 1\otimes\varphi\otimes 1\right)\circ\left(1\otimes\partial_j^{(q)}+\bar{\partial}_j^{(q)}\otimes 1\right)\circ[1\otimes\sigma_{-i}](P\otimes Q).
	\end{align}
\end{cor}
\begin{proof}
We make the following notational simplifications: $\<\cdot,\cdot\>_{\varphi}=\<\cdot,\cdot\>$ and $\<\cdot,\cdot\>_{\varphi\otimes\varphi^{op}} = \<\cdot,\cdot\>_\otimes$. First note that for $A,B,C,D\in \mathscr{P}$ we have
	\begin{align*}
		\varphi\otimes\varphi^{op}( A\otimes B\# C\otimes D))&=\varphi\otimes \varphi^{op
}( (AC)\otimes (DB))= \varphi(AC)\varphi(DB)=\varphi(\sigma_i(C)A)\varphi(B\sigma_{-i}(D)) \\
			&= \varphi\otimes \varphi^{op} \left( (\sigma_i(C)\otimes \sigma_{-i}(D))\# A\otimes B\right)=\varphi\otimes \varphi^{op} (\hat{\sigma}_i(C\otimes D)\# A\otimes B).
	\end{align*}
Also observe that
	\begin{align*}
		\hat{\sigma}_i\left( (a\otimes b)^\dagger\right)= \hat{\sigma}_i (b^*\otimes a^*)= \sigma_i(b^*)\otimes \sigma_{-i}(a^*)= \sigma_{-i}(b)^* \otimes \sigma_i(a)^* = \hat{\sigma}_i(a\otimes b)^\dagger.
	\end{align*}
Now, let $Q\in\mathscr{P}$, then
	\begin{align*}
		\<\eta\cdot P, \partial_j^{(q)}(Q)\>_\otimes&=\<\eta, \partial_j^{(a)}(Q)\cdot P^*\>_\otimes=\<\eta, \partial_j^{(q)}(QP^*) - Q\cdot\partial_j^{(q)}(P^*)\>_\otimes\\
			&=\varphi\left(\left[\partial_j^{(q)*}(\eta)\right]^* QP^*\right) - \varphi\otimes\varphi^{op}\left(\hat{\sigma}_i\left(\partial_j^{(q)}(P^*)\right)\# \eta^*\# Q\otimes 1\right)\\
			&=\<\partial_j^{(q)*}(\eta)\sigma_{-i}(P),Q\> -\varphi\otimes\varphi^{op} \left( \hat{\sigma}_i\left(\bar{\partial}_j^{(q)}(P)\right)^\dagger \#\eta^*\#Q\otimes 1\right)\\
			&=\<\partial_j^*(\eta)\sigma_{-i}(P) -1\otimes\varphi^{op} \left(\eta\#\hat{\sigma}_i\circ\bar{\partial}_j^{(q)}(P)^\diamond\right), Q\>.
	\end{align*}
Similarly,
	\begin{align*}
		\<P\cdot\eta,\partial_j^{(q)}(Q)\>_\otimes&= \<\eta, P^*\cdot\partial_j^{(q)}(Q)\>_\otimes = \<\eta, \partial_j^{(q)}(P^*Q) - \partial_j^{(q)}(P^*)\cdot Q\>_\otimes\\ 
			&=\<\partial_j^{(q)*}(\eta), P^*Q\> - \varphi\otimes \varphi^{op}\left( \hat{\sigma}_i\circ \partial_j^{(q)}(P^*)\#\eta^*\# 1\otimes Q\right)\\
			&=\<P\partial_j^{(q)*}(\eta), Q\> - \<\sigma_{-i}\left(\varphi\otimes 1^{op}\left( \eta\#\hat{\sigma}_i\circ \bar{\partial}_j^{(q)}(P)^\diamond\right)\right), Q\>\\
			&=\< P\partial_j^{(q)*}(\eta) - [\varphi\otimes 1^{op}]\circ\hat{\sigma}_i\left( \eta\#\hat{\sigma}_i\circ \bar{\partial}_j^{(q)}(P)^\diamond\right), Q\>\\
			&=\< P\partial_j^{(q)*}(\eta) - \varphi\otimes 1^{op}\left( \hat{\sigma}_i(\eta)\# \bar{\partial}_j^{(q)}(P)^\diamond\right), Q\>.
	\end{align*}
Applying both of these formulas and (\ref{adjoint_of_1}) yields
	\begin{align*}
		\partial_j^{(q)*}(P\otimes Q)&=PX_j\sigma_{-i}(Q) -m\left( 1\otimes\left[\sigma_{-i}(\varphi\otimes 1^{op})\bar{\partial}_j^{(q)}\right] + \left[(1\otimes \varphi^{op})\bar{\partial}_j^{(q)}\right]\otimes \sigma_{-i}\right)(P\otimes Q)\\
		&=[1\otimes\sigma_{-i}](P\otimes Q)\# X_j -m\circ\left( 1\otimes\varphi\otimes \sigma_{-i}\right)\circ\left(1\otimes\bar{\partial}_j^{(q)}+\bar{\partial}_j^{(q)}\otimes 1\right)(P\otimes Q).
	\end{align*}
The equivalent form follows easily from Equation (\ref{differentiating_sigma_with_q}).
\end{proof}

For each $j$ we also define the \textit{$\sigma$-difference quotient} $\partial_j\colon\mathscr{P}\rightarrow\mathscr{P}\otimes\mathscr{P}^{op}$ as
	\begin{equation*}
		\partial_j=\sum_{k=1}^N \alpha_{kj}\delta_k,
	\end{equation*}
which is the unique derivation satisfying $\partial_j(X_k)=\alpha_{kj}1\otimes 1$. We see that
	\begin{align*}
		\partial_j^{(q)}(P)=\partial_j(P)\# \Xi_q.
	\end{align*}
For $q=0$, we have $\partial_j=\partial_j^{(0)}$ since $\Xi_0=1\otimes 1$, but otherwise $\partial_j\neq \partial_j^{(q)}$. We also consider
	\begin{align*}
		\bar{\partial}_j(P):=\sum_{k=1}^N \alpha_{jk} \delta_k(P)\qquad\text{ and }\qquad \tilde{\partial}_j(P) :=\sum_{k=1}^N \alpha_{jk} \delta_k(P)^\diamond,
	\end{align*}
which are related to $\partial_j(P)$ in the expected way. Furthermore, we see that
	\begin{equation}\label{differentiating_sigma}
		(\sigma_i\otimes\sigma_i)\circ\partial_j\circ\sigma_{-i}=\bar{\partial}_j,
	\end{equation}
by the same argument that produced (\ref{differentiating_sigma_with_q}).\par
These latter derivations do not depend on $q$ and in fact could have been defined on $\C\<t_1,\ldots, t_N\>$ where the $t_j$ are some abstract indeterminates. This ``universality'' means that they are suitable for stating a Schwinger-Dyson equation (\emph{cf.} Subsection \ref{S-D_section}), which is a non-commutative differential equation satisfied by a unique state under certain restrictions. This uniqueness is precisely what will allow us to to establish the state-preserving isomorphism $M_q\cong M_0$, for small $|q|$. 
	
\subsection{The Banach algebra $\mathscr{P}^{(R,\sigma)}$ and norm $\|\cdot\|_{R,\sigma}$.}\label{setup}

We use the convention that an underline connotes a multi-index: $\ul{j}=(j_1,\ldots, j_n) \in N^n$ for some $n$. Then $|\ul{j}|$ gives the length of the multi-index. We write $\ul{j}\cdot \ul{k}$  to mean the concatenation of multi-indices $\ul{j}$ and $\ul{k}$: $(j_1,\ldots,j_n,k_1,\ldots,k_m)$. We also allow concatenation of multi-indices with single indices: $\ul{j}\cdot l=(j_1,\ldots,j_n,l)$. Monomials of the form $X_{j_1}\cdots X_{j_n}$ may be denoted by $X_{\ul{j}}$ when $\ul{j}=(j_1,\ldots, j_n)$. Hence an arbitrary $P\in\mathscr{P}$ may be written as
	\begin{equation*}
		P=\sum_{n=0}^{\deg{P}} \sum_{|\ul{j}|=n} c(\ul{j}) X_{\ul{j}},
	\end{equation*}
with $c(\ul{j})\in\C$. Denote the reversed multi-index by $\ul{j}^{-1}=(j_n,\ldots, j_1)$, then $X_{\ul{j}}^*=X_{\ul{j}^{-1}}$. For each $n\geq 0$, we let $\pi_n\colon \mathscr{P}\rightarrow \mathscr{P}$ be the projection onto monomials of degree $n$:
	\begin{equation*}
		\pi_n(P)=\sum_{|\ul{j}|=n} c(\ul{j}) X_{\ul{j}}.
	\end{equation*}
For $R>0$ we consider the norm $\|\cdot \|_R$ defined in \cite{SG11}:
	\begin{align*}
		\|P\|_R= \sum_{n=0}^{\deg{P}} \sum_{|\ul{j}|=n} |c(\ul{j})| R^n.
	\end{align*}\par
	Denote the \textit{centralizer of $\varphi$ in $\mathscr{P}$} by $\mathscr{P}_{\varphi}=\mathscr{P}\cap M_\varphi$, where $M_\varphi=\{a\in M\colon \sigma_i(a)=a\}$. Observe that as $\sigma_i$ does not change the degree of a monomial (i.e. $[\sigma_i,\pi_n]=0$ for each $n$), $P\in\mathscr{P}_\varphi$ iff $\pi_n(P)\in\mathscr{P}_\varphi$ for every $n\geq 0$.\par
Define a map on monomials by
	\begin{align*}
		\rho(X_{j_1}\cdots X_{j_n})=\sigma_{-i}(X_{j_n})X_{j_1}\cdots X_{j_{n-1}},
	\end{align*}
then by letting $\rho(c)=c$ for $c\in\C$ we can extend this to a linear map $\rho\colon\mathscr{P}\rightarrow\mathscr{P}$. We refer to $\rho^k(P)$ as a \textit{$\sigma$-cyclic rearrangement of $P$}. We note that
	\begin{align*}
		\rho^{-1}(X_{j_1}\cdots X_{j_n})=X_{j_2}\cdots X_{j_n}\sigma_i(X_{j_1}).
	\end{align*}
We define
	\begin{align*}
		\|P\|_{R,\sigma}:=\sum_{n=0}^{\deg{P}} \sup_{k_n\in\Z} \|\rho^{k_n}(\pi_n(P))\|_R \in [0,\infty].
	\end{align*}
Then from the norm properties of $\|\cdot\|_R$ and the subadditivity of the supremum it is easy to see that for $P,Q\in\mathscr{P}$ and $c\in\C$
	\begin{enumerate}
			\item[1.] $\| cP\|_{R,\sigma}=|c| \|P\|_{R,\sigma}$
			
			\item[2.] $\|P+Q\|_{R,\sigma}\leq \|P\|_{R,\sigma}+\|Q\|_{R,\sigma}$
			
			\item[3.] $\|P\|_{R,\sigma}=0\ \Longrightarrow\ P=0$.
	\end{enumerate}
Hence, $\|\cdot\|_{R,\sigma}$ restricted to the set $\{P\in\mathscr{P}\colon \|P\|_{R,\sigma}<\infty\}=:\mathscr{P}^{finite}$ is a norm.\par
Observe that $\rho^k(\sigma_{-im}(X_{j_1}\cdots X_{j_n}))=\rho^{k+mn}(X_{j_1}\cdots X_{j_n})$, so $\|\cdot\|_{R,\sigma}$ is invariant under $\sigma_{im}$, $m\in\Z$. Consequently, $\mathscr{P}_\varphi\subset \mathscr{P}^{finite}$. Indeed, if $P\in\mathscr{P}_\varphi$ then $\pi_n(P)\in\mathscr{P}_\varphi$ for all $n$. Hence $\rho^{k_n}(\pi_n(P))=\rho^{l_n}(\pi_n(P))$ where $k_n\equiv l_n\mod{n}$ and $l_n\in\{0,\ldots,n-1\}$. Consequently
	\begin{align*}
		\|P\|_{R,\sigma} =\sum_{n=0}^{\deg{P}} \max_{l_n\in\{0,\ldots,n-1\}} \|\rho^{l_n}(\pi_n(P))\|_R <\infty.
	\end{align*}
In fact, since $\|\cdot\|_R$ is a Banach norm and 
	\begin{align*}
		\|\sigma_{-i}(X_j)\|_R=\sum_{k=1}^N |[A]_{jk}| R\leq \|A\| R,
	\end{align*}
it is easy to see that $\|\rho^{l_n}(\pi_n(P))\|_R\leq \|A\|^{n-1} \|\pi_n(P)\|_R$ for $n\geq 1$ and any $l_n\in\{1,\ldots,n-1\}$. Thus we have the bound
	\begin{align*}
		\|P\|_{R,\sigma}\leq \|A\|^{\deg{P}-1}\|P\|_R,\qquad \text{for }P\in\mathscr{P}_\varphi.
	\end{align*}
\par
We let $\mathscr{P}^{(R)}$ and $\mathscr{P}^{(R,\sigma)}$ denote the closures of $\mathscr{P}$ and $\mathscr{P}^{finite}$ with respect to the norms $\|\cdot\|_R$ and $\|\cdot\|_{R,\sigma}$, respectively. Both can be thought of as non-commutative power series: the former whose radii of convergence are at least $R$ and the latter whose radii of convergence for each $\sigma$-cyclic rearrangement are at least $R$. Note that $\pi_n$ can be extended to both $\mathscr{P}^{(R)}$ and $\mathscr{P}^{(R,\sigma)}$ with
	\begin{align*}
		\|P\|_R=\sum_{n=0}^\infty \|\pi_n(P)\|_R\qquad\text{ and }\qquad \|P\|_{R,\sigma}=\sum_{n=0}^\infty \|\pi_n(P)\|_{R,\sigma}.
	\end{align*}\par
We claim that $\mathscr{P}^{(R,\sigma)}$ is a Banach algebra. It suffices to show $\|PQ\|_{R,\sigma}\leq \|P\|_{R,\sigma}\|Q\|_{R,\sigma}$. Initially we consider the case $P=\sum_{|\ul{i}|=m}a(\ul{i})X_{\ul{i}}$ and $Q=\sum_{|\ul{j}|=n} b(\ul{j})X_{\ul{j}}$ for $m,n\geq 0$. Fix $k\in\Z$ and write $k=r(m+n)+l$. We treat the case $0\leq l \leq n$, the case $n<l<n+m$ being similar. We also introduce the following notation for $|\ul{i}|=|\ul{j}|=n$ and $t\in\R$:
	\begin{equation*}
		A^t(\ul{i},\ul{j})=\prod_{u=1}^n \left[A^t\right]_{i_u j_u}.
	\end{equation*}
Now,
	\begin{align*}
		\rho^k(PQ)=&\sum_{\substack{|\ul{i}|=m\\ |\ul{j}|=n-l \\ \ul{k}|=l}} a(\ul{i})b\left(\ul{j}\cdot\ul{k}\right)\sigma_{-i(r+1)}(X_{\ul{k}})\sigma_{-ir}(X_{\ul{i}}X_{\ul{j}}) \\
			=& \sum_{\substack{|\ul{i}|=m\\ |\ul{j}|=n-l\\  |\ul{k}|=l}}\sum_{\substack{|\ul{\hat{i}}|=m\\ |\ul{\hat{j}}|=n-l \\ |\ul{\hat{k}}|=l}} a(\ul{i})b\left(\ul{j}\cdot\ul{k}\right) A^{r+1}\left( \ul{k}, \ul{\hat{k}}\right) A^r\left(\ul{i},\ul{\hat{i}}\right)A^r\left(\ul{j},\ul{\hat{j}}\right) X_{\ul{\hat{k}}}X_{\ul{\hat{i}}} X_{\ul{\hat{j}}},
	\end{align*}
hence
	\begin{align*}
		\left\|\rho^k(PQ)\right\|_{R} =&\sum_{\substack{|\ul{\hat{i}}|=m\\ |\ul{\hat{j}}|=n-l \\ |\ul{\hat{k}}|=l}} \left|\sum_{\substack{|\ul{i}|=m\\ |\ul{j}|=n-l\\  |\ul{k}|=l}} a(\ul{i})b\left(\ul{j}\cdot\ul{k}\right) A^{r+1}\left( \ul{k}, \ul{\hat{k}}\right)  A^r\left(\ul{i},\ul{\hat{i}}\right)A^r\left(\ul{j},\ul{\hat{j}}\right)\right| R^{n+m}\\
					=&\sum_{|\ul{\hat{i}}|=m} \left|\sum_{|\ul{i}|=m}a(\ul{i}) A^r\left(\ul{i},\ul{\hat{i}}\right)\right| R^m \sum_{\substack{|\ul{\hat{j}}|=n-l\\ |\ul{\hat{k}}|=l}} \left| \sum_{\substack{|\ul{j}|=n-l\\ |\ul{k}|=l}}b\left(\ul{j}\cdot\ul{k}\right) A^{r+1}\left( \ul{k},\ul{\hat{k}}\right) A^r\left( \ul{j}, \ul{\hat{j}}\right) \right|R^n\\
					=&\|\rho^{rm}(P)\|_R\|\rho^{rn+l}(Q)\|_R\leq \|P\|_{R,\sigma}\|Q\|_{R,\sigma}.
	\end{align*}
Thus $\|PQ\|_{R,\sigma}\leq \|P\|_{R,\sigma}\|Q\|_{R,\sigma}$.\par
Now let $P,Q\in\mathscr{P}^{(R,\sigma)}$ be arbitrary. Then
	\begin{align*}
		\|PQ\|_{R,\sigma}&\leq \sum_{m,n=0}^\infty \| \pi_m(P)\pi_n(Q)\|_{R,\sigma} \leq \sum_{m,n=0}^\infty \|\pi_m(P)\|_{R,\sigma}\|\pi_n(Q)\|_{R,\sigma}\\
					& = \left(\sum_{m=0}^\infty \|\pi_m(P)\|_{R,\sigma}\right)\left(\sum_{n=0}^\infty \|\pi_n(Q)\|_{R,\sigma}\right)=\|P\|_{R,\sigma}\|Q\|_{R,\sigma}.
	\end{align*}
Hence $\mathscr{P}^{(R,\sigma)}$ is a Banach algebra.\par
Since $\|\cdot \|_R$ is dominated by $\|\cdot \|_{R,\sigma}$, we can embed $\mathscr{P}^{(R,\sigma)}$ into $\mathscr{P}^{(R)}$. Furthermore, if $R\geq\|X_1\|,\ldots, \|X_N\|$ then $\|\cdot\|_R$ dominates the operator norm and hence we can embed $\mathscr{P}^{(R)}$ into $M$. From Lemma 4 in \cite{BS} we see that $\|X_j\| \leq \frac{2}{1-|q|}$ for all $j=1,\ldots, N$, so we restrict ourselves to $R\geq \frac{2}{1-|q|}$ from now on and consider $\mathscr{P}^{(R,\sigma)}\subset \mathscr{P}^{(R)}\subset M$ as subalgebras. We let $\mathscr{P}_\varphi^{(R,\sigma)}$ and $\mathscr{P}_\varphi^{(R)}$ denote their respective intersections with $M_\varphi$.\par
We shall also use $\|\cdot\|_{R,\sigma}$ to denote the norm on $\left(\mathscr{P}^{(R,\sigma)}\right)^N$ defined by
	\begin{equation*}
		\| (P_1,\ldots, P_N)\|_{R,\sigma}=\max\{\|P_1\|_{R,\sigma},\ldots, \|P_N\|_{R,\sigma}\}.
	\end{equation*}


\subsection{The operators $\mathscr{N}$, $\Sigma$, $\mathscr{S}$, $\Pi$, $\J_\sigma$, $\J$ and $\D$.}\label{tensor_product_notation}

The maps $\mathscr{N}$, $\Sigma$, and $\Pi$ are defined as in \cite{SG11},  but we recall them here for convenience. $\mathscr{N}$ is defined on monomials by
	\begin{align*}
		\mathscr{N}(X_{\ul{i}})=|\ul{i}|X_{\ul{i}},
	\end{align*}
and is linearly extended to a map $\mathscr{N}\colon \mathscr{P}^{(R)}\rightarrow\mathscr{P}^{(R)}$. $\Pi\colon \mathscr{P}^{(R)}\rightarrow \mathscr{P}^{(R)}$ in terms of our present notation is simply $1-\pi_0$: it is the projection onto power series with zero constant term. Lastly, $\Sigma$ is the inverse of $\mathscr{N}$ precomposed with $\Pi$:
	\begin{equation*}
		\Sigma (X_{\ul{i}})= \frac{1}{|\ul{i}|} X_{\ul{i}},
	\end{equation*}
if $|\ul{i}| >0$ and is zero otherwise.\par
Next we consider the following map defined on monomials as:
	\begin{equation*}
		\mathscr{S}(X_{i_1}\cdots X_{i_n})=\frac{1}{n}\sum_{k=0}^{n-1} \rho^k(X_{i_1}\cdots X_{i_n}),
	\end{equation*}
and on constants as simply $\mathscr{S}(c)=c$. For $n\geq 0$ and $P\in \pi_n\left(\mathscr{P}_\varphi\right)$,
	\begin{align*}
		\rho(\mathscr{S}(P))=\frac{1}{n}\sum_{k=0}^{n-1}\rho^{k+1}(P) = \frac{1}{n}\left( \sigma_{-i}(P)+\sum_{k=1}^{n-1} \rho^k(P)\right)= \frac{1}{n}\left( P+\sum_{k=1}^{n-1}\rho^k(P)\right)=\mathscr{S}(P).
	\end{align*}
And of course $\rho(\mathscr{S}(c))=c=\mathscr{S}(c)$. Thus if we denote the set of \textit{$\sigma$-cyclically symmetric} elements by $\mathscr{P}^{(R,\sigma)}_{c.s.}=\{P\in\mathscr{P}^{(R,\sigma)}\colon \rho(P)=P\}$, then 
	\begin{align*}
		\mathscr{S}\left(\mathscr{P}^{(R,\sigma)}_\varphi\right)\subset \mathscr{P}^{(R,\sigma)}_{c.s.}\subset\mathscr{P}^{(R,\sigma)}_{\varphi},
	\end{align*}
with the last inclusion following from the fact that $\rho^n(\pi_n(P))=\sigma_{-i}(\pi_n(P))$ and $P\in\mathscr{P}^{(R,\sigma)}_\varphi$ iff $\pi_n(P)\in \mathscr{P}_\varphi$ for each $n$. Moreover, $\mathscr{S}$ is a contraction on $\mathscr{P}^{(R,\sigma)}_{\varphi}$ with respect to the $\|\cdot\|_{R,\sigma}$. Indeed, since $\|Q\|_{R,\sigma}=\|Q\|_R$ for $Q\in\mathscr{P}^{(R,\sigma)}_{c.s.}$, for $P\in\mathscr{P}^{(R,\sigma)}_{\varphi}$ we have
	\begin{align*}
		\left\| \mathscr{S}(P)\right\|_{R,\sigma}=\|\mathscr{S}(P)\|_{R}\leq \sum_{n=0}^\infty \frac{1}{n}\sum_{k=0}^{n-1}\|\rho^k(\pi_n(P))\|_R \leq \sum_{n=0}^\infty \| \pi_n(P)\|_{R,\sigma} =\|P\|_{R,\sigma}.
	\end{align*}\par
If $f=(f_1,\ldots, f_N)$ with $f_j\in\mathscr{P}$ then we write $\J f,\J_\sigma f \in M_{N}(\mathscr{P}\otimes\mathscr{P}^{op})$ for the matrices given by
	\begin{align*}
		[\J f]_{ij}=\delta_j f_i\qquad\text{ and }\qquad [\J_\sigma f]_{ij}=\partial_j f_i.
	\end{align*}
On elements $Q\in M_N(\mathscr{P}\otimes \mathscr{P}^{op})$ we define the adjoint, transpose, and dagger involution as:
	\begin{align*}
		[Q^*]_{ij}&=[q]_{ji}^*,\\
		[Q^\text{T}]_{ij}&=[q]_{ji}^\diamond,\\
		[Q^\dagger]_{ij}&=[q]_{ij}^\dagger.
	\end{align*}
Thus $Q^*=(Q^\dagger)^\text{T}=(Q^\text{T})^\dagger$. Consequently, we define
	\begin{align*}
		[\bar{\J}_\sigma f]_{ij}=\bar{\partial}_j f_i\qquad\text{ and }\qquad [\tilde{\J}_\sigma f]_{ij} = \tilde{\partial}_i f_j,
	\end{align*}
so that $(\J_\sigma f)^\dagger= \bar{\J}_\sigma (f^*)$ and $(\J_\sigma f)^*=\tilde{\J}_\sigma(f^*)$.\par

Recall $X=(X_1,\ldots, X_N)$ and observe that $[\J_\sigma X]_{ij}=\alpha_{ij}1\otimes 1$.  So after embedding $M_{N}(\C)$ into $M_{N} (\mathscr{P}\otimes\mathscr{P}^{op})$ in the obvious way,  $\J_\sigma X$ and $\frac{2}{1+A}$ can be used interchangeably. Consequently it is clear that $\J_\sigma X$ is self-adjoint (with respect to the adjoint defined above) and invertible with inverse satisfying $[\J_\sigma X^{-1}]_{ij} = \left[\frac{1+A}{2}\right]_{ij} 1\otimes 1$.\par

We also define the multiplication for $Q,Q'\in M_N(\mathscr{P}\otimes\mathscr{P}^{op})$ and left actions on $f=(f_1,\ldots, f_N),\ g=(g_1,\ldots, g_N)\in \mathscr{P}^N$ by
	\begin{align*}
		[Q\# Q']_{i,j} &= \sum_{k=1}^N [Q]_{ik}\#[Q']_{kj} \in \mathscr{P}\otimes\mathscr{P}^{op},\text{ for }i,j\in\{1,\ldots, N\},\\
		Q\# g&= \left( \sum_{j=1}^N [Q]_{ij}\# g_j\right)_{i=1}^N \in \mathscr{P}^N,\text{ and}\\
		f\# g &= \sum_{j=1}^N f_jg_j \in \mathscr{P}.
	\end{align*}
For $Q\in M_N(\mathscr{P}\otimes\mathscr{P}^{op})$ we extend the notation of (\ref{box_notation}) by writing
	\begin{align*}
	\begin{tikzpicture}
	\draw[thick] (0,0) rectangle (2,1);
	\node[left] at (0,.5) {$[Q]_{ij}=$};
	\node at (1,.75) {$k_1\cdots k_n$};
	\node at (1,.25) {$l_m\cdots l_1$};
	\node[right] at (0,.5) {$i$};
	\node[left] at (2,.5) {$j$};
    	\end{tikzpicture}
	\end{align*}
when $[Q]_{ij}=X_{k_1}\cdots X_{k_n}\otimes X_{l_1}\cdots X_{l_m}$, $i,j\in\{1,\ldots, N\}$.

Lastly, we define the $j$th \textit{$\sigma$-cyclic derivative} $\D_j\colon\mathscr{P}\rightarrow\mathscr{P}$ by
	\begin{equation*}
		\D_j(X_{k_1}\cdots X_{k_n})= \sum_{l=1}^n \alpha_{j k_l} \sigma_{-i}(X_{k_{l+1}}\cdots X_{k_n})X_{k_1}\cdots X_{k_{l-1}}.
	\end{equation*}
$\D_j$ can also be written as $m\circ\diamond\circ(1\otimes\sigma_{-i})\circ\bar{\partial}_j$. Let $\D P=(\D_1 P,\ldots, \D_N P)\in\mathscr{P}^N$ be the \textit{$\sigma$-cyclic gradient}. We also define
	\begin{align*}
		\bar{\D}_j(X_{k_1}\cdots X_{k_n})= \sum_{l=1}^n \alpha_{k_l j} X_{k_{l+1}}\cdots X_{k_n}\sigma_i(X_{k_1}\cdots X_{k_{l-1}}),
	\end{align*}
or $\bar{\D}_j=m\circ\diamond\circ(\sigma_i\otimes 1)\circ\partial_j$. Then $(\D_j P)^*=\bar{\D}_j (P^*)$, and from (\ref{differentiating_sigma}) we also have $\D_j\circ\sigma_i=\bar{\D}_j$.


\subsection{The norm $\|\cdot\|_{R\otimes_\pi R}$.}\label{projective_tensor_norm}

Following \cite{SG11}, we denote by $\|\cdot\|_{R\otimes_\pi R}$ the projective tensor product norm on $\mathscr{P}\otimes \mathscr{P}^{op}$; that is,
	\begin{align*}
		\left\| \sum_i a_i\otimes b_i \right\|_{R\otimes_\pi R} = \sup_{\eta} \left\| \eta\left(\sum_i a_i\otimes b_i\right)\right\|,
	\end{align*}
where the supremum is taken over all maps $\eta$ valued in a Banach algebra such that $\eta(a\otimes 1)$ and $\eta(1\otimes b)$ commute and have norms bounded by $\|a\|_R$ and $\|b\|_R$, respectively. In particular, letting $\eta$ be given by left and right multiplication on $\mathscr{P}$ we see that for $D\in \mathscr{P}\otimes\mathscr{P}^{op}$ and $g\in\mathscr{P}$, we have
	\begin{align*}
		\| D\# g\|_R \leq \|D\|_{R\otimes_\pi R} \|g\|_R.
	\end{align*}
We extend the norm to $\left(\mathscr{P}\otimes\mathscr{P}^{op}\right)^N$ by putting for $F=(F_1,\ldots, F_N)\in \left(\mathscr{P}\otimes\mathscr{P}^{op}\right)^N$
	\begin{align*}
		\|F\|_{R\otimes_\pi R} = \max_{1\leq i \leq n} \|F_i\|_{R\otimes_\pi R}.
	\end{align*}
The same symbol is used to denote the norm imposed on $M_N(\mathscr{P}\otimes\mathscr{P}^{op})$ by identifying it with the Banach space of left multiplication operators on $\left(\mathscr{P}\otimes\mathscr{P}^{op}\right)^N$. In \cite{SG11} it is noted that this norm is given by
	\begin{align*}
		\| Q\|_{R\otimes_\pi R} = \max_{1\leq i\leq N} \sum_{j=1}^N \| [Q]_{ij}\|_{R\otimes_\pi R}.
	\end{align*}


\subsection{Cyclic derivatives of $\sigma$-cyclically symmetric polynomials}

Suppose $g\in\pi_n\left(\mathscr{P}^{(R,\sigma)}_{c.s.}\right)$ and write $g=\sum_{|\ul{j}|=n} c(\ul{j}) X_{\ul{j}}$. Then the condition $\rho^l(g)=g$ for $l\in \{0,\ldots, n-1\}$ implies
	\begin{align*}
		 g&=\rho^l(g)=\sum_{\substack{|\ul{j}|=n-l\\ |\ul{k}|=l}} c(\ul{j}\cdot \ul{k}) \sigma_{-i}(X_{\ul{k}})X_{\ul{j}} \\
		 	&= \sum_{\substack{|\ul{j}|=n-l\\ |\ul{k}|=l}} c(\ul{j}\cdot \ul{k})\sum_{|\ul{i}|=l} A(\ul{k},\ul{i}) X_{\ul{i}}X_{\ul{j}} = \sum_{\substack{|\ul{i}|=l\\ |\ul{j}|=n-l}} \left\{ \sum_{|\ul{k}|=l} c(\ul{j}\cdot \ul{k})A(\ul{k},\ul{i})\right\} X_{\ul{i}}X_{\ul{j}}.
	\end{align*}
Hence
	\begin{align}\label{cyclically_symmetric_coefficients_positive}
		c(\ul{i}\cdot\ul{j})=\sum_{|\ul{k}|=l} c(\ul{j}\cdot \ul{k}) A(\ul{k},\ul{i}).
	\end{align}
A similar computation using $l\in \{-n+1,\ldots, -1,0\}$ yields
	\begin{align}\label{cyclically_symmetric_coefficients_negative}
		c(\ul{i}\cdot\ul{j})=\sum_{|\ul{k}|=l} c(\ul{k}\cdot\ul{i}) A^{-1}(\ul{k}\cdot\ul{j}).
	\end{align}
Since $\rho^n(g)=\sigma_{-i}(g)$ for $g\in\pi_n(\mathscr{P})$, we can use Equation (\ref{cyclically_symmetric_coefficients_positive}) to characterize the coefficients of $g\in \pi_n\left(\mathscr{P}_\varphi\right)$:
	\begin{align}\label{centralizer_coefficients}
		c(\ul{i})=\sum_{|\ul{k}|=n} c(\ul{k}) A(\ul{k},\ul{i}).
	\end{align}
With these formulas in hand, the following lemmas are easily obtained.

\begin{lem}\label{cyclic_derivative_bounded}
For $P = \sum_{|\ul{i}|=n}c(\ul{i})X_{\ul{i}} \in \pi_n\left(\mathscr{P}_{c.s.}\right)$ and each $t\in\{1,\ldots, N\}$ we have
	\begin{equation}\label{cyclic_derivative_of_cyclically_symmetric}
		\D_t \Sigma P=\sum_{|\ul{i}|=n} \alpha_{t i_n} c(\ul{i}) X_{i_1}\cdots X_{i_{n-1}}.
	\end{equation}
Moreover, $\D\Sigma$ can be extended to a bounded operator $\D\Sigma\colon \mathscr{P}_{c.s.}^{(R,\sigma)}\rightarrow \left(\mathscr{P}^{(R)}\right)^N$ with $\|\D\Sigma \|\leq \frac{1}{R}$. Additionally, for $1<S<R$, $\D$ can be extended to a bounded operator $\D\colon \mathscr{P}_{c.s.}^{(R,\sigma)}\rightarrow \left(\mathscr{P}^{(S)}\right)^N$ with $\|\D\| \leq C\left(\frac{R}{S}\right)$ depending only on the ratio $\frac{R}{S}$.
\end{lem}
\begin{proof}
Let $P=\sum_{|\ul{i}|=n} c(\ul{i}) X_{\ul{i}}$. Equation (\ref{cyclic_derivative_of_cyclically_symmetric}) follows easily from Equation (\ref{cyclically_symmetric_coefficients_positive}), which then implies
	\begin{align*}
		\| \D_t\Sigma P\|_R= \left\| \sum_{|\ul{i}|=n} \alpha_{t i_n} c(\ul{i}) X_{i_1}\cdots X_{i_{n-1}}\right\|_R \leq \sum_{|\ul{i}|=n} |c(\ul{i})| R^{n-1} = \frac{1}{R} \|P\|_{R}=\frac{1}{R} \|P\|_{R,\sigma}.
	\end{align*}
So for arbitrary $P\in \mathscr{P}_{c.s.}$ we have
	\begin{align*}
		\| \D\Sigma P\|_R =\max_{t\in\{1,\ldots,N\}} \|\D_t\Sigma P\|_R\leq \sum_{n=0}^{\deg{P}} \max_{t\in\{1,\ldots,N\}} \|\D_t\Sigma \pi_n(P)\|_R \leq \sum_{n=0}^{\deg{P}} \frac{1}{R} \|\pi_n(P)\|_{R,\sigma} = \frac{1}{R} \|P\|_{R,\sigma},
	\end{align*}
and so $\D\Sigma$ extends to $\mathscr{P}_{c.s.}^{(R,\sigma)}$ with the claimed bound on its norm.\par
Considering only $\D$,  (\ref{cyclic_derivative_of_cyclically_symmetric}) implies
	\begin{align*}
		\D_t  P=n\sum_{|\ul{i}|=n} \alpha_{t i_k} c(\ul{i}) X_{i_1}\cdots X_{i_{n-1}}
	\end{align*}
for $P\in\pi_n\left(\mathscr{P}_{c.s.}\right)$. Hence
	\begin{align*}
		\| \D P\|_S \leq n \sum_{|\ul{i}|=n} |c(\ul{i})| S^{n-1} = n\left(\frac{S}{R}\right)^{n-1} \sum_{|\ul{i}|=n} |c(\ul{i})| R^{n-1} = \frac{ n S^{n-1}}{R^n} \|P\|_{R,\sigma}.
	\end{align*}
A routine computation shows for each $n$
	\begin{align*}
		\frac{ n S^{n-1}}{R^n} \leq \frac{ c S^{c-1}}{R^c}\leq c\left(\frac{S}{R}\right)^c=:C\left(\frac{R}{S}\right),
	\end{align*}
where $c=\frac{1}{\ln(R/S)}$. The rest of the argument then proceeds as in the previous case.
\end{proof}

\begin{lem}\label{D_of_S}
For $P\in\mathscr{P}_{\varphi}^{(R,\sigma)}$
	\begin{equation*}
		\D\mathscr{S}\Pi P=\D P.
	\end{equation*}
\end{lem}
\begin{proof}
Suppose $P\in \pi_n(\mathscr{P}_\varphi)$. The cases $n=0,1$ are clear so suppose $n\geq 2$. Write $P=\sum_{|\ul{i}|=n} c(\ul{i}) X_{\ul{i}}$, then
	\begin{align*}
		\mathscr{S}\Pi P&=\frac{1}{n}\sum_{l=0}^{n-1} \sum_{\substack{|\ul{j}|=n-l\\ |\ul{k}|=l}}c(\ul{j}\cdot \ul{k})\sigma_{-i}(X_{\ul{k}}) X_{\ul{j}} = \frac{1}{n}\sum_{l=0}^{n-1} \sum_{\substack{|\ul{i}|=l\\ |\ul{j}|=n-l}} \sum_{|\ul{k}|=l} c(\ul{j}\cdot \ul{k}) A(\ul{k},\ul{i}) X_{\ul{i}} X_{\ul{j}}\\
			&=\frac{1}{n}\sum_{|\ul{i}|=n} \left\{ \sum_{l=0}^{n-1}\sum_{|\ul{k}|=l} c\left( (i_{l+1},\ldots, i_{n})\cdot \ul{k}\right) A\left(\ul{k}, (i_1,\ldots, i_l) \right) \right\} X_{\ul{i}} =: \frac{1}{n} \sum_{|\ul{i}|=n} b(\ul{i}) X_{\ul{i}}.
	\end{align*}
So if we let $Q=\sum_{|\ul{i}|=n} b(\ul{i}) X_{\ul{i}}$, then $\mathscr{S}\Pi P=\Sigma Q$ and using Equation (\ref{cyclic_derivative_of_cyclically_symmetric}) we obtain
	\begin{align*}
		\D_t\mathscr{S}\Pi P=\sum_{|\ul{i}|=n} \alpha_{t i_n} b(\ul{i}) X_{i_1}\cdots X_{i_{n-1}},
	\end{align*}
for each $t\in\{1,\ldots, N\}$. It is then a straightforward computation to show that the above equals $\D_t P$. The case for general $P\in\mathscr{P}_\varphi^{(R,\sigma)}$ then follows from linearity.
\end{proof}


\subsection{Notation}

We use the same notation as in \cite{SG11}, adjusted slightly to accommodate our new operators. For $Q\in M_{N}(\mathscr{P}\otimes \mathscr{P}^{op})$ we write
	\begin{align*}
		\text{Tr}(Q)&=\sum_{i=1}^N [Q]_{ii} \in\mathscr{P}\otimes\mathscr{P}^{op},\\
		\text{Tr}_{A}(Q)&=\text{Tr}(A\# Q) = \sum_{i,j=1}^N [A]_{ij} [Q]_{ji} \in \mathscr{P}\otimes\mathscr{P}^{op},\\
		\text{Tr}_{A^{-1}}(Q) &= \text{Tr}(A^{-1}\# Q) = \sum_{i,j=1}^N [A^{-1}]_{ij} [Q]_{ji}\in\mathscr{P}\otimes\mathscr{P}^{op}.
	\end{align*}
By Corollary \ref{difference_quotient_adjoint} $\mathscr{P}\otimes\mathscr{P}^{op}\subset \dom{\partial_j^*}$, so we note
	\begin{align*}
		\J_\sigma^*(Q)=\left( \sum_i \partial_i^*([Q]_{ji})\right)_{j=1}^N \in L^2(\mathscr{P},\varphi)^N.
	\end{align*}
where $\J_\sigma$ is viewed as a densely defined operator from $L^2(\mathscr{P}^N,\varphi)$ to $L^2(M_N(\mathscr{P}\otimes \mathscr{P}^{op}), \varphi\otimes\varphi^{op}\otimes\text{Tr})$ and the above is its adjoint.


\subsection{Transport and invertible power series}

Let $(\M,\psi)$ be a von Neumann algebra with state $\psi$ and let $T_1,\ldots, T_N\in \M$ be self-adjoint elements which generate $\M$. Then, after \cite{SG11}, $\M$ can be thought of as a completion of the algebra $\C\<T_1,\ldots, T_N\>$, and $\psi$ induces a linear functional $\psi_T$ on $\C\<t_1,\ldots, t_N\>$, the non-commutative polynomials in abstract indeterminates $t_1,\ldots, t_N$, via $\psi_T(t_{k_1}\cdots t_{k_n})=\psi(T_{k_1}\cdots T_{k_n})$, $k_1,\ldots, k_n\in\{1,\ldots, N\}$. $\psi_T$ is called the \textit{non-commutative law of} $T_1,\ldots, T_N$ and we write $W^*(\psi_T)\cong \M$. Let $S_1,\ldots, S_N\in \mc{N}$ be self-adjoint elements generating another von Neumann algebra $\mc{N}$ with state $\chi$ and let $\chi_S$ be their law so that $W^*(\chi_S)\cong \mc{N}$. 

\begin{defi}\label{transport}
By \textit{transport} from $\psi_T$ to $\chi_S$ we mean an $N$-tuple of self-adjoint elements $Y_1,\ldots, Y_N\in \M$ having the same law as $S_1,\ldots,S_N$:
	\begin{align*}
		\chi(P(S_1,\ldots, S_N)) = \psi(P(Y_1,\ldots, Y_N)),
	\end{align*}
for all non-commutative polynomials $P$ in $N$ variables. If such an $N$-tuple exists then there is a state-preserving embedding $\mc{N}\cong W^*(Y)\subset \M$.
\end{defi}

Let $M=W^*(X_1,\ldots, X_N)$ be as before. Suppose $L$ is a von Neumann algebra generated by self-adjoint $Z_1,\ldots, Z_N$ with state $\psi$ and there exists transport $Y=(Y_1,\ldots, Y_N)$  from $\varphi_X$ to $\psi_Z$ such that $Y=G(X)\in (\mathscr{P}^{(R)})^N$. That is, $Y_j=G_j(X)$ is a power series in terms of $X_1,\ldots, X_N$. If we can invert this power series so that $X=H(Y)$, then $H(Z)\in L^N$ is transport from $\psi_Z$ to $\varphi_X$. It would then follow that we have a state-preserving isomorphism $L\cong M$. The following lemma, which is presented as Corollary 2.4 in \cite{SG11}, shows that such inverses can be found.

\begin{lem}\label{invertible_power_series}
Let $R<S$ and consider the equation $Y=X+f(X)$ with $f\in (\mathscr{P}^{(S)})^N$ and $\|Y\|_{R}<S$. Then there exists a constant $C>0$, depending only on $S$ and $R$ so that whenever $\|f\|_S<C$, then there exists $H\in (\mathscr{P}^{(R)})^N$ so that $X=H(Y)$.
\end{lem}
\begin{proof}
Fix $S'\in (\|Y\|_R, S)$ and define
	\begin{align*}
		C(S')=\|f\|_S \max_{k\geq 0} k(S')^{k-1}S^{-k}.
	\end{align*}
Since $\|f\|_S<C$, we can choose $C$ sufficiently small so that $C(S')<1$ and
	\begin{align*}
		\|Y\|_{R}+\frac{C}{1-C(S')}\leq S'.
	\end{align*}
We define a sequence of $N$-tuples
	\begin{align*}
		H_k(Y)= Y - f(H_{k-1}(Y)) \in (\mathscr{P}^{(R)})^N,
	\end{align*}
with $H_0(Y)=Y$. Denote the component functions of $H_k(Y)$ by $[H_k(Y)]_j$, $j\in\{1,\ldots, N\}$. We claim that $\|H_k(Y)\|_R\leq S'$ for all $k\geq 0$. This clearly holds for $H_0(Y)$, so assume it holds for $1,\ldots, k-1$. Fix $j\in\{1,\ldots, N\}$ and suppose
	\begin{align*}
		f_j(X_1,\ldots, X_N)=\sum_{|\ul{i}|\geq 0} c(\ul{i}) X_{\ul{i}}.
	\end{align*}
Then for any $0\leq l\leq k-1$ we have
	\begin{align*}
		\|[H_{l+1}(Y)]_j - [H_l(Y)]_j\|_R &= \| f_j(H_l(Y)) - f_j(H_{l-1}(Y))\|_R \\
			&\leq \sum_{n=0}^\infty \sum_{|\ul{i}|=n} |c(\ul{i})| \sum_{u=1}^n \| H_l(Y)\|_R^{u-1} \| [H_l(Y)]_{i_u} - [H_{l-1}(Y)]_{i_u}\|_R \| H_{l-1}(Y)\|_R^{n-u}\\
			&\leq\|H_l(Y) - H_{l-1}(Y)\|_R\sum_{n=0}^\infty n(S')^{n-1}S^{-n}\sum_{|\ul{i}|=n} |c(\ul{i})| S^n\\
			&\leq \|H_l(Y)- H_{l-1}(Y)\|_R C(S').
	\end{align*}
As $j$ was arbitrary, we obtain through iteration
	\begin{align*}
		\|H_{l+1}(Y) - H_l(Y)\|_R \leq \|H_1(Y) - H_0(Y)\|_R C(S')^l = \| f(Y)\|_R C(S')^l \leq C C(S')^l,
	\end{align*}
and thus
	\begin{align*}
		\| H_k(Y)\|_R &\leq \|H_0(Y)\|_R + \|H_k(Y)- H_0(Y)\|_R \leq \|Y\|_R + \sum_{l=0}^{k-1} \|H_{l+1}(Y) - H_l(Y)\|_R \\
				&\leq \|Y\|_R + \sum_{l=0}^{k-1} C C(S')^l \leq \|Y\|_R + \frac{C}{1 - C(S')} \leq S',
	\end{align*}
by our assumption on $C$. So the claim holds and by induction we have the bound $\|H_k(Y)\|_R\leq S'$ for all $k\geq 0$. Moreover, by a standard argument we can see that $\{H_k(Y)\}_{k\geq 0}$ is Cauchy and so converges to some $H(Y)\in (\mathscr{P}^{(R)})^N$ which satisfies $H(Y)=Y-f(H(Y))$ and $\|H(Y)\|_R\leq S'$. Since $\|X\|_R = R\leq S'$ and $Y=X+f(X)$ we obtain (via the same argument as above)
	\begin{align*}
		\| X - H(Y)\|_R =\| Y - f(X) - Y + f(H(Y)) \|_R \leq \| X - H(Y)\|_R C(S').
	\end{align*}
As $C(S')<1$, this implies $H(Y)=X$.
\end{proof}


\subsection{Monotonicity of transport.}

We introduce a definition for what it means for transport to be ``monotone.\rq{}\rq{} Note that in the tracial case ($A=1$) this coincides with Definition 2.1 in \cite{SG11}.
\begin{defi}
We say that transport from $\varphi_X$ to $\psi_Z$ via the $N$-tuple $Y=(Y_1,\ldots, Y_N)$ is monotone if $Y=\D G$ for some $G\in \mathscr{P}^{(R)}$, $R\geq4\sqrt{\|A\|}$, such that $(\sigma_\frac{i}{2}\otimes 1)(\J_\sigma \D G)\geq 0$ as an operator on $L^2(\mathscr{P}\otimes\mathscr{P}^{op},\varphi\otimes\varphi^{op})^N$.
\end{defi}

Suppose $(\M,\psi)$ is a von Neumann algebra with a faithful normal state $\psi$. Let $\H_\psi=L^2(\M,\psi,\xi_0)$ be the Hilbert space obtained via the GNS construction with a cyclic vector implementing $\psi$. Let $S_\psi$ be the Tomita conjugation for the left Hilbert algebra $\M\xi_0$, and let $\Delta_\psi$ and $J_\psi$ be the modular operator and conjugation (respectively). Recall (\emph{cf.} \cite{T2}, Chapter IX, \textsection 1) that there is a canonical pointed convex cone
	\begin{align*}
		\mathfrak{P}_\psi=\overline{\{ \Delta_\psi^{1/4}x\xi_0\colon x\in\M_+\}}^{\|\cdot\|_{\psi}},
	\end{align*}
which is self-dual in the sense that if $\eta\in\H_\psi$ satisfies $\<\eta,\xi\>_\psi\geq 0$ for all $\xi\in \mf{P}_\psi$ then $\eta\in \mf{P}_\psi$. The embedding
	\begin{align*}
		x\mapsto \Delta_\psi^\frac{1}{4} x \xi_0
	\end{align*}
of $\M$ into $\H_\psi$ then has the benefit of sending positive elements in $\M$ into $\mf{P}_\psi$.\par
In particular, if $\M=M_N(M\bar{\otimes} M^{op})$ and $\psi=\varphi\otimes\varphi^{op}\otimes \text{Tr}_A$ then
	\begin{align*}
		\Delta_\psi^\frac{1}{4} q \xi_0= (\sigma_{-\frac{i}{4}}\otimes \sigma_\frac{i}{4})(A^\frac{1}{4}\# q\# A^{-\frac{1}{4}})\xi_0.
	\end{align*}
We shall see in Lemma \ref{change_of_variables}.(iv) that if $G\in \mathscr{P}_\varphi^{(R,\sigma)}$ then $A^s\#\J_\sigma \D G \# A^{-s}=(\sigma_{-is}\otimes\sigma_{-is})(\J_\sigma \D G)$. Hence if $Y=\D G$ for such $G$, then $(\sigma_\frac{i}{2}\otimes 1)(\J_\sigma Y)$ embeds into $\H_\psi$ as
	\begin{align*}
		(\sigma_{-\frac{i}{4}}\otimes \sigma_\frac{i}{4})(A^\frac{1}{4}\# (\sigma_\frac{i}{2}\otimes 1)(\J_\sigma Y)\# A^{-\frac{1}{4}})\xi_0 = \J_\sigma Y \xi_0.
	\end{align*}


\subsection{The Schwinger-Dyson equation and free Gibbs state.}\label{S-D_section}
Our construction of the transport $Y$ will exploit the condition that $\varphi_Y$ satisfies the so-called Schwinger-Dyson equation:

\begin{defi}
Given $V\in\mathscr{P}^{(R,\sigma)}_{c.s.}$, we say a linear functional $\varphi_V$ on $\mathscr{P}$ satisfies \textit{the Schwinger-Dyson equation with potential $V$} if
	\begin{align}\label{Schwinger-Dyson_1}
		\varphi_V( \D(V) \# P) = \varphi_V\otimes\varphi_V^{op}(\text{Tr}(\J_\sigma P)),\qquad\forall P\in\mathscr{P}.
	\end{align}
The law $\varphi_V$ is called the \textit{free Gibbs state with potential $V$}.
\end{defi}

Note that when $\J_\sigma$ is viewed as a densely defined operator from $L^2(\mathscr{P}^N,\varphi)$ to $L^2(M_N(\mathscr{P}\otimes\mathscr{P}^{op}),\varphi\otimes\varphi^{op}\otimes\text{Tr})$, (\ref{Schwinger-Dyson_1}) is equivalent to
	\begin{align}\label{Schwinger-Dyson_2}
		\J_\sigma^*(1)=\D V,
	\end{align}
where $1\in M_N(\mathscr{P}\otimes\mathscr{P}^{op})$ is the identity matrix.\par
Consider the potential
	\begin{align}\label{quadratic_potential}
		V_0=\frac{1}{2}\sum_{j,k=1}^N \left[\frac{1+A}{2}\right]_{jk} X_k X_j.
	\end{align}
Then
	\begin{align*}
		\rho(V_0)=\frac{1}{2}\sum_{j,k=1}^N  \left[\frac{1+A}{2}\right]_{jk} \sigma_{-i}(X_j)X_k=\frac{1}{2}\sum_{j,k,l=1}^N  \left[\frac{1+A^{-1}}{2}\right]_{kj}[A]_{jl} X_l X_k = \frac{1}{2} \sum_{k,l=1}^N  \left[\frac{1+A}{2}\right]_{kl} X_l X_k= V_0,
	\end{align*}
and hence $V_0\in\mathscr{P}_{c.s.}^{(R,\sigma)}$. Also,
	\begin{align*}
		\D_l(V_0)&=\frac{1}{2}\sum_{i,j}  \left[\frac{1+A}{2}\right]_{ij} \left(\alpha_{lj} \sigma_{-i}(X_i) + \alpha_{li} X_j\right)\\
			&= \frac{1}{2} \sum_{i,j,k=1}^N  \left[\frac{2}{1+A}\right]_{lj}\left[\frac{1+A^{-1}}{2}\right]_{ji}[A]_{ik} X_k + \frac{1}{2}\sum_{i,j=1}^N  \left[\frac{2}{1+A}\right]_{li} \left[\frac{1+A}{2}\right]_{ij} X_j \\
			&=\frac{1}{2}X_l+\frac{1}{2}X_l=X_l,
	\end{align*}
so that $\D V_0=X$. Using $A=A^*$ it is also easy to see that $V_0^*=V_0$.\par
Now, (\ref{Schwinger-Dyson_2}) for $V=V_0$ states $\J_\sigma^*(1)=X$, or $\partial_j^*(1\otimes 1)=X_j$ for each $j=1,\ldots, N$, where the the adjoint is with respect to $\varphi_{V_0}$. However, from (\ref{adjoint_of_1}) we know this relation holds when the adjoint of $\partial_j=\partial_j^{(0)}$ is taken with respect to the free quasi-free state $\varphi_0$. We therefore immediately obtain the following result.

\begin{thm}\label{free_Gibbs_is_free_quasi-free}
The free Gibbs state with potential $V_0$ is the free quasi-free state $\varphi_0$ on $M_0=\Gamma(\H_\R,U_t)''$.
\end{thm}\par

It is clear that the $\varphi_{V_0}$ is unique since (\ref{Schwinger-Dyson_1}) for $V=V_0$ recursively defines $\varphi_{V_0}$ for all monomials. However, even for small perturbations (in the $\|\cdot\|_{R,\sigma}$-norm) $V=V_0+W$  of $V_0$ the free Gibbs state with potential $V$ is unique, which we demonstrate below. Consequently, if $\psi_Z$ satisfies the Schwinger-Dyson equation for a $V$, then to find transport from $\varphi_X$ it suffices to produce $Y\in M^N$ whose law $\varphi_Y$ (determined by $\varphi$) satisfies the Schwinger-Dyson equation with the same potential $V$. The proof of uniqueness presented here differs from the proof of Theorem 2.1 in \cite{GMS06} only in the differential operators considered.

\begin{thm}\label{Gibbs_state_unique}
Fix $R\geq 4\sqrt{\|A\|}$. Let $V=V_0+W\in \mathscr{P}_{c.s.}^{(R,\sigma)}$. Then for sufficiently small $\|W\|_{R,\sigma}$, the Schwinger-Dyson equation has a unique solution amongst states that satisfy
	\begin{align}\label{phi_monomial0}
		|\varphi(X_{\ul{j}})| \leq 3^{|\ul{j}|}
	\end{align}
for any multi-index $\ul{j}$.
\end{thm}
\begin{proof}
Suppose two states $\varphi$ and $\varphi'$ both solve the Schwinger-Dyson equation with potential $V$. Then $\varphi(1)=\varphi'(1)=1$ and hence they agree on $\pi_0(\mathscr{P})$. Fix $l\geq 1$ and a monomial $P\in \pi_{l-1}(\mathscr{P})$. Then we have
	\begin{align*}
		(\varphi-\varphi')(X_i P)= ((\varphi-\varphi')\otimes\varphi)(\partial_i P) + (\varphi'\otimes (\varphi-\varphi'))(\partial_i P) - (\varphi-\varphi')(\D_i W P).
	\end{align*}
(Note that for $l=1$ the first two terms disappear). Define
	\begin{equation*}
		\Delta_l(\varphi,\varphi'):=\max_{|\ul{j}|=l} |(\varphi-\varphi')(X_{\ul{j}})|.
	\end{equation*}
In particular $\Delta_0(\varphi,\varphi')=0$. Write $\D W=\sum_{\ul{i}} c(\ul{j}) X_{\ul{j}}$. Then we have
	\begin{align*}
		\Delta_l(\varphi,\varphi') \leq 2 \sum_{k=0}^{l-2} \Delta_k(\varphi,\varphi') 3^{l-2-k} + \sum_{p=0}^\infty \sum_{|\ul{j}|=p} |c(\ul{j})| \Delta_{p+l-1}(\varphi,\varphi').
	\end{align*}
For $\gamma>0$, set
	\begin{equation*}
		d_\gamma(\varphi,\varphi') = \sum_{l=1}^\infty \gamma^l \Delta_l(\varphi,\varphi').
	\end{equation*}
Since (\ref{phi_monomial0}) implies $\Delta_l(\varphi,\varphi')\leq 2 (3)^l$, we see that $d_\gamma(\varphi,\varphi')<\infty$ so long as $\gamma< \frac{1}{3}$. In the above equality we multiply both sides of the equation by $\gamma^l$ and then sum over $l\geq 1$ to obtain
	\begin{align*}
		d_\gamma(\varphi,\varphi') &\leq 2\sum_{l=2}^\infty \gamma^l\sum_{k=0}^{l-2} \Delta_k(\varphi,\varphi') 3^{l-2-k} + \sum_{l=1}^\infty \gamma^l \sum_{p=0}^\infty \sum_{|\ul{j}|=p} |c(\ul{j})| \Delta_{p+l-1}(\varphi,\varphi')\\
				&=2\gamma^2\sum_{k=0}^\infty \gamma^k\Delta_k(\varphi,\varphi')\sum_{l=k+2}^\infty \gamma^{l-2-k}3^{l-2-k} + \sum_{p=0}^\infty \sum_{|\ul{j}|=p} |c(\ul{j})| \gamma^{-p+1}\sum_{l=1}^\infty \gamma^{p+l-1}\Delta_{p+l-1}(\varphi,\varphi')\\
				&\leq \frac{2\gamma^2}{1-3\gamma} d_\gamma(\varphi,\varphi') + \gamma\sum_{p=0}^\infty \sum_{|\ul{j}|=p} |c(\ul{j})| \gamma^{-p}d_\gamma(\varphi,\varphi').
	\end{align*}
Let $\gamma=\frac{28}{25 R}$. Then $\gamma^{-1}<R$ and $R>4$ implies $3\gamma<\frac{21}{25}$. Hence
	\begin{align*}
		d_\gamma(\varphi,\varphi')&\leq d_\gamma(\varphi,\varphi')\left(\frac{49}{50} + \frac{7}{25}\|\D W\|_{\frac{25R}{28}}\right).
	\end{align*}
Recall from Lemma \ref{cyclic_derivative_bounded}, that $\|\D W\|_{\frac{25R}{28}} \leq C\|W\|_{R,\sigma}$ where the constant only depends on the ratio $\frac{R}{25R/28}=\frac{28}{25}$. Thus if $\|W\|_{R,\sigma} <\frac{1}{14C}$ then
	\begin{align*}
		d_\gamma(\varphi,\varphi') \leq c d_\gamma(\varphi,\varphi')\qquad \text{with }c<1,
	\end{align*}
implying $d_\gamma(\varphi,\varphi')=0$ and hence $\Delta_l(\varphi,\varphi')=0$ for all $l\geq 1$.
\end{proof}

This theorem implies that if the law $\psi_Z$ of $Z=(Z_1,\ldots, Z_N)\subset (L,\psi)$ and the law $\varphi_Y$ of $Y=(Y_1,\ldots, Y_N)\subset (M,\varphi)$ both solve the Schwinger-Dyson equation with potential $V$, then $W^*(Z_1,\ldots, Z_N)\cong W^*(Y_1,\ldots, Y_N)\cong W^*(\varphi_V)$. In particular, $W^*(\varphi_V)$ is well-defined.


\subsection{Outline of the paper.}

The general outline for the paper is as follows: we begin in Section \ref{construction} by fixing $q=0$ and a potential $V=V_0+W\in \mathscr{P}^{(R,\sigma)}_{c.s.}$ and assuming there exists $Y=(Y_1,\ldots, Y_N)\in (\mathscr{P}^{(R)})^N$ whose law (induced by $\varphi$) satisfies the Schwinger-Dyson equation with potential $V$. Several equivalent versions of this equation will be derived in Sections \ref{equivalent_S-D} and \ref{identities} until we arrive at a final version for which a fixed point argument can be applied. Several technical estimates will be produced in Section \ref{technical_estimates} for the purposes of this fixed point argument so that in Section \ref{existence_of_g}, given certain assumptions regarding $V$, we can assert the existence of $Y$. Having obtained the desired transport, we then use Lemma \ref{invertible_power_series} to refine the transport into an isomorphism in Section \ref{summary}. Finally, in Section \ref{application} we present the main application to q-deformed Araki-Woods algebras.


\section{Construction of the Non-tracial monotone transport map}\label{construction}

For all this section, we consider only $q=0$ and maintain the same notational simplifications as above ($M=M_0$, $\varphi=\varphi_0$, $X_j^{(0)}=X_j$, and $\sigma^{\varphi_0}_{z}=\sigma_z$).  Recall that $V_0$ is defined by (\ref{quadratic_potential}) and that by Theorem \ref{free_Gibbs_is_free_quasi-free}, $\varphi$ is the free Gibbs state with potential $V_0$. Our goal is to construct $Y=(Y_1,\ldots, Y_N)\in (\mathscr{P}^{(R)})^N$ whose law with respect to $\varphi$ is the free Gibbs state with potential $V=V_0+W \in \mathscr{P}^{(R,\sigma)}_{c.s.}$, for $\|W\|_{R,\sigma}$ sufficiently small.\par
We will need differential operators $\partial_j$, $\J_\sigma$, $\J$, and $\D$ for $Y$ as well as $X$, so we adopt the following convention: differential operators which have no indices or have a numeric index refer to differentiation with respect to $X_1,\ldots, X_N$. Operators involving differentiation with respect to $Y_1,\ldots, Y_N$ shall be labeled $\partial_{Y_j}$, $\D_{Y_j}$, etc. We define these latter operators using the comments at the end of Subsection \ref{derivations_on_M_q}; that is, $\partial_{Y_j}(Y_{k_1}\cdots Y_{k_n})$ is computed exactly as one would compute $\partial_j(X_{k_1}\cdots X_{k_n})$ and exchanging $X_j$'s for $Y_j$'s in the end.\par
Assuming the law $\varphi_Y$ of $Y=(Y_1,\ldots, Y_N)$ is the free Gibbs state with potential $V$ and $1\otimes 1\in\dom{\partial_{Y_j}^*}$, (\ref{Schwinger-Dyson_2}) implies
	\begin{equation*}
		\partial_{Y_j}^*(1\otimes 1)=\D_{Y_j} (V_0(Y)+W(Y))=Y_j+\D_{Y_j}(W(Y)),
	\end{equation*}
or, in short
	\begin{align}\label{Schwinger-Dyson_Y}
		(\J_\sigma)_Y^*(1)=Y+(\D W)(Y).
	\end{align}
It will turn out that $Y=X+f$ for some $f=\D g$ and $g\in\mathscr{P}^{(R,\sigma)}_{c.s.}$, and so we start by considering the implications of assuming $Y$ is of this form.


\subsection{Change of variables formula.}

\begin{lem}\label{change_of_variables}
Assume $Y$ is such that $\J_\sigma Y=(\partial_{X_j} Y_i)_{ij}\in M_N(M\bar{\otimes}M^{op})$ is bounded and invertible.
	\begin{enumerate}
	\item[(i)] Define
		\begin{equation*}
			\hat{\partial}_j(P)=\sum_{i=1}^N \partial_{X_i}(P)\#\left[ \J_\sigma Y^{-1}\# \J_\sigma X\right]_{ij},
		\end{equation*}
	then $\hat{\partial}_j=\partial_{Y_j}$.
	
	\item[(ii)] $\partial_{Y_j}^*(1\otimes 1)=\sum_l \partial_{X_l}^*\circ\hat{\sigma}_{-i}\left(\left[ \J_\sigma Y^{-1}\# \J_\sigma X\right]_{lj}^*\right)$. Hence
		\begin{equation}\label{change_of_variables_formula}
			(\J_\sigma)_Y^*(1)=\J_\sigma^*\left(\hat{\sigma}_{-i}\left( \J_\sigma X\# \left(\J_\sigma Y^{-1}\right)^*\right)\right),
		\end{equation}
	where $1\in M_N(M\bar{\otimes}M^{op})$ is the identity matrix.

	\item[(iii)] Assume in addition that $Y_j=\D_jG$ for some $G\in\mathscr{P}^{(R,\sigma)}_\varphi$ with $G=G^*$. Then $(\J_\sigma Y)^*=(\sigma_i\otimes 1)(\J_\sigma Y)$ and $\left(\J_\sigma Y^{-1}\right)^*=(\sigma_i\otimes 1)(\J_\sigma Y^{-1})$ and hence Equation (\ref{change_of_variables_formula}) becomes
		\begin{equation}\label{change_of_variables_formula_2}
			(\J_\sigma)_Y^*(1)=\J_\sigma^*\circ (1\otimes \sigma_i)\left( \J_\sigma X \# \J_\sigma Y^{-1}\right).
		\end{equation}
	
	\item[(iv)] For $G\in\mathscr{P}^{(R,\sigma)}_{\varphi}$,
		\begin{align*}
			(\sigma_{-is}\otimes\sigma_{-is})(\J_\sigma\D G) = A^s\# \J_\sigma\D G \# A^{-s},\qquad \forall s\in\R.
		\end{align*}
	\end{enumerate}
\end{lem}
\begin{proof}
Let $Q=\J_\sigma Y$.
	\begin{enumerate}
	\item[(i):] We verify
		\begin{align*}
			\hat{\partial}_iY_k&=\sum_{i=1}^N \partial_{X_i}Y_k\#[Q^{-1}\# \J_\sigma X]_{ij}=\sum_{i=1}^N Q_{ki}\#[Q^{-1}\#\J_\sigma X]_{ij}\\
				&=[Q\#Q^{-1}\#\J_\sigma X]_{kj}=[\J_\sigma X]_{kj}=\partial_{X_j}X_k=\alpha_{kj}1\otimes 1=\partial_{Y_j}Y_k.
		\end{align*}
		
	\item[(ii):] We compute
		\begin{align*}
			\<\partial_{Y_j}^*(1\otimes 1), X_{k_1}\cdots X_{k_p}\>_\varphi &= \<1\otimes 1, \partial_{Y_j}(X_{k_1}\cdots X_{k_p})\>_{\varphi\otimes\varphi^{op}}\\
							&=\sum_{l=1}^N \<1\otimes 1, \partial_{X_l}(X_{k_1}\cdots X_{k_p})\#[Q^{-1}\#\J_\sigma X]_{lj}\>_{\varphi\otimes\varphi^{op}}\\
							&=\sum_{l=1}^N \<\hat{\sigma}_{-i}\left([Q^{-1}\# \J_\sigma X]_{lj}^*\right), \partial_{X_l}(X_{k_1}\cdots X_{k_p})\>_{\varphi\otimes\varphi^{op}}\\
							&=\< \sum_{l=1}^N \partial_{X_l}^*\circ \hat{\sigma}_{-i}\left([Q^{-1}\# \J_\sigma X]_{lj}^*\right),X_{k_1}\cdots X_{k_p}\>_\varphi.
		\end{align*}
		 Recalling that $\J_\sigma X=\J_\sigma X^*$, the definition of $\J_\sigma^*$ implies (\ref{change_of_variables}).
		 
	\item[(iii):] Suppose $G=G^*\in\mathscr{P}^{(R,\sigma)}_{\varphi}$. Then 
		\begin{equation*}
			[(\J_\sigma\D G)^*]_{jk}=[\J_\sigma\D G]_{kj}^*=\partial_j\circ\D_k(G)^*=\tilde{\partial}_j\circ\bar{\D}_k(G^*)=\tilde{\partial}_j\circ\bar{\D}_k(G).
		\end{equation*}		
	A computation on monomials shows that
		\begin{align*}
			\left[\tilde{\partial}_j\circ\bar{\D}_k-(\sigma_i\otimes 1)\circ\partial_k\circ\D_j\right]\circ\sigma_{it}(P)=H_t(P)-H_{t-1}(P),
		\end{align*}
	where
		\begin{align*}
			H_t(P)=\sum_{a,b=1}^N \sum_{P=AX_bB X_a C}\left[\frac{2A^t}{1+A^{-1}}\right]_{ka}\left[\frac{2A^{t+1}}{1+A}\right]_{jb} \sigma_{i(t+1)}(B)\otimes \sigma_{it}(C)\sigma_{i(t+1)}(A).
		\end{align*}
	We claim that $H_{t-1}(P)=H_t(\sigma_{-i}(P))$. Indeed
	\begin{align*}
		H_{t-1}(P)=&\sum_{a,b=1}^N \sum_{P=AX_bBX_aC} \sum_{p,q=1}^N \left[\frac{2 A^t}{1+A^{-1}}\right]_{kp}[A^{-1}]_{pa}\left[\frac{2A^{t+1}}{1+A}\right]_{jq}[A^{-1}]_{qb}\\
				&\times \sigma_{i(t+1)}( \sigma_{-i}(B))\otimes \sigma_{it}(\sigma_{-i}(C)) \sigma_{i(t+1)}(\sigma_{-i}(A))\\
 =& \sum_{a,b,p,q=1}^N [A^{-1}]_{pa}[A^{-1}]_{qb}\sum_{P=\sigma_i(A)X_l\sigma_i(B)X_k\sigma_i(C)} \left[\frac{2 A^t}{1+A^{-1}}\right]_{kp} \left[\frac{2A^{t+1}}{1+A}\right]_{jq}\\
 				&\times \sigma_{it(t+1)}(B)\otimes \sigma_{it}(C)\sigma_{i(t+1)}(A).
	\end{align*}
Note that $\sigma_i(X_p)=\sum_{a=1}^N [A^{-1}]_{pa} X_a$ and $\sigma_i(X_q)=\sum_{b=1}^N [A^{-1}]_{qb} X_b$. So continuing the above computation we have
	\begin{align*}
		H_{t-1}(P)&=\sum_{p,q=1}^N \sum_{P=\sigma_i(AX_pBX_qC)} \left[\frac{2 A^t}{1+A^{-1}}\right]_{kp} \left[\frac{2A^{t+1}}{1+A}\right]_{jq} \sigma_{it(t+1)}(B)\otimes \sigma_{it}(C)\sigma_{i(t+1)}(A)\\
				&=\sum_{p,q=1}^N \sum_{\sigma_{-i}(P)=AX_pBX_qC} \left[\frac{2 A^t}{1+A^{-1}}\right]_{kp} \left[\frac{2A^{t+1}}{1+A}\right]_{jq} \sigma_{it(t+1)}(B)\otimes \sigma_{it}(C)\sigma_{i(t+1)}(A)\\
				&=H_t(\sigma_{-i}(P)).
	\end{align*}
Thus from $G=\sigma_{-i}(G)$ we obtain
	\begin{align*}
		\left[\tilde{\partial}_j\circ\bar{\D}_k-(\sigma_i\otimes 1)\circ\partial_k\circ\D_j\right]\circ\sigma_{it}(G)=0,
	\end{align*}
and hence
	\begin{align*}
		(\J_\sigma \D G)^*=(\sigma_i\otimes 1)(\J_\sigma\D G).
	\end{align*}
	Now, if $Y=\D G$ for such $G$ then $\J_\sigma Y^*=(\sigma_i\otimes 1)(\J_\sigma Y)$ and $\J_\sigma Y^{-1}$ satisfies this formula as well because $\sigma_i\otimes 1$ is a homomorphism. That Equation (\ref{change_of_variables_formula}) becomes Equation (\ref{change_of_variables_formula_2}) is then clear after realizing $\J_\sigma X=(\sigma_{it} \otimes \sigma_{is})(\J_\sigma X)$  for all $t,s\in \R$ (since the entries of $\J_\sigma X$ are merely scalars multiplied with $1\otimes 1$).
	
	\item[(iv):] Recall
		\begin{align*}
			(\sigma_{-it}\otimes\sigma_{-it})\circ\partial_j\circ\sigma_{it} = \sum_{k=1}^N [A^{-t}]_{kj} \partial_k.
		\end{align*}
	Also,
		\begin{align*}
			\bar{\partial}_j=(\sigma_i\otimes\sigma_i)\circ\partial_j\circ\sigma_{-i},
		\end{align*}
	so that
		\begin{align*}
			(\sigma_{-it}\otimes\sigma_{-it})\circ\bar{\partial}_j\circ\sigma_{it}=\sum_{k=1}^N [A^{-t}]_{kj}\bar{\partial}_k.
		\end{align*}
	Using these identities we have
		\begin{align*}
			(\sigma_{-is}\otimes\sigma_{-is})\circ \partial_k\circ \D_j &= (\sigma_{-is}\otimes\sigma_{-is})\circ\partial_k\circ m \circ\diamond\circ(1\otimes\sigma_{-i})\circ\bar{\partial}_j\\
							&=\sum_{a,b=1}^N [A^{-s}]_{ak}[A^{-s}]_{bj} \partial_a\circ m\circ\diamond\circ(1\otimes\sigma_{-i})\circ \circ\bar{\partial}_b\circ\sigma_{-is}\\
							&=\sum_{a,b,=1}^N [A^{s}]_{jb} [A^{-s}]_{ak}  \partial_a\circ\D_b\circ\sigma_{-is}.
		\end{align*}
	Hence for $G\in\mathscr{P}_\varphi^{(R,\sigma)}$ we have
		\begin{align*}
			[(\sigma_{-is}\otimes\sigma_{-is})(\J_\sigma \D G)]_{jk}&=(\sigma_{-is}\otimes\sigma_{-is})\circ \partial_k\circ \D_j(G)\\
				&=\sum_{a,b,=1}^N [A^s]_{jb} [A^{-s}]_{ak} \partial_a\circ\D_b\circ\sigma_{-is}(G)\\
				&=\sum_{a,b,=1}^N [A^s]_{jb} [A^{-s}]_{ak} [\J_\sigma\D G]_{ba} = [A^s\# \J_\sigma\D G\# A^{-s}]_{jk},
		\end{align*}
	for each $j,k=1,\ldots, N$.\qedhere
	\end{enumerate}
\end{proof}

\begin{cor}
Assume $g=g^*\in\mathscr{P}^{(R,\sigma)}_{\varphi}$ and put $G=V_0+g$ and $f_j=\D_j g$. Let $Y_j=X_j+f_j$ so that $Y=\D G$. Define $B=\J_\sigma f\# \J_\sigma X^{-1}$ and assume $1+B$ is invertible. Then Equation (\ref{Schwinger-Dyson_Y}) is equivalent to the equation
	\begin{equation}\label{S-D_Y_2}
		\J_\sigma^*\circ(1\otimes\sigma_i)\left( \frac{1}{1+B}\right) = X+f+(\D W)(X+f).
	\end{equation}
\end{cor}
\begin{proof}
Since $\J_\sigma X+ \J_\sigma f=(1+B)\#\J_\sigma X$, $\J_\sigma Y=\J_\sigma X+\J_\sigma f$ is invertible as a consequence of $1+B$ and $\J_\sigma X$ both being invertible. Then upon noting that
	\begin{equation*}
		\J_\sigma X\# (\J_\sigma X+ \J_\sigma f)^{-1}= 1\# (1+B)^{-1}=\frac{1}{1+B},
	\end{equation*}
the corollary follows immediately from Lemma \ref{change_of_variables}, (ii) and (iii).
\end{proof}


\subsection{An equivalent form of Equation (\ref{S-D_Y_2})}\label{equivalent_S-D}

\begin{lem}\label{K_1}
Assume that the map $\xi\mapsto (1+B)\# \xi$ is invertible on $\left(\mathscr{P}^{(R)}\right)^N$, and that $f=\D g$ for some self-adjoint $g\in\mathscr{P}_\varphi^{(R,\sigma)}$. Let
	\begin{equation*}
		K(f)=-\J_\sigma^*\circ(1\otimes \sigma_i)(B)- f.
	\end{equation*}
Then Equation (\ref{S-D_Y_2}) is equivalent to
	\begin{equation*}
		K(f)=\D(W(X+f))+\left[B\# f+ B\#\J_\sigma^*\circ(1\otimes\sigma_i)\left(\frac{B}{1+B}\right)-\J_\sigma^*\circ(1\otimes\sigma_i)\left(\frac{B^2}{1+B}\right)\right].
	\end{equation*}
\end{lem}
\begin{proof}
Using $\frac{1}{1+x}=1-\frac{x}{1+x}$ and $\J_\sigma^*\circ(1\otimes \sigma_i)(1)=\J_\sigma^*(1)=X$, we see that Equation (\ref{S-D_Y_2}) is equivalent to
	\begin{align*}
		0=\J_\sigma^*\circ(1\otimes\sigma_i)\left(\frac{B}{1+B}\right)+f+(\D W)(X+f).
	\end{align*}
By the assumed invertibility of multiplying by $(1+B)$, this is then equivalent to
	\begin{align*}
		0=&\J_\sigma^*\circ(1\otimes\sigma_i)\left(\frac{B}{1+B}\right)+f+(\D W)(X+f)\\
			&+B\#\J_\sigma^*\circ(1\otimes\sigma_i)\left(\frac{B}{1+B}\right)+B\#f+B\#(\D W)(X+f).
	\end{align*}
Using $\frac{x}{1+x}=x-\frac{x^2}{1+x}$, we then obtain
	\begin{align*}
		K(f)=&(\D W)(X+f)+B\#(\D W)(X+f) \\
			&+\left[B\# f+ B\#\J_\sigma^*\circ(1\otimes\sigma_i)\left(\frac{B}{1+B}\right)-\J_\sigma^*\circ(1\otimes\sigma_i)\left(\frac{B^2}{1+B}\right)\right].
	\end{align*}
Thus it remains to show
	\begin{align*}
		\D_j (W(X+f)) = [(1+B)\# (\D W)(X+f)]_j,
	\end{align*}
for each $j=1,\ldots,N$. Initially suppose $W=X_{k_1}\cdots X_{k_n}$ (the general case will follow via linearity), then
	\begin{align*}
		W(X+f)= (X_{k_1}+f_{k_1})\cdots (X_{k_n}+f_{k_n}).
	\end{align*}
For notational convenience, if we are focusing on the $k_l$th factor then we will write $W(X+f)=A_l(X_{k_l}+f_{k_l})B_l$. Using the derivation property of $\bar{\partial}_j$ in $\D_j=m\circ\diamond\circ(1\otimes\sigma_{-i})\circ\bar{\partial}_j$ we have
	\begin{align*}
		\D_j(W(X+f))=&\sum_{l=1}^n m\circ\diamond\circ(1\otimes \sigma_{-i})\left[ A_l (\alpha_{jk_l}1\otimes 1+\bar{\partial}_j(f_{k_l}) )B_l\right]\\
				=&\sum_{l=1}^n  \alpha_{j k_l} \sigma_{-i}(B_l)A_l + (1\otimes\sigma_{-i})\circ\bar{\partial}_j(f_{k_l})^\diamond\# \sigma_{-i}(B_l)A_l\\
				=&(\D W)(X+f)+\sum_{l=1}^n \partial_{k_l}(f_j)\#\sigma_{-i}(B_l)A_l,
	\end{align*}
where we have used $(\J_\sigma f)^*=(\J_\sigma\D g)^*=(\sigma_i\otimes 1)(\J_\sigma\D g)=(\sigma_i\otimes 1)(\J_\sigma f)$. Now
	\begin{align*}
		[B\# (\D W)(X+f)]_j&=\sum_{k=1}^N [B]_{jk}\# (\D_k W)(X+f)=\sum_{k=1}^N \sum_{l=1}^N [\J_\sigma f]_{jl}\# [\J_\sigma X^{-1}]_{lk}\# (\D_k W)(X+f)\\
				&= \sum_{l=1}^N [\J_\sigma f]_{jl}\#\sum_{k=1}^N [\J_\sigma X^{-1}]_{lk} \sum_{p=1}^N \alpha_{kp}\ m\circ\diamond\circ(1\otimes\sigma_{-i})\circ\delta_p (W)(X+f)\\
				&=\sum_{l=1}^N [\J_\sigma f]_{jl}\# m\circ\diamond\circ(1\otimes\sigma_{-1})\circ\delta_l(W)(X+f)=\sum_{l=1}^n [\J_\sigma f]_{j k_l}\# \sigma_{-i}(B_l)A_l,
	\end{align*}
which is precisely the second term in our above computation of $\D_j(W(X+f))$.
\end{proof}


\subsection{Some identities involving $\J_\sigma$ and $\D$.}\label{identities}

\begin{lem}\label{pictorial_lemma}
Let $g\in\mathscr{P}_{\varphi}^{(R,\sigma)}$ and let $f=\D g$. Then for any $m\geq -1$ we have:
	\begin{align*}
		-\J_\sigma^*\circ(1\otimes \sigma_{i})(B^{m+2})+B\#\J_\sigma^*\circ(1\otimes\sigma_i)(B^{m+1})=\frac{1}{m+2}
\D\left[(\varphi\otimes 1)\circ\text{Tr}_{A^{-1}} +(1\otimes \varphi)\circ\text{Tr}_A\right](B^{m+2})
	\end{align*}
\end{lem}
\begin{proof}
We prove the identity weakly. Let $P\in (\mathscr{P}^{(R)})^N$ be a test function and denote $\phi=\varphi\otimes\varphi^{op}\otimes\text{Tr}$. Then
	\begin{align*}
		\<P\right.&\left., -\J_\sigma^*\circ(1\otimes\sigma_i)(B^{m+2}) + B\# \J_\sigma^*\circ(1\otimes\sigma_i)(B^{m+1})\>_\varphi\\
		 &= -\< \J_\sigma P, (1\otimes\sigma_i)(B^{m+2})\>_{\phi} + \varphi\left( \sum_{i,j=1}^N P_i^* \# B_{ij} \# \left[\J_\sigma^*\circ(1\otimes\sigma_i)(B^{m+1})\right]_j \right)\\
		 &=-\< \J_\sigma P, (1\otimes\sigma_i)(B^{m+2})\>_{\phi} + \varphi\left( \sum_{i,j=1}^N (\sigma_i\otimes 1)(B_{ij}^\diamond) \# P_i^*\# \left[\J_\sigma^*\circ(1\otimes\sigma_i)(B^{m+1})\right]_j\right)\\
		 &=-\< \J_\sigma P, (1\otimes\sigma_i)(B^{m+2})\>_{\phi} + \sum_{i,j=1}^N\< (1\otimes\sigma_{-i})(B_{ij}^*)\# P_i, \left[\J_\sigma^*\circ(1\otimes\sigma_i)(B^{m+1})\right]_j\>_\varphi\\
		 &=-\< \J_\sigma P, (1\otimes\sigma_i)(B^{m+2})\>_{\phi} + \<\J_\sigma\left\{(1\otimes\sigma_{-i})(B^*)\# P\right\}, (1\otimes \sigma_i)(B^{m+1})\>_{\phi}\\
		  &=-\< \J_\sigma P, (1\otimes\sigma_i)(B^{m+2})\>_{\phi} + \<\J_\sigma X^{-1}\# \J_\sigma\left\{\hat{\sigma}_i(\J_\sigma f)\# P\right\}, (1\otimes \sigma_i)(B^{m+1})\>_{\phi},
	\end{align*}
where we have used $(\J_\sigma f)^*=(\sigma_i\otimes 1)(\J_\sigma f)$ from Lemma \ref{change_of_variables}.(iii). Now we focus on the term $\J_\sigma\{\hat{\sigma}_i(\J_\sigma f)\# P\}$:
	\begin{align*}
		\left[\J_\sigma\{\hat{\sigma}_i(\J_\sigma f)\# P\}\right]_{jk}= \sum_{l=1}^N (\partial_k\otimes 1)\circ \hat{\sigma}_i\circ\partial_l (f_j) \#_2 P_l + (1\otimes \partial_k)\circ\hat{\sigma}_i\circ\partial_l(f_j) \#_1 P_l + \hat{\sigma}_i\circ\partial_l (f_j) \# \partial_k(P_l),
	\end{align*}
where $a\otimes b\otimes c\#_1 \xi = a\xi b\otimes c$ and $a\otimes b\otimes c \#_2 \xi= a\otimes b\xi c$. Define
	\begin{align*}
		Q^P_{jk} = \sum_{l=1}^N (\partial_k\otimes 1)\circ \hat{\sigma}_i\circ\partial_l (f_j) \#_2 P_l + (1\otimes \partial_k)\circ\hat{\sigma}_i\circ\partial_l(f_j) \#_1 P_l ,
	\end{align*}
so that
	\begin{align*}
		\J_\sigma\{\hat{\sigma}_i(\J_\sigma f)\# P\} = Q^P + \hat{\sigma}_i (\J_\sigma f)\# \J_\sigma P.
	\end{align*}
Continuing our initial computation we obtain
	\begin{align*}
		\<P, -\J_\sigma^*\circ(1\otimes\sigma_i)\right.&\left.(B^{m+2}) + B\# \J_\sigma^*\circ(1\otimes\sigma_i)(B^{m+1})\>_\varphi\\
			 =&-\phi\left( \left(\J_\sigma P\right)^*\# (1\otimes\sigma_i)(B^{m+1})\right)+ \phi\left( (Q^P)^*\#\J_\sigma X^{-1} \# (1\otimes\sigma_i)(B^{m+1})\right)\\
			   &+ \phi\left( (\J_\sigma P)^*\# \hat{\sigma}_{-i}( (\J_\sigma f)^*) \# \J_\sigma X^{-1} \# (1\otimes \sigma_i)(B^{m+1})\right)\\
			  =& \< Q^P, \J_\sigma X^{-1}\# (1\otimes\sigma_i)(B^{m+1})\>_{\phi}.
	\end{align*}
Hence
	\begin{align*}
		\< -\J_\sigma^*\circ \right.&\left.(1\otimes\sigma_i)(B^{m+2})  + B\# \J_\sigma^*\circ(1\otimes\sigma_i)(B^{m+1}),P\>_\varphi\\
			&= \< \J_\sigma X^{-1}\# (1\otimes\sigma_i)(B^{m+1}), Q^P\>_{\phi}= \phi( (1\otimes \sigma_{-i})( (B^^*)^{m+1}) \J_\sigma X^{-1} Q^P)\\
			&=\phi( (1\otimes\sigma_{-i})(\underbrace{ \J_\sigma X^{-1} (\J_\sigma f)^* \cdots \J_\sigma X^{-1} (\J_\sigma f)^*}_{m+1} \J_\sigma X^{-1} ) Q^P)\\
			&=\phi( (1\otimes\sigma_{-i})(\underbrace{ \J_\sigma X^{-1} (\sigma_i\otimes 1)(\J_\sigma f) \cdots \J_\sigma X^{-1} (\sigma_i\otimes 1)(\J_\sigma f)}_{m+1} \J_\sigma X^{-1} ) Q^P)\\
			&=\phi( \J_\sigma X^{-1} \hat{\sigma}_i(B^{m+1}) Q^P)=\phi(Q^P \J_\sigma X^{-1} B^{m+1}).
	\end{align*}
We break from the present computation to consider the terms on the other side of the desired equality.\par
For each $u=1,\ldots, m+2$ let $R_u$ be the matrix such that $[R_u]_{i_uj_u}=a_u\otimes b_u$ for some $i_u, j_u\in\{1,\ldots, N\}$ and all other entries are zero. Then
	\begin{align*}
		\text{Tr}_{A^{-1}}( R_1\cdots R_{m+2})=\text{Tr}(A^{-1}R_1\cdots R_{m+2})= [A^{-1}]_{j_{m+2}i_1}\prod_{u=1}^{m+2}\delta_{j_u=i_{u+1}} a_1\cdots a_{m+2}\otimes b_{m+2}\cdots b_1.
	\end{align*}
Denote $C=[A^{-1}]_{j_{m+2}i_1}\prod_{u=1}^{m+2}\delta_{j_u=i_{u+1}}$. Then
	\begin{align*}
		\sum_{k}\varphi &\left(\bar{D}_k(\varphi\otimes 1)\text{Tr}_{A^{-1}}(R_1\cdots R_{m+2}) P_k\right)= \sum_{k}C \varphi(a_1\cdots a_{m+2}) \varphi( \bar{\D}_k(b_{m+2}\cdots b_1) P_k)\\
		&=\sum_{k,u} C \varphi(\sigma_i(a_u\cdots a_{m+2})a_1\cdots a_{u-1}) \varphi( b_{u-1}\cdots b_1\sigma_i(b_{m+2}\cdots b_{u+1})\cdot \hat{\sigma}_i\circ\partial_k(b_u)\# P_k)\\
		&=\sum_{u} \varphi\otimes\varphi^{op}\otimes\text{Tr} ( \Delta_{(1,P)}(R_u) (\sigma_i\otimes\sigma_i)(R_{u+1}\cdots R_{m+2})A^{-1}R_1\cdots R_{u-1}),
	\end{align*}
where for an arbitrary matrix $O$
	\begin{align*}
		[\Delta_{(1,P)}(O)]_{ij}=\sum_k \sigma_{i}\otimes(\hat{\sigma}_i\circ\partial_k)([O]_{ij})\#_2 P_k.
	\end{align*}
Using linearity, replace $R_u$ with $B$ for each $u$. From Lemma \ref{change_of_variables}.(iv) we know $(\sigma_i\otimes\sigma_i)(\J_\sigma f) A^{-1}=A^{-1}\J_\sigma f$. As $[A,\J_\sigma X^{-1}]=0$, we also have $(\sigma_i\otimes\sigma_i)(B)A^{-1}=A^{-1}B$ and hence
	\begin{align*}
		\sum_{k}\varphi\left(\bar{D}_k(\varphi\otimes 1)\text{Tr}_{A^{-1}}(B^{m+2}) P_k\right)
		=(m+2)\phi(\Delta_{(1,P)}(B)A^{-1}B^{m+1}).
	\end{align*}
Observe that the left-hand side is $\<\D(\varphi\otimes 1)\text{Tr}_{A^{-1}}(B^{m+2}, P\>$. Indeed,
	\begin{align*}
		\<\D(\varphi\otimes 1)\text{Tr}_{A^{-1}}(B^{m+2}), P\> =\sum_k \varphi( \bar{\D}_k (\varphi\otimes 1)\text{Tr}_{A^{-1}}((B^*)^{m+2}) P_k),
	\end{align*}
and
	\begin{align*}
		(\varphi\otimes 1)\text{Tr}(A^{-1} (B^*)^{m+2})&=(\varphi\otimes1)\text{Tr}(A^{-1} \underbrace{\J_\sigma X^{-1} (\sigma_i\otimes 1)(\J_\sigma f) \cdots \J_\sigma X^{-1} (\sigma_i\otimes 1)(\J_\sigma f)}_{m+2})\\
			&=(\varphi\otimes 1)(\sigma_i\otimes 1)\text{Tr}(A^{-1} B^{m+2})=(\varphi\otimes 1)\text{Tr}(A^{-1} B^{m+2}),
	\end{align*}
where in the second to last equality we have used the fact that $A^{-1}$ and $\J_\sigma X^{-1}$ commute.\par
So 
	\begin{align*}
		\frac{1}{m+2}\<\D(\varphi\otimes 1)\text{Tr}_{A^{-1}}(B^{m+2}), P\> = \phi(\Delta_{(1,P)}(B)A^{-1}B^{m+1}),
	\end{align*}
and a similar computation yields
	\begin{align*}
		\frac{1}{m+2}\<\D(1\otimes \varphi)\text{Tr}_{A}(B^{m+2}), P\> = \phi(\Delta_{(2,P)}(B)A B^{m+1}),
	\end{align*}
where for an arbitrary matrix $O$
	\begin{align*}
		[\Delta_{(2,P)}(O)]_{ij}= \sum_k (\hat{\sigma}_i\circ\partial_k)\otimes \sigma_{-i} ([O]_{ij})\#_1 P_k.
	\end{align*}\par
Thus it suffices to show
	\begin{align*}
		\Delta_{(1,p)}(B)A^{-1}+\Delta_{(2,P)}(B)A = Q^P \J_\sigma X^{-1}.
	\end{align*}
This is easily verified entry-wise using the identities
	\begin{align*}
		(\delta_r\otimes 1)\circ\hat{\sigma}_i\circ\partial_k&= (\sigma\otimes (\hat{\sigma}\circ\partial_k))\circ\left(\sum_{b=1}^N [A^{-1}]_{br}\delta_b\right),\\
		(1\otimes \delta_r)\circ\hat{\sigma}_i\circ\partial_k&=((\hat{\sigma}_i\circ\partial_k)\otimes\sigma_{-i})\circ\left(\sum_{b=1}^N [A]_{br}\delta_b\right),
	\end{align*}
and the definitions of $Q^P$, $\Delta_{(1,P)}$, $\Delta_{(2,P)}$.
\end{proof}

\begin{lem}\label{Q}
Assume $f=\D g$ for $g=g^*\in \mathscr{P}_\varphi^{(R,\sigma)}$ and that $\|B\|_{R\otimes_\pi R}<1$. Let
	\begin{align*}
		Q(g)=\left[(1\otimes\varphi)\circ\text{Tr}_{A}+(\varphi\otimes 1)\circ\text{Tr}_{A^{-1}}\right](B-\log(1+B)).
	\end{align*}
Then
	\begin{equation*}
		\D Q(g)=B\# \J_\sigma^*\circ(1\otimes\sigma_i)\left(\frac{B}{1+B}\right) - \J_\sigma^*\circ(1\otimes\sigma_i)\left(\frac{B^2}{1+B}\right).
	\end{equation*}
\end{lem}
\begin{proof}
Using the previous lemma this follows from comparing the convergent power series of each side.
\end{proof}

\begin{lem}\label{K_2}
Let
	\begin{equation*}
		K(f)=-\J_\sigma^*\circ(1\otimes\sigma_i)(B)-f.
	\end{equation*}
Assume that $f=\D g$ for $g=g^*\in \mathscr{P}_\varphi^{(R,\sigma)}$. Then
	\begin{align*}
		K(f)=\D\left\{ \left[ (\varphi\otimes 1)\circ\text{Tr}_{A^{-1}} + (1\otimes\varphi)\circ \text{Tr}_A\right](B) - \mathscr{N} g\right\}.
	\end{align*}
\end{lem}
\begin{proof}
When $m=-1$, the equality in Lemma \ref{pictorial_lemma} becomes
	\begin{equation*}
		\D\left[ (\varphi\otimes 1)\circ\text{Tr}_{A^{-1}} + (1\otimes\varphi)\circ \text{Tr}_A\right](B) = -\J_\sigma^*\circ(1\otimes\sigma_i)(B)+B\#\J_\sigma^*\circ(1\otimes \sigma_i)(1).
	\end{equation*}
Since $X=\J_\sigma^*(1)=\J_\sigma^*\circ(1\otimes \sigma_i)(1)$, the last term becomes $B\# X= \J f\# X=\mathscr{N} f$. Since $\D\mathscr{N} g=(\mathscr{N}+1)\D g=\mathscr{N} f+  f$, we have
	\begin{align*}
		\D\left\{ \left[ (\varphi\otimes 1)\circ\text{Tr}_{A^{-1}} + (1\otimes\varphi)\circ \text{Tr}_A\right](B) - \mathscr{N} g\right\}=-\J_\sigma^*\circ(1\otimes\sigma_i)(B) +\mathscr{N} f- \D\mathscr{N} g=K(f),
	\end{align*}
as claimed.
\end{proof}

\begin{lem}\label{S-D_Y_2=3}
Assume $f=\D g$ for $g=g^*\in \mathscr{P}_\varphi^{(R,\sigma)}$ and $\|\J\D g\|_{R\otimes_\pi R}<1$. Let $Q(g)$ be as before. Then Equation (\ref{S-D_Y_2}) is equivalent to
	\begin{align}\label{S-D_Y_3}
		\D\left\{ \left[ (\varphi\otimes 1)\circ\text{Tr}_{A^{-1}} +\right.\right.&\left.\left. (1\otimes\varphi)\circ \text{Tr}_A\right](\J\D g) - \mathscr{N} g\right\} \notag\\
											&= \D(W(X+\D g))+\D Q(g)+ \J_\sigma\D g\#(\J_\sigma X)^{-1}\#\D g.
	\end{align}
\end{lem}
\begin{proof}
By Lemma \ref{K_2}, the left-hand side is $K(f)$. Then using Lemmas \ref{K_1} and \ref{Q} we have
	\begin{align*}
		K(f)&=\D(W(X+f)) + B\#\J_\sigma^*\circ(1\otimes\sigma_i)\left(\frac{B}{1+B}\right)-\J_\sigma^*\circ(1\otimes\sigma_i)\left(\frac{B^2}{1+B}\right) + B\# f\\
		&=\D(W(X+\D g))+ \D Q(g) + \J_\sigma\D g\#(\J_\sigma X)^{-1}\# \D g.
	\end{align*}
Note that the hypothesis in Lemma 2.5 that the map $\xi\mapsto (1+B)\#\xi$ is invertible is satisfied since $\|B\|_{R\otimes_\pi R}=\|\J\D g\|_{R\otimes_\pi R}<1$.
\end{proof}

To prove the existence of a $g$ satisfying the equation above we use a fixed point argument and therefore require some preliminary estimates.


\subsection{Technical estimates.}\label{technical_estimates}

Recall that $\|X_j\|\leq 2$ for each $j=1,\ldots, N$. Since $\varphi$ is a state it then follows that
	\begin{equation}\label{phi_monomial}	
		|\varphi(X_{i_1}\cdots X_{i_n})|\leq 2^n.
	\end{equation}
	
\begin{lem}\label{centralizer_commutes_with_A}
For $g_1,\ldots, g_m\in \mathscr{P}_{\varphi}$
	\begin{align*}
		\left[(1\otimes\varphi)\circ\text{Tr}_A+(\varphi\otimes1)\circ\text{Tr}_{A^{-1}}\right]( \J\D g_1\#\cdots\# \J\D g_m)\in \mathscr{P}_\varphi.
	\end{align*}
\end{lem}
\begin{proof}
Recall $A^{-t}\# \J\D g\# A^t=(\sigma_{it}\otimes \sigma_{it})(\J\D g)$ for $g\in\mathscr{P}_{\varphi}^{(R,\sigma)}$ by Lemma \ref{change_of_variables}.(iv). Given this identity, for $g_1,\ldots, g_m\in \mathscr{P}_\varphi$ we have
	\begin{align*}
		(1\otimes \varphi)\circ\text{Tr}(A\# \J\D g_1\#\cdots \# \J\D g_m)&=(1\otimes\varphi)\circ\text{Tr}\left( (\sigma_{-i}\otimes\sigma_{-i})(\J\D g_1\#\cdots\# \J\D g_m)\# A\right)\\
												&=(1\otimes\varphi)\circ\text{Tr}\left( A\#(\sigma_{-i}\otimes\sigma_{-i})(\J\D g_1\#\cdots\# \J\D g_m)\right)\\
												&=\sigma_{-i}\circ(1\otimes\varphi)\circ\text{Tr}(A\# \J\D g_1\#\cdots\# \J\D g_m),
	\end{align*}
implying $(1\otimes \varphi)\circ\text{Tr}_A(\J\D g_1\#\cdots \# \J\D g_m)\in\mathscr{P}_\varphi$. Similarly
	\begin{align*}
		(\varphi\otimes 1)\circ\text{Tr}(A^{-1}\# \J\D g_1\#\cdots \# \J\D g_m)=\sigma_{i}\circ (\varphi\otimes 1)\circ\text{Tr}(A^{-1}\# \J\D g_1\#\cdots \# \J\D g_m),
	\end{align*}
implying $(\varphi\otimes 1)\circ\text{Tr}_{A^{-1}} (\J\D g_1\#\cdots \# \J\D g_m)\in\mathscr{P}_\varphi$.
\end{proof}

Using Equation (\ref{cyclic_derivative_of_cyclically_symmetric}) we see that for $g\in\pi_n\left(\mathscr{P}_{c.s.}^{(R,\sigma)}\right)$
	\begin{align*}
	\begin{tikzpicture}[baseline]
	\node[left] at (0,0) {$\displaystyle\J\D\Sigma g=\sum_{j=1}^N \sum_{l=1}^{n-1}\sum_{|\ul{i}|=n} c(\ul{i}) \alpha_{ji_n}$};
	\draw[thick] (0,-.5) rectangle (2.5,.5);
	\node at (1.25,.25) {$i_1\cdots i_{l-1}$};
	\node at (1.25,-.25) {$i_{n-1}\cdots i_{l+1}$};
	\node[right] at (0,0) {$j$};
	\node[left] at (2.5,0) {$i_l$};
	\node[right] at (2.5,-.5) {.};
	\end{tikzpicture}
	\end{align*}

\begin{lem}
Let $g_1,\ldots, g_m\in \Pi\left(\mathscr{P}_{c.s.}\right)$. Set
	\begin{equation*}
		Q_m(g_1,\ldots,g_m)=\left[ (1\otimes\varphi)\circ\text{Tr}_A +(\varphi\otimes 1)\circ\text{Tr}_{A^{-1}}\right]\left( \J\D g_1\cdots \J\D g_m\right).
	\end{equation*}
Assume $R\geq4$, so that $\frac{2}{R}\leq\frac{1}{2}$. Then
	\begin{align*}
		\| Q_m(\Sigma g_1,\ldots, \Sigma g_m)\|_{R,\sigma} \leq \|A\|\frac{2^{m+1}}{R^{2m}} \prod_{u=1}^m \|g_u\|_{R,\sigma}.
	\end{align*}
In particular, $Q_m$ extends to a bounded multilinear operator on $\mathscr{P}^{(R,\sigma)}_{c.s.}$ with values in $\mathscr{P}^{(R,\sigma)}_{\varphi}$.
\end{lem}
\begin{proof}
First, for each $u=1,\ldots, m$  assume $g_u\in \pi_{n_u}(\mathscr{P}_{c.s.})$ and write $g_u=\sum_{|\ul{i}^{(u)}|=n_u} c_u(\ul{i}^{(u)}) X_{\ul{i}^{(u)}}$. By the computation preceding the statement of the lemma we can see that
	\begin{align*}
		\text{Tr}_{A^{-1}}( \J\D\Sigma g_1\cdots \J\D\Sigma g_m)=&\sum_{p_0,\ldots,p_m=1}^N [A^{-1}]_{p_mp_0}[\J\D\Sigma g_1]_{p_0p_1}\cdots [\J\D\Sigma g_m]_{p_{m-1}p_m}\\
			=&\sum_{j=1}^N \sum_{l_1=1}^{n_1-1}\cdots \sum_{l_m=1}^{n_m-1} \sum_{|\ul{i}^{(1)}|=n_1}\cdots \sum_{|\ul{i}^{(m)}|=n_m}\prod_{u=1}^m c_u(\ul{i}^{(u)}) \alpha_{i_{l_{u-1}}^{(u-1)}i_{n_u}^{(u)}}\\
			&\times
				\begin{tikzpicture}[baseline]
				\draw[thick] (0,-.65) rectangle (3.4,.65);
				\node at (1.7,.3) {$i_1^{(1)}\cdots i_{l_1-1}^{(1)}$};
				\node at (1.7,-.3) {$i_{n_1-1}^{(1)}\cdots i_{l_1+1}^{(1)}$};
				\node[right] at (0,0) {$i_{l_0}^{(0)}$};
				\node[left] at (3.4,0) {$i_{l_1}^{(1)}$};
				\node at (3.75,0) {$\cdots$};
				\draw[thick] (4,-.65) rectangle (8,.65);
				\node at (6.15,.3) {$i_1^{(m)}\cdots i_{l_m-1}^{(m)}$};
				\node at (6.15,-.3) {$i_{n_m-1}^{(m)}\cdots i_{l_m+1}^{(m)}$};
				\node[right] at (4,0) {$i_{l_m-1}^{(m-1)}$};
				\node[left] at (8,0) {$i_{l_m}^{(m)}$};
				\node[right] at (8,0) {$[A^{-1}]_{i_{l_m}^{(m)} j}$};
				\node at (9.7,-.2) {,};
				\end{tikzpicture}
	\end{align*}
where $i_{l_0}^{(0)}=j$. Hence
	\begin{align*}
		(\varphi\otimes 1)\circ \text{Tr}_{A^{-1}}( \J\D\Sigma g_1\cdots \J\D\Sigma g_m)=&\sum_{j=1}^N \sum_{l_1,\ldots,l_m}\sum_{\ul{i}^{(1)},\ldots,\ul{i}^{(m)}} \prod_{u=1}^m c_u(\ul{i}^{(u)}) \alpha_{i_{l_{u-1}}^{(u-1)} i_{n_u}^{(u)}} \\
						&\times\varphi( X_{i_1^{(1)}}\cdots X_{i_{l_1-1}^{(1)}} \cdots X_{i_1^{(m)}}\cdots X_{i_{l_m-1}^{(m)}})\\
						&\times X_{i_{l_{m+1}}^{(m)}} \cdots X_{i_{n_m-1}^{(m)}} \cdots X_{i_{l_1+1}^{(1)}} \cdots X_{i_{n_1-1}^{(1)}}\cdot [A^{-1}]_{i_{l_m}^{(m)}j}.
	\end{align*}
Fix $l_1,\ldots, l_m$ in the above quantity, then the sum over $i_{l_0}^{(0)}$ and the multi-indices $\ul{i}^{(1)},\ldots,\ul{i}^{(m)}$ is a sum of monomials all with the same degree: $\sum_u n_u-l_u-1=:n_0$. By Lemma \ref{centralizer_commutes_with_A}, it suffices to bound $\|\rho^k(\cdot)\|_R$ for $k\in\{-n_0+1,\ldots, -1,0\}$. For $k=0$ we have
	\begin{align*}
		\left\| (\varphi\otimes 1)\circ\text{Tr}_{A^{-1}}\right.&\left.(\J\D\Sigma g_1\cdots \J\D\Sigma g_m)\right\|_R \\
				&\leq \sum_{j=1}^N \sum_{l_1,\ldots,l_m} \sum_{\ul{i}^{(1)},\ldots, \ul{i}^{(m)}} \prod_{u=1}^m \left|c_u\left(\ul{i}^{(u)}\right)\right| \left|[A^{-1}]_{i_m^{(m)} j}\right| R^{n_1-l_1-1+\cdots +n_m-l_m-1} 2^{l_1-1+\cdots +l_m-1}\\
				&\leq \sum_{l_1,\ldots,l_m} \sum_{\ul{i}^{(1)},\ldots, \ul{i}^{(m)}} \prod_{u=1}^m \left|c_u\left(\ul{i}^{(u)}\right)\right| \left\|A^{-1}\right\| R^{n_1+\cdots +n_m - 2m} \left(\frac{2}{R}\right)^{l_1+\cdots +l_m-m}\\
				&=\|A\| \prod_{u=1}^m \frac{1}{R^2}\| g_u\|_R \sum_{l_u=1}^{n_m-1} \left(\frac{2}{R}\right)^{l_u-1} \leq \|A\| \prod_{u=1}^m \frac{2}{R^2}\|g_u\|_R=\|A\| \prod_{u=1}^m \frac{2}{R^2}\|g_u\|_{R,\sigma},
	\end{align*}
where we have used $\|g_u\|_R=\|g_u\|_{R,\sigma}$.\par
Next, let $k\in\{ -n_0+1,\ldots, -1\}$ and suppose
	\begin{align*}
		\rho^k\left( X_{i_{l_{m}+1}^{(m)}} \cdots X_{i_{n_m-1}^{(m)}}\right. & \left.\cdots X_{i_{l_1+1}^{(1)}} \cdots X_{i_{n_1-1}^{(1)}}\right)\\
			&=X_{i_{a+1}^{(v)}}\cdots X_{i_{n_v-1}^{(v)}} \cdots X_{i_{l_1+1}^{(1)}} \cdots X_{i_{n_1-1}^{(1)}}\sigma_i\left( X_{i_{l_m+1}^{(m)}}\cdots X_{i_{n_m-1}^{(m)}} \cdots X_{i_{l_v+1}^{(v)}}\cdots X_{i_a^{(v)}}\right),
	\end{align*}
for some $v\in \{1,\ldots,m\}$ and some $a\in\{l_v+1,\ldots ,n_v-1\}$. The corresponding $\varphi$ output is
	\begin{align*}
		\varphi\left( X_{i_1^{(1)}}\cdots X_{i_{l_1-1}^{(1)}}\right.& \left.\cdots X_{i_1^{(m)}}\cdots X_{i_{l_m-1}^{(m)}}\right)\\
				&=\varphi\left( \sigma_i\left( X_{i_1^{(v)}}\cdots X_{i_{l_v -1}^{(v)}} \cdots X_{i_1^{(m)}}\cdots X_{i_{l_m-1}^{(m)}}\right) X_{i_1^{(1)}}\cdots X_{i_{l_1-1}^{(1)}} \cdots X_{i_1^{(v-1)}}\cdots X_{i_{l_{v-1}-1}^{(v-1)}}\right).
	\end{align*}
Using Lemma \ref{centralizer_commutes_with_A} we can in this case replace $\text{Tr}(A^{-1}\# \J\D\Sigma g_1\cdots \J\D\Sigma g_m)$ with
	\begin{equation*}
		\text{Tr}(\J\D\Sigma g_1 \cdots \J\D\Sigma g_v\#A^{-1}\#(\sigma_{-i}\otimes\sigma_{-i})(\J\D\Sigma g_{v+1}\cdots \J\D\Sigma g_m))
	\end{equation*}
so that output of $\rho^k$ changes to
	\begin{align*}
		X_{i_{a+1}^{(v)}}\cdots X_{i_{n_v-1}^{(v)}} \cdots X_{i_{l_1+1}^{(1)}} \cdots X_{i_{n_1-1}^{(1)}} X_{i_{l_m+1}^{(m)}}\cdots X_{i_{n_m-1}^{(m)}} \cdots \sigma_i\left(X_{i_{l_v+1}^{(v)}}\cdots X_{i_a^{(v)}}\right),
	\end{align*}
and the output of $\varphi$ changes to
	\begin{align*}
		\varphi\left( \sigma_i\left( X_{i_1^{(v)}}\cdots X_{i_{l_v -1}^{(v)}}\right) \cdots X_{i_1^{(m)}}\cdots X_{i_{l_m-1}^{(m)}} X_{i_1^{(1)}}\cdots X_{i_{l_1-1}^{(1)}} \cdots X_{i_1^{(v-1)}}\cdots X_{i_{l_{v-1}-1}^{(v-1)}}\right).
	\end{align*}
Hence it suffices to consider when $v=m$. In this case we further fix $\ul{i}^{(1)},\ldots, \ul{i}^{(m-1)}$ and denote $F_u:=X_{i_1^{(u)}}\cdots X_{i_{l_u-1}^{(u)}}$ and $G_u:=X_{i_{l_u+1}^{(u)}}\cdots X_{i_{n_u-1}^{(u)}}$. Consider
	\begin{align*}
		\sum_{j=1}^N \sum_{\ul{i}^{(m)}} & c_m\left(\ul{i}^{(m)}\right) \alpha_{i_{l_{m-1}}^{(m-1)} i_{n_m}^{(m)}}[A^{-1}]_{i_{l_m}^{(m)} j}\\
				&\times \varphi\left( \sigma_i\left( X_{i_1^{(m)}}\cdots X_{i_{l_m -1}^{(m)}}\right)  F_1\cdots F_{m-1}\right)   X_{i_{a+1}^{(m)}}\cdots X_{i_{n_m-1}^{(m)}}G_1\cdots G_{m-1} \sigma_i\left(X_{i_{l_m+1}^{(m)}}\cdots X_{i_a^{(m)}}\right)\\
		=&\sum_{j=1}^N \sum_{\ul{i}^{(m)}}\ \sum_{\hat{i}_1^{(m)},\ldots, \hat{i}_{l_m-1}^{(m)}=1}^N\  \sum_{\hat{i}_{l_m+1}^{(m)},\ldots, \hat{i}_a^{(m)}=1}^N c_m\left(\ul{i}^{(m)}\right) \alpha_{i_{l_{m-1}}^{(m-1)} i_{n_m}^{(m)}}\prod_{t\neq l_m}[A^{-1}]_{i_t^{(m)} \hat{i}_t^{(m)}} \cdot [A^{-1} ]_{i_{l_m}^{(m)} j}\\
				&\times \varphi\left( X_{\hat{i}_1^{(m)}}\cdots X_{\hat{i}_{l_m-1}^{(m)}}F_1\cdots F_{m-1}\right)  X_{i_{a+1}^{(m)}}\cdots X_{i_{n_m-1}^{(m)}}G_1\cdots G_{m-1} X_{\hat{i}_{l_m+1}^{(m)}}\cdots X_{\hat{i}_a^{(m)}}\\
		=&\sum_{\ul{i}^{(m)}} \sum_{|\ul{\hat{i}}^{(m)}|=a} c_m\left( \ul{i}^{(m)}\right) \alpha_{i_{l_{m-1}}^{(m-1)} i_{n_m}^{(m)}}\prod_{t=1}^a [A^{-1}]_{i_t^{(m)} \hat{i}_t^{(m)}} \varphi\left( X_{\hat{i}_1^{(m)}}\cdots X_{\hat{i}_{l_m-1}^{(m)}}F_1\cdots F_{m-1}\right) \\
				&\times X_{i_{a+1}^{(m)}}\cdots X_{i_{n_m-1}^{(m)}}G_1\cdots G_{m-1} X_{\hat{i}_{l_m+1}^{(m)}}\cdots X_{\hat{i}_a^{(m)}}\\
		=&\sum_{\ul{j}^{(m)}} c_m\left( \ul{j}^{(m)}\right) \alpha_{i_{l_{m-1}}^{(m-1)} j_{n_m-a}^{(m)}}\varphi\left( X_{j_{n_m-a+1}^{(m)}}\cdots X_{i_{n_m-a+l_m-1}^{(m)}} F_1\cdots F_{m-1}\right)\\
			&\times  X_{j_1^{(m)}} \cdots X_{j_{n_m-a-1}^{(m)}} G_1\cdots G_{m-1} X_{i_{n_m-a+l_m+1}^{(m)}} \cdots X_{j_{n_m}^{(m)}},
	\end{align*}
where in the final equality we have used the characterization of the coefficients of elements of $\mathscr{P}_{c.s.}$ given by (\ref{cyclically_symmetric_coefficients_negative}). We note that while the multi-index has changed to $\ul{j}^{(m)}$, there are still $l_m-1$ terms inside $\varphi$ and $n_m-l_m-1$ outside. Thus we have
	\begin{align*}
		\left\| \rho^k\circ(\varphi\otimes 1)\circ\right.&\left.\text{Tr}_{A^{-1}}(\J\D\Sigma g_1\cdots \J\D\Sigma g_m)\right\|_R \\
				&\leq \sum_{l_1,\ldots,l_m} \sum_{\ul{i}^{(1)},\ldots, \ul{i}^{(m)}} \prod_{u=1}^m \left| c_u\left(\ul{i}^{(u)}\right)\right| R^{n_1-l_1-1+\cdots +n_m-l_m-1}2^{l_1-1+\cdots +l_m-1}\\
				&=\sum_{l_1,\ldots,l_m} \sum_{\ul{i}^{(1)},\ldots, \ul{i}^{(m)}} \prod_{u=1}^m \left| c_u\left(\ul{i}^{(u)}\right)\right| R^{n_1+\cdots +n_m - 2m}\left(\frac{2}{R}\right)^{l_1-1+\cdots +l_m-1}\\
				&=\prod_{u=1}^m \frac{1}{R^2} \|g_u\|_{R} \sum_{l_u=1}^{n_u-1} \left(\frac{2}{R}\right)^{l_u-1}\leq \prod_{u=1}^m \frac{2}{R^2} \|g_u\|_{R,\sigma}\leq \|A\| \prod_{u=1}^m \frac{2}{R^2} \|g_u\|_{R,\sigma}.
	\end{align*}
Thus	
	\begin{align*}
		\left\| (\varphi\otimes 1)\circ\text{Tr}_{A^{-1}}(\J\D\Sigma g_1\cdots \J\D\Sigma g_m)\right\|_{R,\sigma}\leq \|A\| \frac{2^m}{R^{2m}}\prod_{u=1}^m \|g_u\|_{R,\sigma},
	\end{align*}
and similar estimates show
	\begin{align*}
		\left\| (\varphi\otimes 1)\circ\text{Tr}_{A}(\J\D\Sigma g_1\cdots \J\D\Sigma g_m)\right\|_{R,\sigma}\leq \|A\| \frac{2^m}{R^{2m}}\prod_{u=1}^m \|g_u\|_{R,\sigma}.
	\end{align*}\par
Now let $g_1,\ldots, g_m\in \mathscr{P}_{c.s.}$ be arbitrary. We note that $\pi_{n_u}(g_u)\in\mathscr{P}_{c.s.}$ for each $n_u\geq 0$ since $[\rho,\pi_{n_u}]=0$. Then since $Q_m$ is multi-linear we have
	\begin{align*}
		Q_m(\Sigma g_1,\ldots,\Sigma g_m)= \sum_{n_1,\ldots, n_m=0}^\infty Q_m\left(\Sigma \pi_{n_1}(g_1),\ldots, \Sigma \pi_{n_m}(g_m)\right),
	\end{align*}
and hence
	\begin{align*}
		\left\| Q_m(\Sigma g_1,\ldots,\Sigma g_m)\right\|_{R,\sigma} &\leq \sum_{n_1,\ldots, n_m} \|A\|\frac{2^{m+1}}{R^{2m}} \prod_{u=1}^m \|\pi_{n_u}(g_u)\|_{R,\sigma} \\
				&= \|A\| \frac{2^{m+1}}{R^{2m}} \prod_{u=1}^m \sum_{n_u=0}^\infty \|\pi_{n_u}(g_u)\|_{R,\sigma}=\|A\|\frac{2^{m+1}}{R^{2m}} \prod_{u=1}^m \|g_u\|_{R,\sigma}.
	\end{align*}
Thus $Q_m$ extends to a bounded multilinear operator on $\mathscr{P}^{(R,\sigma)}_{c.s.}$. That $Q_m$ takes values in $\mathscr{P}^{(R,\sigma)}_\varphi$ follows from Lemma \ref{centralizer_commutes_with_A}.
\end{proof}

\begin{lem}\label{Q_m}
For $f,g\in \mathscr{P}_{c.s.}^{(R,\sigma)}$ set $Q_m(\Sigma g)=Q_m(\Sigma g,\ldots, \Sigma g)$ and assume $R\geq4$. Then
	\begin{align*}
		\left\| Q_m(\Sigma g)- Q_m(\Sigma f)\right\|_{R,\sigma} \leq \|A\|\frac{2^{m+1}}{R^{2m}} \sum_{k=0}^{m-1} \|g\|_{R,\sigma}^k \|f\|_{R,\sigma}^{m-k-1} \|f-g\|_{R,\sigma}.
	\end{align*}
In particular, $\|Q_m(\Sigma g)\|_{R,\sigma}\leq \|A\| \frac{2^{m+1}}{R^{2m}} \|g\|_{R,\sigma}^m$.
\end{lem}
\begin{proof}
Using a telescoping sum we have
	\begin{align*}
		\|Q_m(\Sigma f)-Q_m(\Sigma g)\|_{R,\sigma}&=\left\| \sum_{k=0}^{m-1} Q_m(\underbrace{\Sigma g,\ldots,\Sigma g}_k,\underbrace{\Sigma f,\ldots,\Sigma f}_{m-k})-Q_m(\underbrace{\Sigma g,\ldots,\Sigma g}_{k+1},\underbrace{\Sigma f,\ldots,\Sigma f}_{m-k-1})\right\|_{R,\sigma}\\
					&\leq \sum_{k=0}^{m-1} \| Q_m(\underbrace{\Sigma g,\ldots \Sigma g}_{k},\Sigma f-\Sigma g,\underbrace{\Sigma f,\ldots,\Sigma f}_{m-k-1})\|_{R,\sigma}\\
					&\leq  \|A\| \frac{2^{m+1}}{R^{2m}} \sum_{k=0}^{m-1}\|g\|_{R,\sigma}^k\|f\|_{R,\sigma}^{m-k-1}\|f-g\|_{R,\sigma}.
	\end{align*}
\end{proof}

\begin{lem}\label{Q_2}
Assume $R\geq4$. Let $g\in\mathscr{P}^{(R,\sigma)}_{c.s.}$ be such that $\|g\|_{R,\sigma}<\frac{R^2}{2}$, and set
	\begin{equation*}
		Q(\Sigma g)=\sum_{m\geq 0} \frac{(-1)^m}{m+2}Q_{m+2}(\Sigma g).
	\end{equation*}
Then this series converges in $\|\cdot\|_{R,\sigma}$. Moreover, in the sense of analytic functional calculus on $M_{N}(W^*(\mathscr{P}\otimes \mathscr{P}^{op},\varphi\otimes\varphi^{op}))$, we have the equality
	\begin{equation*}
		Q(\Sigma g)=\left[ (1\otimes\varphi)\circ\text{Tr}_{A}+(\varphi\otimes 1)\circ\text{Tr}_{A^{-1}} \right] \left\{ \J\D \Sigma g - \log(1+\J\D\Sigma g)\right\}.
	\end{equation*}
Furthermore, the function $Q$ satisfies the local Lipschitz condition on $\left\{g\in\mathscr{P}^{(R,\sigma)}_{c.s.}\colon \|g\|_{R,\sigma} < R^2/2\right\}$
	\begin{align*}
		\left\| Q(\Sigma g)-Q(\Sigma f)\right\|_{R,\sigma} \leq \| f-g\|_{R,\sigma} \frac{2\|A\|}{R^2} \left( \frac{1}{\left(1-\frac{2\|f\|_{R,\sigma}}{R^2}\right)\left(1-\frac{2\|g\|_{R,\sigma}}{R^2}\right)} - 1\right),
	\end{align*}
and the bound
	\begin{align*}
		\|Q(\Sigma g)\|_{R,\sigma} \leq \frac{ 4\|A\| \|g\|_{R,\sigma}^2}{R^4-2R^2 \|g\|_{R,\sigma}}.
	\end{align*}
\end{lem}
\begin{proof}
Let $\kappa=R^2/2$ and $\lambda = \|g\|_{R,\sigma}$. From Lemma \ref{Q_m} we know $\|Q_{m+2}(\Sigma g)\|_{R,\sigma} \leq  2\|A\| \left(\frac{\lambda}{\kappa}\right)^{m+2}$. Since $\lambda<\kappa$, the series defining $Q$ converges. The functional calculus equality then follows from $\log(1+x)=-\sum_{m\geq 1} \frac{(-x)^m}{m}$. Finally, since $m+2\geq 2$ in our series we obtain
	\begin{align*}
		\| Q(\Sigma g) - Q(\Sigma f)\|_{R,\sigma} &\leq \sum_{m\geq 0} \frac{1}{m+2} \| Q_{m+2}(\Sigma g) - Q_{m+2}(\Sigma f)\|_{R,\sigma}\\
								&\leq \|f-g\|_{R,\sigma} \|A\| \sum_{m\geq 0} \sum_{k=0}^{m+1} \kappa^{-m-2}\|f\|_{R,\sigma}^{m-k+1}\|g\|_{R,\sigma}^k\\
								&\leq \|f -g\|_{R,\sigma} \frac{\|A\| }{\kappa} \left( \sum_{l\geq 0}\sum_{k\geq 0} \kappa^{-l}\|f\|_{R,\sigma}^l \kappa^{-k} \|g\|_{R,\sigma}^k - 1\right),
	\end{align*}
where we have written $m=l+k-1$ which is non-negative so long as $l$ and $k$ are not both zero. Using $\|f\|_{R,\sigma}, \|g\|_{R,\sigma} < \kappa$ we see that
	\begin{align*}
		\| Q(\Sigma g) - Q(\Sigma f)\|_{R,\sigma} \leq \| f- g\|_{R,\sigma} \frac{2\|A\|}{R^2}\left( \frac{1}{\left(1-\frac{2\|f\|_{R,\sigma}}{R^2}\right)\left(1-\frac{2\|g\|_{R,\sigma}}{R^2}\right)} - 1\right).
	\end{align*}
Setting $f=0$ yields the bound
	\begin{align*}
		\|Q(\Sigma g)\|_{R,\sigma}\leq \|g\|_{R,\sigma} \frac{2\|A\|}{R^2} \frac{2\|g\|_{R,\sigma}}{R^2-2\|g\|_{R,\sigma}} = \frac{ 4\|A\| \|g\|_{R,\sigma}^2}{R^4-2R^2 \|g\|_{R,\sigma}},
	\end{align*}
as claimed.
\end{proof}

The proof of the following lemma is purely computational and left to the reader.

\begin{lem}
If $f=\D g$ for $g\in \mathscr{P}_\varphi^{(R,\sigma)}$ then 
	\begin{align}\label{eigenvalue_f}
		A^{-1}\# \sigma_{-i}(f)=f.
	\end{align}
Moreover, if $g=g^*$ then
	\begin{align}\label{dot_product_to_vector_product}
		\D\left( \frac{1}{2}\J_\sigma X^{-1}\# f\# f\right)=\J_\sigma f\# \J_\sigma X^{-1}\# f=\J f\# f.
	\end{align}
\end{lem}

\begin{lem}\label{dot_product_in_centralizer}
Suppose $f^{(i)}=\D\Sigma g_i$ with $g_i\in \mathscr{P}_{c.s.}$ for $i=1,2$. Then $(1+A)\# f^{(1)}\# f^{(2)}\in\mathscr{P}_\varphi$. Furthermore,
	\begin{align*}
		\left\|(1+A)\# f^{(1)}\# f^{(2)} \right\|_{R,\sigma} \leq \frac{2N\|A\|}{R^2} \|g_1\|_{R,\sigma}\|g_2\|_{R,\sigma}.
	\end{align*}
\end{lem}
\begin{proof}
From (\ref{eigenvalue_f}) it is easy to see that $(1+A)\# f^{(1)}\# f^{(2)}\in\mathscr{P}_\varphi$. Now, write $g_1=\sum_{m=1}^\infty \sum_{|\ul{i}|=m} c_1(\ul{i}) X_{\ul{i}}$ and $g_2 = \sum_{n=1}^\infty \sum_{|\ul{j}|=n}c_2(\ul{j})X_{\ul{j}}$. Then (\ref{cyclic_derivative_of_cyclically_symmetric}) implies
	\begin{align*}
		f_j^{(1)}=\sum_{m=1}^\infty \sum_{\substack{|\ul{i}|=m-1\\ a\in\{1,\ldots, N\}}}\alpha_{ja} c_1(\ul{i}\cdot a) X_{\ul{i}}\qquad\text{ and }\qquad f_i^{(2)} = \sum_{n=1}^\infty \sum_{\substack{|\ul{j}|=n-1\\ b\in\{1,\ldots, N\}}} \alpha_{ib} c_2(\ul{j}\cdot b) X_{\ul{j}}.
	\end{align*}
Hence
	\begin{align*}
		(1+A)\# f^{(1)}\# f^{(2)} =\sum_{m,n=1}^\infty \sum_{i,j=1}^N [1+A]_{ij}  \sum_{a,b=1}^N \sum_{\substack{|\ul{i}|=m-1\\ |\ul{j}|=n-1}} \alpha_{ja}\alpha_{ib} c_1(\ul{i}\cdot a)c_2(\ul{j}\cdot b) X_{\ul{i}} X_{\ul{j}}.
	\end{align*}
It suffices to bound $\|\rho^k(\cdot )\|_R$ for $k\in\{-m-n+1,\ldots, 0\}$. First, for $k=0$ we simply have
	\begin{align*}
		\| (1+A)\#f^{(1)}\# f^{(2)}\|_R &\leq \sum_{i,j=1}^N | [1+A]_{ij}| \sum_{m,n=1}^\infty \sum_{\substack{|\ul{i}|=m-1,a\\ |\ul{j}|=n-1, b}} |c_1\left(\ul{i}\cdot a\right) c_2\left(\ul{j}\cdot b\right)| R^{m+n-2}\\
							&\leq N(1+\|A\|)\frac{1}{R^2} \left(\sum_{m=1}^\infty\sum_{\ul{i},a} |c_1(\ul{i}\cdot a)| R^m\right)\left(\sum_{n=1}^\infty \sum_{\ul{j},b} |c_2(\ul{j}\cdot b) R^n\right)\\
							&\leq \frac{2N\|A\|}{R^2}\|g_1\|_{R,\sigma}\|g_2\|_{R,\sigma}.
	\end{align*}
For $-m+1\leq k\leq -1$, we further fix $i,j,a,b$. Then using (\ref{cyclically_symmetric_coefficients_negative}) we have
	\begin{align*}
		\sum_{\substack{|\ul{i}|=m-1\\ |\ul{j}|=n-1}}  c_1(\ul{i}\cdot a)c_2(\ul{j}\cdot b)\rho^k\left( X_{\ul{i}} X_{\ul{j}}\right)=\sum_{\substack{ |\ul{\hat{l}}|=k\\ |\ul{i}|=m-k-1\\ |\ul{j}=n-1}}  c_2(\ul{j}\cdot b) c_1\left(\ul{i}\cdot a\cdot\ul{\hat{l}}\right) X_{\ul{j}} X_{\ul{\hat{l}}}.
	\end{align*}
Thus
	\begin{align*}
		\sum_{m,n=1}^\infty &\left\| \sum_{i,j=1}^N [1+A]_{ij}  \sum_{a,b=1}^N \sum_{\substack{|\ul{i}|=m-1\\ |\ul{j}|=n-1}} \alpha_{ja}\alpha_{ib} c_1(\ul{i}\cdot a)c_2(\ul{j}\cdot b)\rho^k\left( X_{\ul{i}} X_{\ul{j}}\right)\right\|_R\\
				 &\leq \sum_{m,n=1}^\infty \sum_{i,j=1}^N |[1+A]_{ij}| \sum_{\substack{ |\ul{i}|=m\\ |\ul{j}|=n}} |c_1(\ul{i}) c_2(\ul{j}) | R^{n+m-2}\\
				 &\leq N(1+\|A\|) \frac{1}{R^2}\left(\sum_{m=1}^\infty\sum_{\ul{i},a} |c_1(\ul{i}\cdot a)| R^m\right)\left(\sum_{n=1}^\infty \sum_{\ul{j},b} |c_2(\ul{j}\cdot b) R^n\right)\\
							&\leq \frac{2N\|A\|}{R^2}\|g_1\|_{R,\sigma}\|g_2\|_{R,\sigma}.
	\end{align*}
The cases for $-m-n+1\leq k \leq -m$ are similar after using $\sigma_i(g_1)=g_1$. Thus the claimed bound holds.
\end{proof}

\begin{cor}\label{F}
Assume $R\geq4$. Let $g\in \mathscr{P}^{(R,\sigma)}_{c.s.}$ and assume that $\|g\|_{R,\sigma}< R^2/2$. Let $S\geq R+\|g\|_{R,\sigma}$ and let $W\in \mathscr{P}^{(S)}_{c.s.}$. Let
	\begin{align*}
		F(g)&= - W(X+\D\Sigma g) -   \frac{1}{4}\left\{(1+A)\#\D\Sigma g\right\} \#\D\Sigma g + \left[(1\otimes\varphi)\circ\text{Tr}_{A}+(\varphi\otimes 1)\circ\text{Tr}_{A^{-1}}\right]\circ\log(1+\J\D\Sigma g)\\
			&=- W(X+\D\Sigma g) -   \frac{1}{4}\left\{(1+A)\#\D\Sigma g\right\} \#\D\Sigma g  +  \left[(1\otimes\varphi)\circ\text{Tr}_{A}+(\varphi\otimes 1)\circ\text{Tr}_{A^{-1}}\right](\J\D\Sigma g)-Q(\Sigma g).
	\end{align*}
Then $F(g)$ is a well-defined function from $\mathscr{P}^{(R,\sigma)}_{c.s.}$ to $\mathscr{P}^{(R,\sigma)}_{\varphi}$. Moreover, $g\mapsto F(g)$ is locally Lipschitz on $\{g\colon \|g\|_{R,\sigma}< R^2/2\}$:
	\begin{align*}
		\| F(g)&-F(f)\|_{R,\sigma}\leq\\
			&\leq \| f - g\|_{R,\sigma} \left\{ \frac{2\|A\|}{R^2} \left( \frac{1}{\left(1-\frac{2\|f\|_{R,\sigma}}{R^2}\right)\left(1-\frac{2\|g\|_{R,\sigma}}{R^2}\right)} + 1+ \frac{N}{4}(\|g\|_{R,\sigma}+\|f\|_{R,\sigma}) \right) + \sum_{j=1}^N\| \delta_j(W)\|_{S\otimes_\pi S}\right\},
	\end{align*}
and bounded:
	\begin{align*}
		\| F(g)\|_{R,\sigma} \leq \| g\|_{R,\sigma} \left\{ \frac{2\|A\|}{R^2-2\|g\|_{R,\sigma}} + \frac{2\|A\|}{R^2}+ \frac{N\|A\|}{2R^2}\|g\|_{R,\sigma} + \sum_{j=1}^N\| \delta_j(W)\|_{S\otimes_\pi S}\right\}+\|W\|_{R,\sigma}.
	\end{align*}
In particular, if
	\begin{align}\label{contractive_data}
		\left\{\begin{array}{l} R\geq4\sqrt{\|A\|},\qquad 0<\rho\leq 1\\
					    \|W\|_{R,\sigma}<\frac{\rho}{2N}\\
					    \sum_j\|\delta_j(W)\|_{(R+\rho)\otimes_{\pi}(R+\rho)}<\frac{1}{8}\end{array}\right.,
	\end{align}
then $F$ takes the ball $E_1=\left\{g\in\mathscr{P}^{(R,\sigma)}_{c.s.}\colon \|g\|_{R,\sigma}<\frac{\rho}{N}\right\}$ to the ball $E_2=\left\{ g\in\mathscr{P}^{(R,\sigma)}_\varphi\colon \|g\|_{R,\sigma}<\frac{\rho}{N}\right\}$ and is uniformly contractive with constant $\lambda\leq \frac{1}{2}$ on $E_1$.
\end{cor}
\begin{proof}
From Lemma \ref{cyclic_derivative_bounded} we know $S\geq R+ \|g\|_{R,\sigma}> R+\|\D\Sigma g\|_R$, and thus $W(X+\D\Sigma g)$ is well-defined as an element of $\mathscr{P}^{(R)}$. We claim that in fact, $W(X+\D\Sigma g)\in \mathscr{P}_{\varphi}^{(R,\sigma)}$. Indeed, first note that from (\ref{eigenvalue_f}) we have $\sigma_{-i}(X+\D\Sigma g)=A\#(X+\D\Sigma g)$. Hence
	\begin{equation*}
		\sigma_{-i}\left( W(X+\D\Sigma g)\right) = W( \sigma_{-i}(X+\D \Sigma g)) = W(A\#(X+\D\Sigma g))=\sigma_{-i}(W)(X+\D\Sigma g)=W(X+\D\Sigma g),
	\end{equation*}
where we have used (\ref{centralizer_coefficients}). Then, using (\ref{cyclically_symmetric_coefficients_positive}) it is not hard to see that $\|W(X+\D\Sigma g)\|_{R,\sigma}=\|W(X+\D\Sigma g)\|_{R}<\infty$. That $F(g)\in \mathscr{P}_\varphi^{(R,\sigma)}$ follows from Lemmas \ref{centralizer_commutes_with_A} and \ref{dot_product_in_centralizer} and $W(X+\D\Sigma g)\in\mathscr{P}^{(R,\sigma)}_\varphi$. \par
Using Lemmas \ref{Q_m} and \ref{Q_2} we have
	\begin{align*}
		 \| Q(\Sigma g)-Q(\Sigma f)\|_{R,\sigma} &+ \left\| \left[(1\otimes\varphi)\circ\text{Tr}_{A}+(\varphi\otimes 1)\circ\text{Tr}_{A^{-1}}\right](J\D\Sigma (g-f))\right\|_{R,\sigma}\\
					        \leq&  \|f-g\|_{R,\sigma}\frac{2\|A\|}{R^2} \left( \frac{1}{\left(1-\frac{2\|f\|_{R,\sigma}}{R^2}\right)\left(1-\frac{2\|g\|_{R,\sigma}}{R^2}\right)} - 1\right) +\|A\| \frac{2^2}{R^2} \|f-g\|_{R,\sigma} \\
						=& \|f-g\|_{R,\sigma}\frac{2\|A\|}{R^2} \left( \frac{1}{\left(1-\frac{2\|f\|_{R,\sigma}}{R^2}\right)\left(1-\frac{2\|g\|_{R,\sigma}}{R^2}\right)} + 1\right).
	\end{align*}\par
The following estimate is essentially identical to the one produced in Corollary 3.12 in \cite{SG11}:
	\begin{align*}
		\|W(X+\D\Sigma g) - W(X+\D\Sigma f)\|_{R,\sigma} &= \|W(X+\D\Sigma g) - W(X+\D\Sigma f)\|_R\\
										&\leq \sum_j\| \delta_j( W)\|_{S\otimes_\pi S} \|\D\Sigma g- \D\Sigma f\|_R\\
										& \leq \sum_j\|\delta_j(W)\|_{S\otimes_\pi S} \|f-g\|_{R,\sigma}.
	\end{align*}\par
Finally, from Corollary \ref{dot_product_in_centralizer}, we know
	\begin{align*}
		\frac{1}{4}\| \left\{(1+A)\#\D\Sigma g\right\}& \#\D\Sigma g - \left\{(1+A)\#\D\Sigma f\right\} \#\D\Sigma f \|_{R,\sigma}\\
					 &\leq \frac{1}{4} \left\| \left\{(1+A)\#\D\Sigma (g-f)\right\} \#\D\Sigma g\right\|_{R,\sigma} + \frac{1}{4} \left\| \left\{(1+A)\#\D\Sigma f\right\} \#\D\Sigma (g-f)\right\|_{R,\sigma}\\
					 &\leq \frac{1}{4} \frac{2N\|A\|}{R^2}\| f-g\|_{R\sigma}\|g\|_{R,\sigma} + \frac{1}{4} \frac{2N\|A\|}{R^2}\| f\|_{R\sigma}\|f-g\|_{R,\sigma}\\
					 &= \frac{N\|A\|}{2R^2}\|f-g\|_{R,\sigma} (\|g\|_{R,\sigma}+\|f\|_{R,\sigma}).
	\end{align*}
Combining the previous three estimates yields
	\begin{align*}
		\|F(f)&-F(g)\|_{R,\sigma} \\
			&\leq \| f - g\|_{R,\sigma} \left\{ \frac{2\|A\|}{R^2} \left( \frac{1}{\left(1-\frac{2\|f\|_{R,\sigma}}{R^2}\right)\left(1-\frac{2\|g\|_{R,\sigma}}{R^2}\right)} + 1+ \frac{N}{4}(\|g\|_{R,\sigma}+\|f\|_{R,\sigma} \right) + \sum_{j=1}^N\| \delta_j(W)\|_{S\otimes_\pi S}\right\},
	\end{align*}
as claimed. The estimate on $\|F(g)\|_{R,\sigma}$ then follows from the above and $F(0)=-W(X)$.\par
Now, suppose (\ref{contractive_data}) holds and let $f,g\in E_1$. Note that $R\geq4$ and $\|f\|_{R,\sigma},\|g\|_{R,\sigma}<\frac{1}{N}\leq 1$. Hence the Lipschitz property implies
	\begin{align*}
		\|F(f)-F(g)\|_{R,\sigma} \leq \|f-g\|_{R,\sigma}\left\{ \frac{1}{8}\left( \frac{64}{49}+1+\frac{1}{2}\right)+\frac{1}{8}\right\} =\|f-g\|_{R,\sigma} \left\{\frac{8}{49}+\frac{5}{16}\right\}< \frac{1}{2}\|f-g\|_{R,\sigma}.
	\end{align*}
The bound on $F$ then implies
	\begin{align*}
		\|F(g)\|_{R,\sigma} \leq \frac{\rho}{N}\left\{ \frac{1}{7}+\frac{1}{8}+\frac{1}{32}+\frac{1}{8}\right\}+\frac{\rho}{2N}<\frac{\rho}{2N}+\frac{\rho}{2N}=\frac{\rho}{N},
	\end{align*}
and so $F$ maps $E_1$ into $E_2$.
\end{proof}


\subsection{Existence of $g$.}\label{existence_of_g}

\begin{prop}\label{g_exists}
Assume that for some $R\geq4\sqrt{\|A\|}$ and some $0<\rho\leq 1$, $W\in\mathscr{P}_{c.s.}^{(R+\rho,\sigma)}\subset\mathscr{P}_{c.s.}^{(R,\sigma)}$ and that
	\begin{align}\label{simple_contractive_data}
			\left\{\begin{array}{l}\|W\|_{R,\sigma}<\frac{\rho}{2N}\\
					   	 \sum_j\|\delta_j(W)\|_{(R+\rho)\otimes_{\pi}(R+\rho)}<\frac{1}{8}\end{array}\right..
	\end{align}
Then there exists $\hat{g}$ and $g=\Sigma \hat{g}$ with the following properties:
	\begin{enumerate}
	\item[(i)] $\hat{g},g\in\mathscr{P}_{c.s.}^{(R,\sigma)}$
	
	\item[(ii)] $\hat{g}$ satisfies the equation $\hat{g}=\mathscr{S}\Pi F(\hat{g})$
	
	\item[(iii)] $g$ satisfies the equation
			\begin{align}\label{hatless_g}
				\mathscr{N} g= \mathscr{S}\Pi\left[ -W(X+\D g) \vphantom{\frac{1}{4}}\right.&- \frac{1}{4}\left\{(1+A)\#\D g\right\}\# \D g\notag\\
														&\left. \vphantom{\frac{1}{4}}+ \left[(1\otimes\varphi)\circ\text{Tr}_{A}+(\varphi\otimes 1)\circ\text{Tr}_{A^{-1}}\right]\circ\log(1+\J\D\Sigma g) \right],
			\end{align}
	or, equivalently,
			\begin{align}\label{anti-cyclic_derivative}
				\mathscr{S}\Pi\left[ (1\otimes\varphi)\circ\text{Tr}_A+\right.&\left.(\varphi\otimes 1)\circ\text{Tr}_{A^{-1}}\right](\J\D g) -\mathscr{N}g\notag\\
						&= \mathscr{S}\Pi\left\{ W(X+\D g)  + Q(g) + \frac{1}{4}\left\{(1+A)\# \D g\right\}\# \D g\right\}.
			\end{align}
	
	\item[(iv)] If $W=W^*$, then $\hat{g}=\hat{g}^*$ and $g=g^*$.
	
	\item[(v)] $\hat{g}$ and $g$ depend analytically on $W$, in the following sense: if the maps $\beta\mapsto W_\beta$ are analytic, then also the maps $\beta\mapsto \hat{g}(\beta)$ and $\beta\mapsto g(\beta)$ are analytic, and $g\rightarrow 0$ if $\|W\|_{R,\sigma}\rightarrow 0$.
	\end{enumerate}
\end{prop}
\begin{proof}
We remark that Equation (\ref{hatless_g}) is equivalent to
	\begin{align*}
		\mathscr{N}g=\mathscr{S}\Pi F(\mathscr{N} g),
	\end{align*}
with $F$ as in Corollary \ref{F}. Under our current assumptions, the hypotheses of the corollary are satisfied. We set $\hat{g}_0=W(X_1,\ldots, X_N)\in E_1$ and for each $k\in\N$,
	\begin{equation*}
		\hat{g}_k:=\mathscr{S}\Pi F(\hat{g}_{k-1}).
	\end{equation*}
Since $F$ maps into $\mathscr{P}_\varphi^{(R,\sigma)}$, on which $\mathscr{S}\Pi$ is a linear contraction, and $\mathscr{S}\Pi E_2\subset E_1$, the final part of Corollary \ref{F} implies that $\mathscr{S}\Pi F$ is uniformly contractive with constant $\frac{1}{2}$ on $E_1$ and takes $E_1$ to itself. Thus $\hat{g}_k\in E_1$ for all $k$ and
	\begin{align*}
		\| \hat{g}_k - \hat{g}_{k-1}\|_{R,\sigma}=\left\| \mathscr{S}\Pi F(\hat{g}_{k-1}) - \mathscr{S}\Pi F(\hat{g}_{k-2})\right\|_{R,\sigma} < \frac{1}{2} \| \hat{g}_{k-1} - \hat{g}_{k-2}\|_{R,\sigma},
	\end{align*}
implying that $\hat{g}_k\rightarrow \hat{g}$ in $\|\cdot\|_{R,\sigma}$, with $\hat{g}$ a fixed point of $\mathscr{S}\Pi F$. We note that $\hat{g}\neq 0$ as $\mathscr{S}\Pi F(0)=\mathscr{S}\Pi (W)=W\neq 0$. Since $\hat{g}\in \mathscr{P}_{c.s.}^{(R,\sigma)}$, we also have $g:=\Sigma \hat{g} \in \mathscr{P}_{c.s.}^{(R,\sigma)}$. This proves (i) and (ii), and (iii) simply follows from the relation $\hat{g}=\mathscr{N}g$ and the definition of $F$.\par
It is not hard to see that for $h=h^*$, $\mathscr{S}\Pi F( h)^*=\mathscr{S}\Pi F(h)$. Hence if we assume $\hat{g}_0=W$ is self-adjoint, then each successive $\hat{g}_k$ will be self-adjoint. Consequently so will their limit $\hat{g}$ since $\|\cdot\|_R$ (which is invariant under $*$) is dominated by $\|\cdot \|_{R,\sigma}$. It follows that $g=\Sigma \hat{g}$ is self-adjoint as well.\par
Assume $\beta\mapsto W_\beta$ is analytic. Then each iterate $\hat{g}_k(\beta)$ is clearly analytic as well, and the convergence to $\hat{g}(\beta)$ is uniform on any compact disk inside $|\beta|<\beta_0$. Thus the Cauchy integral formula implies the limit $\hat{g}(\beta)$ is analytic as well, and clearly so is $g(\beta)=\Sigma \hat{g}(\beta)$.\par
Finally, we remark that $\|g\|_{R,\sigma}$ is bounded by $\|W\|_{R,\sigma}$. Indeed,
	\begin{align*}
		\|\hat{g}- W\|_{R,\sigma}=\| \hat{g}-\hat{g}_0\|_{R,\sigma} \leq 2\|\hat{g}_1-\hat{g}_0\|_{R,\sigma} \leq 2\left( \left[\|\hat{g}_0\|_{R,\sigma}\left\{\frac{1}{2}\right\} + \|W\|_{R,\sigma}\right] + \|\hat{g}_0\|_{R\sigma} \right) = 5\|W\|_{R,\sigma},
	\end{align*}
or $\|\hat{g}\|_{R,\sigma}\leq 6\|W\|_{R,\sigma}$. Since $\|g\|_{R,\sigma} = \|\Sigma\hat{g}\|_{R,\sigma}\leq \|\hat{g}\|_{R,\sigma}$, it follows that $g\mapsto 0$ as $\|W\|_{R,\sigma} \mapsto 0$.
\end{proof}

\begin{thm}\label{existence_of_f}
Let $R'>R\geq 4\sqrt{\|A\|}$. Then there exists a constant $C>0$ depending only on $R$, $R'$, and $N$ so that whenever $W=W^*\in \mathscr{P}_{c.s.}^{(R'+1)}$ satisfies $\|W\|_{R'+1,\sigma}<C$, there exists $f\in\mathscr{P}^{(R)}$ which satisfies Equation (\ref{S-D_Y_2}). In addition, $f=\D g$ for $g\in\mathscr{P}^{(R,\sigma)}_{c.s.}$. The solution $f=f_W$ satisfies $\|f_W\|_R\rightarrow 0$ as $\|W\|_{R'+1,\sigma}\rightarrow 0$. Moreover, if $W_\beta$ is a family which is analytic in $\beta$ then also the solutions $f_{W_\beta}$ are analytic in $\beta$.
\end{thm}
\begin{proof}
Fix $S\in (R,R')$. Using the bounds in the proof of Theorem 3.15 in \cite{SG11} we have
	\begin{align*}
		\sum_{j=1}^N\|\delta_j(W)\|_{(S+1)\otimes_\pi (S+1)} \leq c(S+1,R'+1)\|W\|_{R'+1},
	\end{align*}
where
	\begin{align*}
		c(S,R)=\sup_{\alpha\geq 1} \alpha S^{-1}(R/S)^{-\alpha}.
	\end{align*}
Also, $S<R'+1$ implies $\|W\|_{S,\sigma}\leq \|W\|_{R'+1,\sigma}$. Hence, by choosing $C>0$ sufficiently small, $\|W\|_{R'+1,\sigma}<C$ will imply the hypothesis of Proposition \ref{g_exists} are satisfied with $\rho=1$ and $R$ replaced with $S$. Thus there exists $g=g^*\in \mathscr{P}_{c.s.}^{(S,\sigma)}$ satisfying (\ref{anti-cyclic_derivative}). Let $f=\D g$, then from Lemma \ref{cyclic_derivative_bounded} we know $f \in \left(\mathscr{P}^{(R)}\right)^N$. Also, using the bounds from the proof of Theorem 3.15 in \cite{SG11} again we have
	\begin{align*}
		\| \J f\|_{R\otimes_\pi R} \leq c'(R,S)\|g\|_{S}=c'(R,S)\|g\|_{S,\sigma},
	\end{align*}
where
	\begin{align*}
		c'(R,S)=\sup_{\alpha\geq 1} \alpha^2 R^{-2}(S/R)^{-\alpha}.
	\end{align*}
Hence by the proof of Proposition \ref{g_exists}.(v) we can (by possibly choosing a smaller $C$) assume $\|\J f\|_{R\otimes_\pi R}<1$. Also, it is clear that $g\in \mathscr{P}_{c.s.}^{(R,\sigma)}\supset \mathscr{P}_{c.s.}^{(S,\sigma)}$.\par
Recall from Lemma \ref{D_of_S} that $\D\mathscr{S}\Pi=\D$ on $\mathscr{P}_\varphi^{(S,\sigma)}$. Hence applying $\D$ to both sides of (\ref{anti-cyclic_derivative}) yields
	\begin{align*}
		\D\left\{ \left[ (\varphi\otimes 1)\circ\text{Tr}_{A^{-1}} \right.\right.+&\left.\left. (1\otimes\varphi)\circ \text{Tr}_A\right](\J\D g) - \mathscr{N} g\right\} \\
										&= \D(W(X+\D g))+\D Q(g)+ \D\left(\frac{1}{4}\left\{(1+A)\# \D g\right\}\# \D g\right).
	\end{align*}
The final term is equivalent to
	\begin{align*}
		\D\left(\frac{1}{4}\left\{(1+A)\# \D g\right\}\# \D g\right) = \D\left(\frac{1}{2} \J_\sigma X^{-1}\# f\# f\right) = \J f\# f=\J_\sigma f\#\J_\sigma X^{-1}\# f,
	\end{align*}
where we have used (\ref{dot_product_to_vector_product}). Thus $f=\D g$ satisfies Equation (\ref{S-D_Y_3}) which, according to Lemma \ref{S-D_Y_2=3} is equivalent to Equation (\ref{S-D_Y_2}).\par
The final statements follow from Lemma \ref{cyclic_derivative_bounded} and Proposition \ref{g_exists}.(v).
\end{proof}


\subsection{Summary of results.}\label{summary}

We aggregate the results of this section in the following theorem.

\begin{thm}\label{thm_summary}
Let $(M,\varphi)=(M_0,\varphi_{V_0})$ be a free Araki-Woods factor with free quasi-free state $\varphi$ corresponding $A$, and generators $X_1,\ldots, X_N\in M$ so that the matrix form of $A$ with respect to the basis $\{X_j\Omega\}_{j=1}^N$ is given by (\ref{matrix_form_A}) and (\ref{matrix_form_A_2}). Let $R'>R\geq4\sqrt{\|A\|}$. Then there exists a constant $C>0$ depending only on $R$, $R'$, and $N$ so that whenever $W=W^*\in\mathscr{P}_{c.s.}^{(R'+1,\sigma)}$ satisfies $\|W\|_{R'+1,\sigma}<C$, there exists $G\in \mathscr{P}_{c.s.}^{(R,\sigma)}$ so that
	\begin{align*}
		(Y_1,\ldots, Y_N)=(\D_1G,\cdots, \D_NG) \in \mathscr{P}^{(R)}
	\end{align*}
has the law $\varphi_V$, $V=\frac{1}{2}\sum_{j,k=1}^N \left[\frac{1+A}{2}\right]_{jk}X_kX_j + W$, which is the unique free Gibbs state with potential $V$.\par
If $R'>R\|A\|^\frac{1}{4}$ then the transport can be taken to be monotone: $(\sigma_\frac{i}{2}\otimes 1)(\J_\sigma \D G) \geq 0$ as an operator on $L^2(\mathscr{P}\otimes\mathscr{P}^{op},\varphi\otimes\varphi^{op})^N$.\par
In particular, there are state-preserving injections $C^*(\varphi_V)\subset C^*(\varphi_{V_0})$ and $W^*(\varphi_V)\subset W^*(\varphi_{V_0})$.\par
If the map $\beta\mapsto W_\beta$ is analytic, then $Y_1,\ldots, Y_n$ are also analytic in $\beta$. Furthermore, $\|Y_j-X_j\|_{R}$ vanishes as $\|W\|_{R'+1,\sigma}$ goes to zero.
\end{thm}
\begin{proof}
Note for $Y_j=X_j+f_j$ we have $\|Y_j\|\leq 2 +\|f_j\|_R$. By requiring $C$ be small enough so that $\|f_j\|_R\leq 1$, we have that
	\begin{align*}
		|\varphi(Y_{\ul{j}})|\leq 3^{|\ul{j}|}.
	\end{align*}
So by Theorem \ref{Gibbs_state_unique}, and further shrinking $C$ if necessary, we see that $\varphi_Y$ is the unique free Gibbs state with potential potential $V$.
The only remaining part of this theorem not covered by Theorem \ref{existence_of_f} is the positivity of $(\sigma_{i/2}\otimes 1)(\J_\sigma f)$, so we merely verify this condition when $R\rq{}>R\|A\|^\frac{1}{4}$.\par
Recall from Lemma \ref{change_of_variables}.(iv),
	\begin{align*}
		(\sigma_\frac{i}{2}\otimes 1)(\J_\sigma f) = A^\frac{1}{4}\#(\sigma_\frac{i}{4}\otimes \sigma_{-\frac{i}{4}})(\J_\sigma f)\# A^{-\frac{1}{4}}.
	\end{align*}
Hence if $S\rq{}=\|A\|^\frac{1}{4} R$ then
	\begin{align*}
		\| (\sigma_\frac{i}{2}\otimes 1)(\J_\sigma f)\|_{R\otimes_\pi R} &\leq \|A^\frac{1}{4}\|^2  \| (\sigma_\frac{i}{4}\otimes \sigma_{-\frac{i}{4}})(\J f)\|_{R\otimes_\pi R} \|\J_\sigma X\|_{R\otimes_\pi R}\\
				& \leq \|A^\frac{1}{4}\|^2 \| \J f\|_{S\rq{}\otimes_\pi S\rq{}} \|\J_\sigma X\|_{R\otimes_\pi R}.
	\end{align*}
Thus in the proof of Theorem \ref{existence_of_f} we can choose $S\in (S\rq{}, R\rq{})$ so that $\|\J f\|_{S\rq{}\otimes_\pi S\rq{}} \leq c\rq{}(S\rq{},S)\|g\|_{S,\sigma}$. In particular, we can make $\|\J f\|_{S\rq{}\otimes_\pi S\rq{}} < \|A^\frac{1}{4}\|^{-2}$ so that
	\begin{align*}
		\|(\sigma_\frac{i}{2}\otimes 1)(\J_\sigma f)\|_{R\otimes_\pi R} < \|\J_\sigma X\|_{R\otimes_\pi R}.
	\end{align*}
Noting that $(\sigma_\frac{i}{2}\otimes 1)(\J_\sigma Y)=\J_\sigma X+ (\sigma_\frac{i}{2}\otimes 1)(\J_\sigma f)$, $\J_\sigma X\geq 0$, and $(\sigma_\frac{i}{2}\otimes 1)(\J_\sigma f)^*=(\sigma_\frac{i}{2}\otimes 1)(\J_\sigma f)$ (via Lemma \ref{change_of_variables}.(iii)) we have that $(\sigma_\frac{i}{2}\otimes 1)(\J_\sigma Y)\geq 0$.
\end{proof}

By shrinking the constant further if needed, we can use Lemma \ref{invertible_power_series} to turn the state-preserving injections into isomorphisms:

\begin{cor}\label{iso_cor}
Let $(M,\varphi)=(M_0,\varphi_{V_0})$ be a free Araki-Woods factor with free quasi-free state $\varphi$ corresponding $A$, and generators $X_1,\ldots, X_N\in M$ so that the matrix form of $A$ with respect to the basis $\{X_j\Omega\}_{j=1}^N$ is given by (\ref{matrix_form_A}) and (\ref{matrix_form_A_2}). Let $R'>R\geq4\sqrt{\|A\|}$. Then there exists $C>0$ depending only on $R$, $R'$, and $N$ so that whenever $W=W^*\in\mathscr{P}_{c.s.}^{(R'+1,\sigma)}$ satisfies $\|W\|_{R'+1,\sigma}<C$, there exists $G\in \mathscr{P}_{c.s.}^{(R,\sigma)}$ so that:
	\begin{enumerate}
	\item[(1)] if we set $Y_j=\D_jG$, then $Y_1,\ldots, Y_N\in \mathscr{P}^{(R)}$ has law $\varphi_V$, with $V=\frac{1}{2}\sum_{j,k=1}^N \left[\frac{1+A}{2}\right]_{jk}X_kX_j + W$;
	
	\item[(2)] $X_j=H_j(Y_1,\ldots, Y_N)$ for some $H_j\in \mathscr{P}^{(R)}$; and
	
	\item[(3)] if $R'>R\|A\|^\frac{1}{4}$ then $(\sigma_\frac{i}{2}\otimes 1)(\J_\sigma \D G)\geq 0$ as an operator on $L^2(\mathscr{P}\otimes\mathscr{P}^{op})^N$.
	\end{enumerate}
	
In particular there are state-preserving isomorphisms
	\begin{align*}
		C^*(\varphi_V)\cong \Gamma(\H_\R, U_t),\qquad W^*(\varphi_V)\cong \Gamma(\H_\R, U_t)''.
	\end{align*}
\end{cor}
\begin{proof}
By Theorem \ref{thm_summary}, it suffices to show the existence of $H=(H_1,\ldots, H_N)\in (\mathscr{P}^{(R)})^N$. From Theorem \ref{existence_of_f}, we know that $Y=X+f(X)$, and that $\|f\|_R\rightarrow 0$ as $\|W\|_{R'+1,\sigma}\rightarrow 0$. In fact, from Lemma \ref{cyclic_derivative_bounded} we know that $f\in (\mathscr{P}^{(S)})^N$ for any $S\in (R,R')$, and $\|f\|_S$ still tends to zero. Set $S=(R+R')/2$, then by shrinking the constant $C$ in the statement of the corollary further if necessary, we may assume that hypothesis of Lemma \ref{invertible_power_series} are satisfied. Thus we obtain the desired inverse mapping $H(Y)=X$.
\end{proof}


\section{Application to the $q$-deformed Araki-Woods algebras}\label{application}

We saw in Theorem \ref{free_Gibbs_is_free_quasi-free} that $\varphi_0$ is the free Gibbs state with potential
	\begin{align*}
		V_0=\frac{1}{2}\sum_{j,k=1}^N \left[\frac{1+A}{2}\right]_{jk}X_k^{(0)}X_j^{(0)}.
	\end{align*}
In this section we will show that for small $|q|$, $\varphi_q$ is the free Gibbs state with potential 
	\begin{align*}
		V=\frac{1}{2}\sum_{j,k=1}^N \left[\frac{1+A}{2}\right]_{jk}X_k^{(q)}X_j^{(q)}+W\in\mathscr{P}_{c..s.}^{(R,\sigma)},
	\end{align*}
and that $\|W\|_{R,\sigma}\rightarrow 0 $ as $|q|\rightarrow 0$. Hence it will follow from Corollary \ref{iso_cor} that $M_q\cong M_0$ for sufficiently small $|q|$. We now let $M=M_q$ for arbitrary (but fixed) $q\in (-1,1)$, with the usual notational simplifications.


\subsection{Invertibility of $\Xi_q$}

Let $\Psi\colon M\Omega\rightarrow M$ be the inverse of canonical embedding of $M$ into $\mc{F}_q(\H)$ via $x\mapsto x\Omega$ for $x\in M$, which we note is injective from the fact that $\Omega$ is separating. Hence for $\xi\in M\Omega$ we have that $\Psi(\xi)$ is the unique element in $M$ such that $\Psi(\xi)\Omega=\xi$. The uniqueness then implies the complex linearity of $\Psi$:  $\Psi(\sum_i \alpha_i\xi_i)=\sum_i \alpha_i\Psi(\xi_i)$. We also note that by the formulas (\ref{Tomita-Takesaki_formulas}) we have
	\begin{align}\label{Psi_delta}
		\Psi(S\xi)\Omega&=S\xi=S(\Psi(\xi)\Omega)=\Psi(\xi)^*\Omega;\qquad\text{ and}\notag\\
		\Psi(\Delta^{iz}\xi)\Omega &= \Delta^{iz}\xi = \Delta^{iz}\Psi(\xi)\Delta^{-iz}\Omega = \sigma_z(\Psi(\xi))\Omega,
	\end{align}
so that the uniqueness implies $\Psi(S\xi)=\Psi(\xi)^*$ and $\Psi(\Delta^{iz}\xi) = \sigma_z(\Psi(\xi))$.\par
Recall that $\Xi_q=\sum q^n P_n$, where $P_n\in HS(\mc{F}_q(\H))$ is the projection onto tensors of length $n$. We claim that (\ref{Psi_delta}) implies each $P_n$, when identified with an element in $L^2(M\bar{\otimes}M^{op},\varphi\otimes\varphi^{op})$, is fixed by $\sigma_{it}\otimes\sigma_{it}$ for all $t\in\R$. Indeed, fix $t\in\R$ and let $\{\xi_{\ul{i}}\}_{|\ul{i}|=n}$ be an orthonormal basis for $\H^{\otimes n}$. Then $P_n$ is identified with $\sum_{|\ul{i}|=n} \Psi(\xi_{\ul{i}})\otimes \Psi(\xi_{\ul{i}})^*$ since for $\eta\in\mc{F}_q(\H)$
	\begin{align*}
		\sum_{|\ul{i}|=n} \< \Psi(\xi_{\ul{i}})\Omega, \eta\>_{U,q} \Psi(\xi_{\ul{i}})\Omega = \sum_{|\ul{i}|=n} \<\xi_{\ul{i}},\eta\>_{U,q}\xi_{\ul{i}} = P_n\eta.
	\end{align*}
Now, using (\ref{Psi_delta}), we see that
	\begin{align*}
		(\sigma_{it}\otimes\sigma_{it})(P_n)= \sum_{|\ul{i}|=n} \Psi(\Delta^{t}\xi_{\ul{i}})\otimes \Psi(\Delta^{-t}\xi_{\ul{i}})^* = \sum_{|\ul{i}|=n} \Psi\left( \left(A^{-t}\right)^{\otimes n}\xi_{\ul{i}}\right) \otimes \Psi\left( \left( A^t\right)^{\otimes n}\xi_{\ul{i}}\right)^*.
	\end{align*}
Let $Q_n\in HS(\mc{F}_q(\H))$ be the element associated with $(\sigma_{it}\otimes\sigma_{it})(P_n)$. That is, for $\eta\in\mc{F}_q(\H)$ we have
	\begin{align*}
		Q_n\eta= \sum_{|\ul{i}|=n} \< \left(A^{t}\right)^{\otimes n}\xi_{\ul{i}},\eta\>_{U,q} \left(A^{-t}\right)^{\otimes n}\xi_{\ul{i}},
	\end{align*}
and so
	\begin{align*}
		\<\left(A^{t}\right)^{\otimes n}\xi_{\ul{j}}, Q_n\eta\>_{U,q} &= \sum_{|\ul{i}|=n} \< \left(A^{t}\right)^{\otimes n}\xi_{\ul{i}},\eta\>_{U,q} \< \left(A^{t}\right)^{\otimes n}\xi_{\ul{j}}, \left(A^{-t}\right)^{\otimes n}\xi_{\ul{i}}\>_{U,q}\\
					&=\sum_{|\ul{i}|=n} \< \left(A^{t}\right)^{\otimes n}\xi_{\ul{i}},\eta\>_{U,q} \<\xi_{\ul{j}}, \xi_{\ul{i}}\>_{U,q} = \< \left(A^{t}\right)^{\otimes n}\xi_{\ul{j}},\eta\>_{U,q}=\< \left(A^{t}\right)^{\otimes n}\xi_{\ul{j}},P_n\eta\>_{U,q}.
	\end{align*}
From Lemma 1.2 of \cite{H}, $A^t>0$ implies $\left(A^t\right)^{\otimes n}>0$. Thus $\left\{ \left(A^{t}\right)^{\otimes n}\xi_{\ul{i}}\right\}_{|\ul{i}|=n}$ is a basis for $\H^{\otimes n}$ and hence $P_n=Q_n=(\sigma_{it}\otimes \sigma_{it})(P_n)$ as claimed.\par
It follows that for any $t\in\R$ we have $(\sigma_{it}\otimes\sigma_{it})(\Xi_q)=\Xi_q$, and more generally 
	\begin{align}\label{Xi_q_sigma_invariant}
		(\sigma_{it}\otimes\sigma_{is})(\Xi_q)=(\sigma_{i(t-s)}\otimes 1)(\Xi_q)=(1\otimes\sigma_{i(s-t)})(\Xi_q)\qquad \forall t,s\in\R.
	\end{align}\par
We remind the reader that the norm $\|\cdot\|_{R\otimes_\pi R}$ is defined in Section \ref{projective_tensor_norm}. Denote the closure of $\mathscr{P}\otimes\mathscr{P}^{op}$ with respect to this norm by $\left(\mathscr{P}\otimes\mathscr{P}^{op}\right)^{(R)}$. We now prove an estimate analogous to those in Corollary 29 in \cite{D} for the non-tracial case.

\begin{prop}\label{Xi_invertible}
Let $R=\left(1+\frac{c}{2}\right)\frac{2}{1-|q|}>\|X_i\|$ for some $c>0$. Fix $t_0\in\R$, then for sufficiently small $|q|$ and all $|t|\leq |t_0|$, $(\sigma_{it}\otimes 1)(\Xi_q)\in \left(\mathscr{P}\otimes\mathscr{P}^{op}\right)^{(R)}$ with
	\begin{align*}
		\| (\sigma_{it}\otimes1)(\Xi_q) - 1\|_{R\otimes_\pi R} \leq \frac{\|A^{t}\|(3+c)^2(1+\|A\|)N^2|q|}{2- \left(4+ \|A^{t}\|(3+c)^2(1+\|A\|)N^2\right)|q|}=:\pi(q,N,A,t).
	\end{align*}
Moreover, $\pi(q,N,A,t)\rightarrow 0$ as $|q|\rightarrow 0$ and $\pi(q,N,A,s)\leq\pi(q,N,A,t)$ for $|s|\leq |t|$. Finally, for $\pi(q,N,A,t_0)<1$ and $|t|\leq |t_0|$, $(\sigma_{it}\otimes 1)(\Xi_q)$ is invertible with $(\sigma_{it}\otimes 1)(\Xi_q)^{-1}=(\sigma_{it}\otimes 1)(\Xi_q^{-1})\in\left(\mathscr{P}\otimes\mathscr{P}^{op}\right)^{(R)}$ and
	\begin{align*}
		\left\| (\sigma_{it}\otimes 1)(\Xi_q^{-1}) -1 \right\|_{R\otimes_\pi R} \leq \frac{\pi(q,N,A,t)}{1-\pi(q,N,A,t)}\longrightarrow 0\	\text{ as }|q|\rightarrow 0.
	\end{align*}
\end{prop}
\begin{proof}
We first construct the operators $\Psi(\xi_{\ul{i}})=:r_{\ul{i}}$ from the remarks preceding the proposition (for a suitable orthonormal basis). However, in order to control their $\|\cdot\|_R$-norms we must build these operators out of $\{\Psi(e_{\ul{i}})\}$ since this latter set is easily expressed as polynomials in the $X_i$. Indeed, for a multi-index $\underline{j}=\{j_1,\ldots,j_n\}$ let $\psi_{\underline{j}}\in\mathscr{P}$ be the non-commutative polynomial defined inductively by
	\begin{align}\label{recursive_psi}
		\psi_{\ul{j}}=X_{j_1}\psi_{j_2,\ldots,j_n} - \sum_{k\geq 2} q^{k-2}\<e_{j_1},e_{j_k}\>_U\psi_{j_2,\ldots,\hat{j_k},\ldots,j_n},
	\end{align}
where $\psi_{\emptyset}=1$. From a simple computation it is clear that $\psi_{\underline{j}}=\Psi(e_{j_1}\otimes\cdots\otimes e_{j_n})$.\par
Fix $n\geq 0$,  then, following \cite{D}, we let $B=B^*\in M_{N^n}(\C)$ be the matrix such that $B^2=\pi_{q,N,n}\left(P_q^{(n)-1}\right)$. In other words, given $h_1,\ldots,h_n\in \H$ if we define $g_{\underline{i}}=\sum_{|\underline{j}|=n} B_{\underline{i},\underline{j}} h_{\underline{j}}$ then
	\begin{equation*}
		\< g_{\underline{i}}, g_{\underline{j}}\>_{U,q} = \<h_{\underline{i}}, h_{\underline{j}}\>_{U,0} = \prod_{k=1}^n \<h_{i_k},h_{j_k}\>_U.
	\end{equation*}
Define $p_{\underline{i}}=\sum_{|\underline{j}|=n} B_{\underline{i},\underline{j}}\psi_{\underline{j}}$. Then the $p_{\underline{i}}$ satisfy
	\begin{equation*}
		\<p_{\underline{i}},p_{\underline{j}}\>_{\varphi} =\<p_{\ul{i}}\Omega, p_{\ul{j}}\Omega\>_{U,q}= \<e_{\underline{i}},e_{\underline{j}}\>_{U,0}.
	\end{equation*}
Let $\alpha\in M_N(\C)$ have entries $\alpha_{ij}=\<e_j,e_i\>_U$, and recall that by a previous computation this implies $\alpha=\frac{2}{1+A}$. We note that the eigenvalues of $\alpha$ are contained in the interval $\left[ \frac{2}{1+\|A\|},\frac{2}{1+\|A\|^{-1}}\right]$. Lemma 1.2 in \cite{H} implies that $\alpha^{\otimes n}$ is strictly positive, so let $D=D^*\in M_{N^n}(\C)$ be such that $D^2=\left(\alpha^{\otimes n}\right)^{-1}$.We claim that $\|D^2\| \leq  \left(\frac{1+\|A\|}{2}\right)^{n}$. Indeed, it suffices to show that the eigenvalues of $\alpha^{\otimes n}$ are bounded below by $\left(\frac{2}{1+\|A\|}\right)^n$. Suppose $\lambda$ is an eigenvalue with eigenvector $h_1\otimes \cdots \otimes h_n\in \H_\R^{\otimes n}$. Upon renormalizing, we may assume $\|h_i\|=1$ for each $i$. Thus
	\begin{align*}
		\lambda= \< h_1\otimes \cdots \otimes h_n, \alpha^{\otimes n}h_1\otimes \cdots \otimes h_n\>_{1,0}=\prod_i \< h_i,\alpha h_i\> \geq \left(\frac{2}{1+\|A\|}\right)^n,
	\end{align*}
and the claim follows. Setting $r_{\underline{i}}=\sum_{\underline{k}} D_{\underline{i},\underline{k}} p_{\underline{k}}$ we have
	\begin{align*}
		\<r_{\underline{i}}, r_{\underline{j}}\>_{\varphi} &= \sum_{\underline{k},\underline{l}} \overline{D_{\underline{i},\underline{k}}}D_{\underline{j},\underline{l}} \<p_{\underline{k}},p_{\underline{l}}\>_{\varphi}=\sum_{\underline{k},\underline{l}} D_{\underline{k},\underline{i}} D_{\underline{j},\underline{l}} \< e_{\underline{k}},e_{\underline{l}}\>_{U,0}\\
			& = \sum_{\underline{k},\underline{l}} D_{\underline{k},\underline{i}} D_{\underline{j},\underline{l}} \< \left(\frac{2}{1+A^{-1}}\right)^{\otimes n} e_{\underline{k}}, e_{\underline{l}}\>_{1,0} = \sum_{\underline{k},\underline{l}} D_{\underline{j},\underline{l}} \left[\left(\frac{2}{1+A^{-1}}\right)^{\otimes n}\right]_{\underline{k},\underline{l}} D_{\underline{k},\underline{i}}\\
			&=\sum_{\underline{k},\underline{l}} D_{\underline{j},\underline{l}} \left[\alpha^{\otimes n}\right]_{\underline{l},\underline{k}} D_{\underline{k},\underline{i}}= [D\alpha^{\otimes n}D]_{\underline{j},\underline{i}}=\delta_{\underline{i}=\underline{j}}.
	\end{align*}
Noting that $r_{\ul{i}}$ is a linear combination of the $\psi_{\ul{j}}$ with $|\ul{j}|=n$, we see that $r_{\ul{i}}\Omega\in\H^{\otimes n}$. Hence $\{r_{\ul{i}}\Omega\}_{|\ul{i}|=n}$ is an orthonormal basis for $\H^{\otimes n}$ and $P_n$ can be identified with $\sum_{|\ul{i}|=n} r_{\ul{i}} \otimes r_{\ul{i}}^*\in\mathscr{P}\otimes\mathscr{P}^{op}$.\par
Repeat this construction for each $n\geq 0$ so that for a multi-index $\ul{i}$ of arbitrary length we have a corresponding $r_{\ul{i}}$ and consequently a representation of $P_n$ in $\mathscr{P}\otimes\mathscr{P}^{op}$ for every $n$. Then by definition we have $\Xi_q=\sum_{n\geq 0} q^n \sum_{|\underline{i}|=n} r_{\underline{i}}\otimes r_{\underline{i}}^*$, provided this sum converges. Let $C_n(t)=\sup_{|\ul{i}|=n}\|\sigma_{it}(\psi_{\ul{i}})\|_R$, then we have
	\begin{align*}
		\left\|\sum_{|\ul{i}|=n}  \sigma_{it}(r_{\ul{i}})\otimes r_{\ul{i}}^*\right\|_{R\otimes_\pi R}&\leq \sum_{\ul{i},\ul{j},\ul{k},\ul{l},\ul{m}} \left|D_{\ul{i},\ul{j}}B_{\ul{j},\ul{l}} \overline{D_{\ul{i},\ul{k}}}\overline{B_{\ul{k},\ul{m}}}\right| \left\| \sigma_{it}(\psi_{\ul{l}})\right\|_R \| \psi_{\ul{m}}\|_R\leq \sum_{\ul{m},\ul{l}} \left| (BD^2B)_{\ul{m},\ul{l}}\right| C_n(t)C_n(0) \\
			&\leq N^{2n} \|BD^2B\| C_n(t)C_n(0) \leq N^{2n}\left(\frac{1+\|A||}{2}\right)^n \|B^2\| C_n(t)C_n(0) \\
			& \leq N^{2n}\left(\frac{1+\|A\|}{2}\right)^n \left( (1-|q|)\prod_{k=1}^\infty \frac{1+|q|^k}{1-|q|^k}\right)^n C_n(t)C_n(0),
	\end{align*}
where we have used the bound on $\|B^2\|$ from \cite{D}. From Equation (\ref{recursive_psi}) and (\ref{differentiating_sigma_with_q}), $C_n(t)\leq \|A^{-t}X\|_RC_{n-1}(t) +C_{n-2}(t)/(1-|q|)$. But $\|A^{-t}X\|_R\leq \|A^{-t}\| R=\|A^t\| R$ (see property 4 of $A$ in section \ref{free_Araki-Woods}), so that $C_n(t)\leq \|A^{t}\|^n\left( R+\frac{1}{1-|q|}\right)^n=\|A^{t}\|^n\left(\frac{3+c}{1-|q|}\right)^n$. Also, we use the bound
	\begin{align*}
		(1-|q|)\prod_{k=1}^\infty \frac{1+|q|^k}{1-|q|^k}\leq \frac{(1-|q|)^2}{1-2|q|},
	\end{align*}
from Lemma 13 in \cite{S09}. Thus
	\begin{align*}
		\left\|\sum_{|\underline{i}|=n} \sigma_{it}( r_{\underline{i}})\otimes r_{\underline{i}}^*\right\|_{R\otimes_\pi R} &\leq N^{2n}\left(\frac{1+\|A\|}{2}\right)^n \left(\frac{(1-|q|)^2}{1-2|q|}\right)^n \|A^{t}\|^n\left(\frac{3+c}{1-|q|}\right)^{2n}\\
					&= \left[ \|A^{t}\|N^2 \frac{1+\|A\|}{2}\frac{(3+c)^2}{1-2|q|}\right]^n.
	\end{align*}
Thus choosing $|q|$ small enough so that
	\begin{align*}
		|q|\|A^{t_0}\|N^2 \frac{1+\|A\|}{2}\frac{(3+c)^2}{1-2|q|}<1,
	\end{align*}
we can use $\|A^t\|\leq \|A^{t_0}\|$ for $|t|\leq|t_0|$ to obtain
	\begin{align*}
		\left\|(\sigma_{it}\otimes 1)(\Xi_q)-1\otimes 1\right\|_{R\otimes_\pi R} &\leq \sum_{n=1}^\infty \left[|q|\|A^{t}\|N^2\frac{1+\|A\|}{2}\frac{(3+c)^2}{1-2|q|}\right]^n\\
													&=\frac{\left\|A^{t}\right\|(3+c)^2(1+\|A\|)N^2|q|}{2- \left(4+ \left\|A^{t}\right\|(3+c)^2(1+\|A\|)N^2\right)|q|}.
	\end{align*}
The limit $\pi(q,N,A,t)\rightarrow 0$ as $|q|\rightarrow 0$ is clear from the definition of $\pi(q,N,A,t)$, and the ordering $\pi(q,N,A,s)\leq \pi(q,N,A,t)$ for $|s|\leq |t|$ simply follows from $\|A^s\| \leq \|A^t\|$. The final statements are then simple consequences of the formula $\frac{1}{x}=\sum_{n=0}^\infty (1-x)^n$.
\end{proof}

\begin{rem}
We note that $\pi(q,N,1,0)=\pi(q,N^2)$ in \cite{D}.
\end{rem}


\subsection{The conjugate variables $\xi_j$}

Recall that $\hat{\sigma}_z=\sigma_z\otimes \sigma_{\bar{z}}$. We will show that $\partial_j^{(q)*}\circ\hat{\sigma}_{-i}\left(\left[\Xi_q^{-1}\right]^*\right)$ defines the conjugate variables for $\partial_j$, but first we require some estimates relating to $\partial_j^{(q)*}$.\par
Fix $c>0$ and let $R=\left(1+\frac{c}{2}\right)\frac{2}{1-|q|}$. For now, we only assume $|q|$ is small enough that $\Xi_q\in \left(\mathscr{P}\otimes\mathscr{P}^{op}\right)^{(R)}$.

\begin{lem}\label{bounded_adjoint}
For each $j=1,\ldots,N$, the maps $(\varphi\otimes 1)\circ\partial_j^{(q)}$ and $(1\otimes\varphi)\circ\bar{\partial}_j^{(q)}$ are bounded operators from $\mathscr{P}^{(R)}$ to itself with norms bounded by $\frac{1-|q|}{c}\|\Xi_q\|_{R\otimes_\pi R}$. Consequently the maps $m\circ(1\otimes\varphi\otimes 1)\circ\left(1\otimes\partial_j^{(q)} +\bar{\partial}_j^{(q)}\otimes 1\right)$ are bounded from $\left(\mathscr{P}\otimes\mathscr{P}^{op}\right)^{(R)}$ to $\mathscr{P}^{(R)}$ with norm bounded by $\frac{2(1-|q|)}{c}\|\Xi_q\|_{R\otimes_\pi R}$.
\end{lem}
\begin{proof}
Recall that $\varphi$ is a state and $\|X_i\| \leq \frac{2}{1-|q|}$ and therefore $\varphi$ satisfies (\ref{phi_monomial}) with $C_0=\frac{2}{1-|q|}$. For $P\in \mathscr{P}^{(R)}$ write $P=\sum_{\ul{i}} a(\ul{i}) X_{\ul{i}}$ and denote $\|\Xi_q\|_{R\otimes_\pi R}=Q_0$. Then
	\begin{align*}
		\left\| (\varphi\otimes 1)\circ\partial_j^{(q)}(P)\right\|_R  &= \left\| \sum_{\ul{i}} a(\ul{i})(\varphi\otimes 1)\left( \sum_{k=1}^{|\ul{i}|} \alpha_{i_k j} X_{i_1}\cdots X_{i_{k-1}} \otimes X_{i_{k+1}}\cdots X_{i_{|\ul{i}|}}\#\Xi_q\right)\right\|_R \\
										&\leq \sum_{\ul{i}} |a(\ul{i})| \sum_{k=1}^{|\ul{i}|} \left(\frac{2}{1-|q|}\right)^{k-1} R^{n-k}Q_0 =\sum_{\ul{i}} |a(\ul{i})| R^{n-1} Q_0\sum_{k=1}^{\ul{i}} \left(\frac{1}{1+c/2}\right)^{k-1}\\
					&\leq \sum_{\ul{i}} a(\ul{i}) R^{n-1}Q_0 \frac{1}{1-\frac{1}{1+c/2}} =\|P\|_RQ_0 \frac{1}{R} \frac{1+c/2}{c/2} = \|P\|_R Q_0\frac{1-|q|}{c}.
	\end{align*}
The estimate for $(1\otimes \varphi)\circ\bar{\partial}_j^{(q)}$ is similar.\par
Define $\eta(P\otimes 1)$ to be left multiplication by $P$ on $\mathscr{P}^{(R)}$ and define $\eta(1\otimes P)$ to be right multiplication by $\frac{c}{1-|q|}Q_0^{-1}(\varphi\otimes 1)\circ\partial_j^{(q)}(P)$ on $\mathscr{P}^{(R)}$. Let $Q\in \mathscr{P}\otimes\mathscr{P}^{op}$, then by the above computations and the definition of $\|\cdot\|_{R\otimes_\pi R}$ we have
	\begin{align*}
		\left\|m\circ(1\otimes\varphi\otimes 1)\circ(1\otimes\partial_j^{(q)})(Q)\right\|_R = Q_0\frac{1-|q|}{c}\left\| \eta(Q)(1)\right\|_R\leq Q_0\frac{1-|q|}{c}\|Q\|_{R\otimes_\pi R}.
	\end{align*}
Similarly, $\| m\circ(1\otimes\varphi\otimes 1)\circ(\bar{\partial}_j^{(q)}\otimes 1)\|\leq Q_0\frac{1-|q|}{c}$ and so the final statement holds.
\end{proof}

Now let $|q|$ be sufficiently small that $\pi(q,N,A,-2)<1$. Then by Proposition \ref{Xi_invertible} and the statements preceding it, $\hat{\sigma}_{i}(\Xi_q^{-1}) = (\sigma_{2i}\otimes 1)(\Xi_q^{-1})$ and $(\sigma_{i}\otimes 1)(\Xi_q^{-1})$ exist as elements of $(\mathscr{P}\otimes\mathscr{P}^{op})^{(R)}$, as do their adjoints $\hat{\sigma}_{-i}\left(\left[\Xi_q^{-1}\right]^*\right)$ and $(\sigma_{-i}\otimes 1)\left(\left[\Xi_q^{-1}\right]^*\right)$. So by the preceding lemma the following defines an element of $\mathscr{P}^{(R)}$ for each $j=1,\ldots, N$: 
	\begin{align}\label{def_xi}
		\xi_j:= (\sigma_{-i}\otimes 1)\left(\left[\Xi_q^{-1}\right]^*\right)\# X_j - m\circ(1\otimes\varphi\otimes 1)\circ \left(1\otimes\partial_j^{(q)} + \bar{\partial}_j^{(q)}\otimes 1\right)\circ(\sigma_{-i}\otimes 1)\left(\left[\Xi_q^{-1}\right]^*\right),
	\end{align}
and
	\begin{align}\label{xi_R_norm}
		\|\xi_j\|_R \leq \left\| (\sigma_{i}\otimes 1)\left(\Xi_q^{-1}\right)\right\|_{R\otimes_\pi R} R + \frac{2(1-|q|)}{c}\left\| \Xi_q\right\|_{R\otimes_\pi R} \left\| (\sigma_{i}\otimes 1)\left(\Xi_q^{-1}\right)\right\|_{R\otimes_\pi R}.
	\end{align}
Now, using (\ref{adjoint_formula_2}) we see that
	\begin{align*}
		\partial_j^{(q)*}\circ\hat{\sigma}_{-i}\left(\left[\Xi_q^{-1}\right]^*\right) = (\sigma_{-i}\otimes 1)&\left(\left[\Xi_q^{-1}\right]^*\right)\# X_j \\
														&- m\circ(1\otimes\varphi\otimes \sigma_{-i})\circ\left(1\otimes \bar{\partial}_j^{(q)} + \bar{\partial}_j^{(q)}\otimes 1\right)\circ(\sigma_{-i}\otimes 1)\left(\left[\Xi_q^{-1}\right]^*\right),
	\end{align*} 
which is equivalent to $\xi_j$ defined above. Hence
	\begin{align*}
		\<\xi_j,P\>&=\<\hat{\sigma}_{-i}\left(\left[\Xi_q^{-1}\right]^*\right), \partial_j^{(q)}(P)\> = \varphi\otimes\varphi^{op}\left( \hat{\sigma}_i\left(\Xi_q^{-1}\right)\#\partial_{j}^{(q)}(P)\right) \\
			&= \varphi\otimes \varphi^{op}\left(\partial_j^{(q)}(P)\#\Xi_q^{-1}\right)=\varphi\otimes\varphi^{op}\left(\partial_j(P)\right)=\<1\otimes 1, \partial_j(P)\>.
	\end{align*}
Thus $\xi_j=\partial_j^*(1\otimes 1)$ is the conjugate variable of $X_1,\ldots,X_N$ with respect to the $\sigma$-difference quotient $\partial_j$. It also holds that $\xi_j=\xi_j^*$:
	\begin{align*}
		\<\xi_j^{*},P\>&=\varphi(\sigma_i(P)\xi_j)=\overline{\<\xi_j,\sigma_{-i}(P^*)\>}=\overline{\varphi\otimes\varphi^{op}(\partial_j\circ\sigma_{-i}(P^*))}\\
					&=\overline{\varphi\otimes\varphi^{op}\left(\bar{\partial}_{j}(P^*)\right)}=\varphi\otimes\varphi^{op}\left(\partial_j(P)\right)=\<\xi_j,P\>.
	\end{align*}
We remark that this could also be observed directly from the definition of $\xi_j$ in (\ref{def_xi}) using a combination of (\ref{Xi_q_sigma_invariant}) and the fact that $\Xi_q^\dagger=\Xi_q$.\par
We claim that there exists $V\in \mathscr{P}_{c.s.}^{(R,\sigma)}\subset M$ such that $\D_j V=\xi_j$. We first require a technical lemma which will lead to what is essentially the converse of Lemma \ref{change_of_variables}.(iii) in the case $Y=(\xi_1,\ldots, \xi_N)$.

\begin{lem}
Let $\xi_1,\ldots, \xi_N$ be as defined above. Then for $j,k\in\{1,\ldots, N\}$,
	\begin{align}\label{xi_has_positive_hessian}
		\partial_k(\xi_j)&=(1\otimes \sigma_{-i})\circ\bar{\partial}_j(\xi_k)^\diamond
	\end{align}
as elements of $L^2(M\bar{\otimes}M^{op},\varphi\otimes\varphi^{op})$. Furthermore,
	\begin{align}\label{xi_is_eigenvector}
		\sigma_{-i}(\xi_j)&=\sum_{k=1}^N [A]_{jk} \xi_k.
	\end{align}
\end{lem}
\begin{proof}
It suffices to check
	\begin{align*}
		\<\partial_i(\xi_j), a\otimes b\> = \<(1\otimes \sigma_{-i})\circ\bar{\partial}_j(\xi_i)^\diamond, a\otimes b\>
	\end{align*}
for elementary tensors $a\otimes b\in L^2(M\bar{\otimes}M^{op},\varphi\otimes\varphi^{op})$. So using (\ref{adjoint_formula}) we compute
	\begin{align*}
		\<\partial_k(\xi_j), a\otimes b\>=&\varphi(\xi_j a\xi_k\sigma_{-i}(b)) - \varphi(\xi_ja[(\varphi\otimes\sigma_{-i})\circ\bar{\partial}_k(b)]) - \varphi(\xi_j[(1\otimes\varphi)\circ\bar{\partial}_k(a)]\sigma_{-i}(b))\\
							      =&\<\partial_j^*\left( (\sigma_{-i}(b)\otimes a)^\dagger\right),\xi_k\>\\
							      	& + \varphi( \left\{ a^*[(\varphi\otimes \sigma_{-i})\circ\bar{\partial}_j\circ\sigma_{i}(b^*)]\right\}^*\xi_k)+\varphi(\{[(1\otimes \varphi)\circ\bar{\partial}_j(a^*)]b^*\}^*\xi_k)\\
							      & - \varphi( [(\varphi\otimes 1)\circ\partial_j(a)][(\varphi\otimes\sigma_{-i})\circ\bar{\partial}_k(b)]) - \varphi( a [(1\otimes \varphi)\circ\partial_j\circ(\varphi\otimes \sigma_{-i})\circ\bar{\partial}_k(b)])\\
							      &-\varphi([(\varphi\otimes 1)\circ\partial_j\circ(1\otimes \varphi)\circ\bar{\partial}_k(a)]\sigma_{-i}(b)) - \varphi([(1\otimes \varphi)\circ\bar{\partial}_k(a)][(1\otimes\varphi)\circ\partial_j\circ\sigma_{-i}(b)]).
	\end{align*}
We note that
	\begin{align*}
		\varphi(P^*\xi_k)=\overline{\<\xi_k,P\>}=\overline{\varphi\otimes \varphi^{op}(\partial_k (P))}=\varphi\otimes\varphi^{op}(\partial_k(P)^\dagger)=\varphi\otimes\varphi^{op}(\bar{\partial}_k(P^*)).
	\end{align*}
Applying this to the second and third terms in the above computation yields
	\begin{align*}
		\<\partial_k(\xi_j), a\otimes b\>=&\<\partial_j^*\left( (\sigma_{-i}(b)\otimes a)^\dagger\right),\xi_k\>\\
							&+ \varphi\otimes\varphi^{op}( \bar{\partial}_k\{[(\sigma_{i}\otimes\varphi)\circ\partial_j\circ\sigma_{-i}(b)]a\})+\varphi\otimes\varphi^{op}(\bar{\partial}_k\{b[(\varphi\otimes 1)\circ\partial_j(a)]\})\\
							& - \varphi( [(\varphi\otimes 1)\circ\partial_j(a)][(\varphi\otimes\sigma_{-i})\circ\bar{\partial}_k(b)]) - \varphi( a [(1\otimes \varphi)\circ\partial_j\circ(\varphi\otimes \sigma_{-i})\circ\bar{\partial}_k(b)])\\
							&-\varphi([(\varphi\otimes 1)\circ\partial_j\circ(1\otimes \varphi)\circ\bar{\partial}_k(a)]\sigma_{-i}(b)) - \varphi([(1\otimes \varphi)\circ\bar{\partial}_k(a)][(1\otimes\varphi)\circ\partial_j\circ\sigma_{-i}(b)])\\
							=& \<[\sigma_{-i}(b)\otimes a]^\dagger, \partial_j(\xi_k)\>\\
							&+\varphi([(\varphi\otimes 1)\circ\bar{\partial}_k\circ(\sigma_i\otimes \varphi)\circ\partial_j\circ\sigma_{-i}(b)]a)  - \varphi( a [(1\otimes \varphi)\circ\partial_j\circ(\varphi\otimes \sigma_{-i})\circ\bar{\partial}_k(b)])\\
							&+\varphi(b[(1\otimes \varphi)\circ\bar{\partial}_k\circ(\varphi\otimes 1)\circ\partial_j(a)]) -\varphi([(\varphi\otimes 1)\circ\partial_j\circ(1\otimes \varphi)\circ\bar{\partial}_k(a)]\sigma_{-i}(b)).
	\end{align*}
Now, applying (\ref{differentiating_sigma_with_q}) to the second line in the last equality above yields
	\begin{equation*}
		\varphi([(\varphi\otimes 1)\circ(\bar{\partial}_k\otimes \varphi)\circ\bar{\partial}_j(b)]a) -\varphi( [(1\otimes \varphi)\circ(\varphi\otimes \bar{\partial}_j)\circ\bar{\partial}_k(b)]a).
	\end{equation*}
This is zero if $(\varphi\otimes 1)\circ(\bar{\partial}_k\otimes \varphi)\circ\bar{\partial}_j=(1\otimes \varphi)\circ(\varphi\otimes \bar{\partial}_j)\circ\bar{\partial}_k$, but this is easily verified by computing on monomials. Finally, the final line in the last equality of the computation is equivalent to
	\begin{equation*}
		\varphi(b[(1\otimes \varphi)\circ\bar{\partial}_k\circ(\varphi\otimes 1)\circ\partial_j(a)]) -\varphi(b[(\varphi\otimes 1)\circ\partial_j\circ(1\otimes \varphi)\circ\bar{\partial}_k(a)]).
	\end{equation*}
This is zero if $(1\otimes\varphi)\circ(\varphi\otimes \bar{\partial}_k)\circ\partial_j = (\varphi\otimes 1)\circ(\partial_j \otimes \varphi)\circ\bar{\partial}_k$, but again this is easily checked on monomials. Thus
	\begin{align*}
		\<\partial_k(\xi_j),a\otimes b\>&=\<[\sigma_{-i}(b)\otimes a]^\dagger, \partial_j(\xi_k)\>=\varphi\otimes \varphi^{op}(a\otimes \sigma_{-i}(b)\#\partial_j(\xi_k))\\
				&=\varphi\otimes \varphi^{op}( (\sigma_i\otimes 1)\circ\partial_j(\xi_k)\# a\otimes b)=\<(1\otimes \sigma_{-i})\circ\bar{\partial}_j(\xi_k^*)^\diamond, a\otimes b\>,
	\end{align*}
showing (\ref{xi_has_positive_hessian}).\par
Towards verifying (\ref{xi_is_eigenvector}), we note that
	\begin{align*}
		\sum_{k=1}^N \left[A^{-1}\right]_{jk}\partial_k = \bar{\partial}_j.
	\end{align*}
Hence for $P\in\mathscr{P}$ we have
	\begin{align*}
		\< \sum_{k=1}^N [A]_{jk} \xi_k, P\> &= \sum_{k=1}^N [A]_{kj} \varphi\otimes\varphi^{op}\left(\partial_k(P)\right) = \varphi\otimes\varphi^{op}\left( \bar{\partial}_j(P) \right)\\
						& = \overline{\varphi\otimes\varphi^{op}\left( \partial_j(P^*)\right)}= \overline{\<\xi_j, P^*\>}=\varphi(P\xi_j)=\varphi(\sigma_i(\xi_j)P)=\<\sigma_{-i}(\xi_j),P\>,
	\end{align*}
which establishes (\ref{xi_is_eigenvector}).
\end{proof}

Define
	\begin{align*}
		V = \Sigma\left(\sum_{j,k=1}^N \left[\frac{1+A}{2}\right]_{jk} \xi_k X_j\right).
	\end{align*}
Note that (\ref{xi_R_norm}) implies $V\in\mathscr{P}^{(R)}$. We further claim that $\D_j V=\xi_j$ and $V\in\mathscr{P}_{c.s.}^{(R,\sigma)}$. The former is equivalent to
	\begin{align*}
		\D_j(\mathscr{N}V)=(1+\mathscr{N})\D_j V=(1+\mathscr{N})\xi_j= \xi_j+\sum_{k=1}^n \delta_k(\xi_j)\# X_k.
	\end{align*}
To show this, we first note that $\D_j=m\circ\diamond\circ(1\otimes\sigma_{-i})\circ\bar{\partial}_j$ and so by the derivation property of $\bar{\partial}_j$ we have
	\begin{align*}
		\D_j(PQ)= (1\otimes\sigma_{-i})\circ\bar{\partial}_j(P)^\diamond \# \sigma_{-i}(Q) + (1\otimes\sigma_{-i})\circ\bar{\partial}_j(Q)^\diamond\#P.
	\end{align*}
Thus using (\ref{xi_has_positive_hessian}) and $\sigma_{-i}(X_j)=[A X]_j$ from (\ref{modular_semicircular}) we have
	\begin{align*}
		\D_t(\mathscr{N}V)&=\sum_{j,k=1}^N \left[\frac{1+A}{2}\right]_{jk} \left( (1\otimes\sigma_{-i})\circ\bar{\partial}_t(\xi_k)\# \sigma_{-i}(X_j) + \alpha_{tj} \xi_k\right)\\
			&=\sum_{j,k,l=1}^N \left[\frac{1+A}{2}\right]_{jk} \partial_k(\xi_t)\# [A]_{jl}X_l + \sum_{j,k=1}^N \left[\frac{2}{1+A}\right]_{tj}\left[\frac{1+A}{2}\right]_{jk} \xi_k= \xi_t+\sum_{l=1}^N \delta_l(\xi_t)\# X_l,
	\end{align*}
as claimed.\par
Now, in order to show $V\in\mathscr{P}_{c.s.}^{(R,\sigma)}$ we will show that $V$ is invariant under $\sigma_{-i}$ and that $\mathscr{S}(V)=V$ Together, these imply that $V$ is invariant under $\rho$ and hence $V\in\mathscr{P}_{c.s.}^{(R,\sigma)}$ (that $V$ has finite $\|\cdot\|_{R,\sigma}$-norm follows from the fact that for $\rho$ invariant elements this norm agrees with the $\|\cdot\|_R$-norm). Using (\ref{xi_is_eigenvector}) and $\sigma_{-i}(X_j)=[A X]_j$ we see that
	\begin{align*}
		\sigma_{-i}(V)&=\Sigma\left( \sum_{j,k=1}^N \left[\frac{1+A}{2}\right]_{jk} \sum_{l=1}^N [A]_{kl} \xi_l \sum_{m=1}^N [A]_{jm} X_m\right)\\
				&=\Sigma\left( \sum_{j,k,l,m=1}^N \left[A^{-1}\right]_{mj}\left[\frac{1+A}{2}\right]_{jk}[A]_{kl} \xi_l X_m\right) = V.
	\end{align*}
Towards seeing $\mathscr{S}(V)=V$, we note that
	\begin{align*}
		\mathscr{S}(X_{i_1}\cdots X_{i_n}) &= \frac{1}{n}\sum_{l=0}^{n-1} \rho^l(X_{i_1}\cdots X_{i_n}) = \frac{1}{n}\sum_{l=1}^N \left[m\circ(1\otimes\sigma_{-i})\circ\delta_l(X_{i_1}\cdots X_{i_n})^\diamond\right]X_l\\
			&=\Sigma\left(\sum_{l=1}^N \left[m\circ(1\otimes\sigma_{-i})\circ\delta_l(X_{i_1}\cdots X_{i_n})^\diamond\right]X_l\right),
	\end{align*}
and by linearity this extends to general polynomials $P$. Hence
	\begin{align*}
		\sum_{l,m=1}^N \left[\frac{1+A}{2}\right]_{lm} &\left[m\circ(1\otimes\sigma_{-i})\circ\bar{\partial}_{m}(P)^\diamond\right]X_l \\
					&= \sum_{l=1}^N \left[m\circ(1\otimes\sigma_{-i})\circ\delta_l(P)^\diamond\right] X_l = \mathscr{N}\mathscr{S}(P)=\mathscr{S}(\mathscr{N}P).
	\end{align*}
Consequently (\ref{xi_has_positive_hessian}) implies
	\begin{align*}
		\mathscr{S}(\mathscr{N}^2V) &= \sum_{l,m=1}^N \left[\frac{1+A}{2}\right]_{lm} \left[(1\otimes\sigma_{-i})\circ\bar{\partial}_m\left(\mathscr{N}V\right)^\diamond\right] X_l\\
						&=\sum_{j,k,l,m=1}^N \left[\frac{1+A}{2}\right]_{lm}\left[\frac{1+A}{2}\right]_{jk} \left[ (1\otimes\sigma_{-i})\circ\bar{\partial}_m(\xi_k)^\diamond\# \sigma_{-i}(X_j) + \alpha_{mj}\xi_k\right]X_l\\
						&=\sum_{j,k,l,m,a=1}^N \left[\frac{1+A}{2}\right]_{lm}\left[\frac{1+A}{2}\right]_{jk} [A]_{ja}\left[ \partial_k(\xi_m)\# X_a\right]X_l + \sum_{k,l=1}^N \left[\frac{1+A}{2}\right]_{lk} \xi_k x_l\\
						&=\sum_{l,m=1}^N \left[\frac{1+A}{2}\right]_{lm} \left[\mathscr{N}-1\right]\left(\xi_m X_l\right) + \mathscr{N}V=\mathscr{N}^2V.
	\end{align*}
Thus $\mathscr{S}(V)=V$, and $V\in\mathscr{P}_{c.s.}^{(R,\sigma)}$ as claimed.\par
Note
	\begin{align*}
		V_0=\frac{1}{2}\sum_{j,k=1}^N \left[\frac{1+A}{2}\right]_{jk} X_kX_j = \Sigma\left(\sum_{j,k=1}^N \left[\frac{1+A}{2}\right]_{jk} X_kX_j\right),
	\end{align*}
and define $W:=V-V_0$. Then $W\in\mathscr{P}_{c..s}^{(R,\sigma)}$ and
	\begin{align*}
		\|W\|_{R,\sigma}=\|W\|_{R} \leq \sum_{j,k=1}^N \left|\left[\frac{1+A}{2}\right]_{jk}\right| \|\xi_k - X_k\|_{R} R.
	\end{align*}
We claim that $\|\xi_k - X_k\|_{R}\rightarrow 0 $ as $|q|\rightarrow 0$, and consequently $\|W\|_{R,\sigma}\rightarrow 0$. Indeed, we can write
	\begin{align*}
		X_k=\left(\left[1\otimes 1\right]^*\right)\# X_k - m\circ(1\otimes\varphi\otimes 1)\circ\left(1\otimes\partial_k^{(q)}+\bar{\partial}_k^{(q)}\otimes 1\right)\left( \left[1\otimes 1\right]^*\right),
	\end{align*}
and so using (\ref{def_xi}) and Lemma \ref{bounded_adjoint} we have
	\begin{align*}
		\|\xi_k - X_k\|_{R} \leq \left\| (\sigma_i\otimes 1)(\Xi_q^{-1}) - 1\otimes 1 \right\|_{R\otimes_\pi R}R + \frac{2(1-|q|)}{c} \|\Xi_q\|_{R\otimes_\pi R}\left\| (\sigma_i\otimes 1)(\Xi_q^{-1}) - 1\otimes 1\right\|_{R\otimes_\pi R}.
	\end{align*}
From the final remark in Proposition \ref{Xi_invertible}, we see that this tends to zero as $|q|\rightarrow 0$. Thus we are in a position to apply our transport results from Section \ref{construction}. Using Corollary \ref{iso_cor} we obtain the following result.

\begin{thm}
There exists $\epsilon>0$ such that $|q|<\epsilon$ implies $\Gamma_q(\H_\R,U_t)\cong \Gamma_0(\H_\R,U_t)$ and $\Gamma_q(\H_\R,U_t)''\cong\Gamma_0(\H_\R,U_t)$.
\end{thm}

Using the classification of $\Gamma_0(\H_\R,U_t)''$ in Theorem 6.1 of \cite{S97} we obtain the following classification result.

\begin{cor}
For $\H_\R$ finite dimensional, let $G$ be the multiplicative subgroup of $\R^\times_+$ generated by the spectrum of $A$. Then there exists $\epsilon>0$ such that for $|q|<\epsilon$
	\begin{align*}
		\Gamma_q(\H_\R,U_t)''\text{ is a factor of type }\left\{\begin{array}{cl}	\mathrm{III}_1	&	\text{if }G=\R_+^\times\\
													\mathrm{III}_\lambda & 	\text{if }G=\lambda^\Z,\ 0<\lambda<1\\
													\mathrm{II}_1		&	\text{if }G=\{1\}.\end{array}\right.
	\end{align*}
Moreover, $\Gamma_q(\H_\R,U_t)''$ is full.
\end{cor}

\end{document}